\documentclass[a4paper,reqno,12pt]{amsart}
\usepackage[utf8]{inputenc}

\usepackage{mathtools,thmtools}
\usepackage{tikz-cd}%

\usepackage{graphicx,color}
\usepackage{bm}
\usepackage{amsmath,amsthm,amstext,amsfonts,amsbsy}%
\usepackage{amssymb}
\usepackage{latexsym}%
\usepackage{todonotes}
\usepackage[margin=3cm]{geometry}
\usepackage{layout}
\usepackage[T1]{fontenc}
\usepackage{physics}
\usepackage{braket}

\usepackage[pagebackref]{hyperref}
\usepackage[capitalize]{cleveref}
\hypersetup{
     colorlinks = true,
     citecolor  = blue,
     linkcolor  = blue, 
     urlcolor   = blue, 
}
\usepackage{bookmark}

\theoremstyle{definition}
\newtheorem{theorem}{Theorem}[section]
\newtheorem*{theorem*}{Theorem}
\newtheorem{definition}[theorem]{Definition}
\newtheorem{definition*}{Definition}
\newtheorem{example}[theorem]{Example}
\newtheorem*{example*}{Example}
\newtheorem{proposition}[theorem]{Proposition}
\newtheorem*{proposition*}{Proposition}

\newtheorem*{note*}{Note}

\newtheorem*{notice*}{Notice}
\newtheorem{lemma}[theorem]{Lemma}
\newtheorem*{lemma*}{Lemma}

\newtheorem*{fact*}{Fact}
\newtheorem{question}[theorem]{Question}
\newtheorem*{question*}{Question}
\newtheorem{conjecture}[theorem]{Conjecture}
\newtheorem*{conjecture*}{Conjecture}
\newtheorem{notation}[theorem]{Notation}
\newtheorem*{notation*}{Notation}
\newtheorem{corollary}[theorem]{Corollary}
\newtheorem*{corollary*}{Corollary}
\newtheorem{remark}[theorem]{Remark}
\newtheorem*{remark*}{Remark}
\newtheorem{condition}[theorem]{Condition}
\newtheorem*{condition*}{Condition}

\newtheorem*{convention*}{Convention}

\newtheorem*{observation*}{Observation}

\newcommand{\deq}{\coloneqq}

\newcommand{\emp}{\varnothing}
\newcommand{\C}{\mathbb{C}}
\newcommand{\Z}{\mathbb{Z}}
\newcommand{\vep}{\varepsilon}

\newcommand{\mb}[1]{\mathbb{#1}}
\newcommand{\mcal}[1]{\mathcal{#1}}%

\newcommand{\opn}[1]{\operatorname{#1}}
\newcommand{\catn}[1]{\mathbf{#1}}

\newcommand{\abk}[1]{\langle {#1} \rangle}
\newcommand{\xto}[1]{\xrightarrow{#1}}

\newcommand{\hookto}{\hookrightarrow}

\newcommand{\myfootnote}[1]{\hspace{-5pt}\footnote{#1}}

\newcommand{\simto}{ 
\mathrel{\raisebox{0.13em}{${\sim}$}}
\kern -0.75em \mathrel{\raisebox{-0.11em}{${\scriptstyle \to}$}}  
}

\renewcommand*{\backrefalt}[4]{%
\ifcase #1 %
\or        [Cited on p.#2.]%
\else      [Cited on pp.#2.]%
\fi}

\DeclareMathOperator{\Pic}{Pic}
\DeclareMathOperator{\CDiv}{Div^{\infty}}
\DeclareMathOperator{\fuk}{fuk}
\DeclareMathOperator{\coh}{coh}
\DeclareMathOperator{\reg}{gen}
\DeclareMathOperator{\tform}{\Omega}
\DeclareMathOperator{\mform}{\mathcal{F}}

\crefname{equation}{}{}
\crefname{conjecture}{Conjecture}{Conjectures}

\usepackage{mathrsfs}
\usepackage{upgreek}
\numberwithin{equation}{section}
\title[Graded modules associated 
with permissible $C^{\infty}$-divisors]{Graded modules 
associated with permissible $C^{\infty}$-divisors
on tropical manifolds}
\author[Y. Tsutsui]{Yuki Tsutsui}
\address{Graduate School of Mathematical Sciences,
The University of Tokyo, 3-8-1 Komaba, Meguro-Ku,
Tokyo, 153-8914, Japan}
\email{tyuki@ms.u-tokyo.ac.jp}

\begin{document}

\begin{abstract}
We use ideas 
from the Strominger--Yau--Zaslow conjecture 
and microlocal sheaf theory 
to define graded modules
associated with permissible $C^{\infty}$-divisors
on compact tropical manifolds.
A $C^{\infty}$-divisor is
a generalization of
a Lagrangian section
on an integral affine manifold.
The group of $C^{\infty}$-divisors
on a tropical manifold surjects
onto the Picard group.
We also prove a Riemann--Roch formula 
for compact tropical curves 
and integral affine manifolds 
admitting Hessian forms. 
Our approach differs from the tropical Riemann--Roch theorem 
established by Gathmann and Kerber.
\end{abstract}

\maketitle

\section{Introduction}

\subsection{Background}

Tropical geometry is a combinatorial 
(or convex-geometric)
analog of algebraic (or analytic) geometry.
A \emph{rational polyhedral space}
is a pair
$(S,\mathcal{O}^{\times}_S)$
of a topological space $S$
and 
a sheaf $\mathcal{O}^{\times}_S$ on $S$
which is locally isomorphic to
the pair
of an open subset
of a polyhedral subset 
of a tropical affine space 
$
\overline{\mathbb{R}}^{k}=
(\mathbb{R}\cup\{+\infty\})^{k}
$
and
the sheaf
of affine linear functions with 
integer slopes
(see e.g.~\cite[Definition 2.2]{gross2019sheaftheoretic}).
The sheaf $\mathcal{O}_S^{\times}$ is a tropical analog
of the sheaf $\mathcal{O}_X^{\times}$ of 
units of the structure sheaf $\mathcal{O}_X$ of 
an algebraic variety $X$.
By a \emph{line bundle}
on $S$,
we mean an element
of
$
\Pic S
\coloneqq
H^{1}(S;\mathcal{O}_S^{\times})
$
(\cite[Section 4.3]{mikhalkinTropicalCurvesTheir2008a},
see also \cite[Definition 3.12 and Proposition 3.13]{gross2019sheaftheoretic}).
A \emph{tropical manifold} is 
a rational polyhedral space which is locally 
isomorphic to the direct product of a 
Bergman fan $\Sigma_M$ for some loopless
matroid $M$ and a tropical affine space 
$\overline{\mathbb{R}}^{k}$
(see \cite[Definition 6.1]{gross2019sheaftheoretic} 
which originates from 
\cite[Definition 1.14]{mikhalkinTropicalEigenwaveIntermediate2014a}).
An example of a tropical manifold is 
a \emph{tropical curve}
\cite[Definition 3.1]{mikhalkinTropicalCurvesTheir2008a}
(see also \cite[Definition 1.1]{gathmannRiemannRochTheoremTropical2008a}
for a combinatorial definition of a tropical curve).
In this paper, the term `tropical curve' refers to a 
$1$-dimensional tropical manifold.
Every (finite) metric graph 
(with no $1$-valent vertex) has 
a natural structure of a tropical curve 
\cite[Proposition 3.6]{mikhalkinTropicalCurvesTheir2008a}.

Finding tropical analogs of various theorems 
in algebraic geometry is one of the 
important themes in tropical geometry.
The tropical Riemann--Roch theorem for a 
divisor $D$
\cite[Definition 1.3]{gathmannRiemannRochTheoremTropical2008a}
on a compact (and connected) tropical curve $C$,
proved 
by Gathmann--Kerber \cite{gathmannRiemannRochTheoremTropical2008a}
as an extension of Riemann--Roch theorem for graphs
\cite{MR2355607} (see also 
\cite{mikhalkinTropicalCurvesTheir2008a}),
is the equality
\begin{align} \label{Gathmann-Kerber-TRR}
r(D)-r(K_C-D)=\opn{deg}(D)+\chi_{\mathrm{top}}(C),
\end{align}
where $r(D)$ is the rank of the linear system of 
$D$ 
\cite[Definition 1.12]{gathmannRiemannRochTheoremTropical2008a}, 
$K_C$ is the canonical divisor of $C$
\cite[Definition 1.3]{gathmannRiemannRochTheoremTropical2008a},
and $\chi_{\mathrm{top}}(C)$ is the topological 
Euler characteristic of $C$.
The tropical Riemann--Roch theorem for divisors
on a compact tropical curve were generalized to
divisors on a weighted metric graph \cite[Theorem 5.4]{MR3046301}, 
or multidivisors on a metric graph \cite[Theorem 4.3]{gross2022principal}.

It is a natural problem 
to find a tropical analog of 
the Hirzebruch--Riemann--Roch theorem
\begin{align} \label{equation-HRR}
\chi(H^{\bullet}(X;\mathcal{L}))
=\int_X \opn{ch}(\mathcal{L})\opn{td}(X)
\end{align}
for line bundles on 
complex projective manifolds
\cite{MR0202713}.
However,
the left-hand side of \cref{equation-HRR}
does not have a straightforward tropical analog,
since line bundles on tropical manifolds correspond
to isomorphism classes of sheaves
not of Abelian groups
but of commutative monoids,
so that 
the basic theory of homological algebra
does not apply.
See \cref{section-tropical-riemann-roch}
for more on this difficulty.
\cref{section-tropical-riemann-roch}
also contains earlier works
on topological Euler characteristics and 
Todd classes of tropical manifolds
(see also \cref{remark-todd-class}).

In this paper,
we pursue a different tropical analog
of the Euler characteristic of line bundles
motivated by ideas from
toric geometry,
geometric quantization,
Morse theory,
microlocal sheaf theory,
homological mirror symmetry,
and
the Strominger--Yau--Zaslow conjecture.

\subsection{SYZ conjecture and main results}

The \emph{Strominger--Yau--Zaslow conjecture} 
\cite{stromingerMirrorSymmetryTduality1996}
(or the SYZ conjecture for short)
suggests that 
a mirror pair $(\mathcal{X}, \check{X})$
of Calabi--Yau manifolds
comes with a dual pair
$
\left(
f_B\colon \mathcal{X} \to B,
\check{f}_{B}\colon \check{X} \to B
\right)
$
of special Lagrangian torus fibrations
in such a way that
homological mirror symmetry
\cite{MR1403918}
between
the (split-closed derived) Fukaya category 
of $\check{X}$ 
and the derived category
of coherent sheaves 
on $\mathcal{X}$
is induced by a Fourier-like transformation
\begin{align} \label{equation-SYZ}
  \Phi_{\mathrm{SYZ}} \colon 
\fuk \check{X} \simto \coh \mathcal{X}
\end{align}
called the \emph{SYZ transformation}.

If a Lagrangian torus fibration
$
\check{f}_{B} \colon \check{X} \to B
$
does not have a singular fiber,
then the Liouville--Arnold theorem
equips the base space $B$
with
an \emph{integral affine structure},
i.e.,
an atlas whose transition maps belong to
$\opn{GL}(n,\mathbb{Z})\ltimes \mathbb{R}^{n}$.
Conversely,
given an integral affine manifold $B$,
one can construct a Lagrangian torus fibration
$
\check{f}_{B} \colon \check{X}(B) \to B
$
inducing the integral affine structure on $B$.
One can also construct
an analytic torus fibration
$f_{B}\colon \mathcal{X}(B)\to B$
in the sense of \cite[Definition 4]{MR2181810}
in such a way that
if $B$ is compact and
$\pi_2(B) = 0$,
then the SYZ conjecture holds
\cite{MR1882331,MR3728639,MR3656481,MR4301560},
as we recall briefly now
(and in a little more detail later
in \cref{section-integral-affine-manifold}):
\footnote{The assumption $\pi_2(B) = 0$ is redundant
if Markus conjecture holds
(see \cref{remark-markus-conjecture}).}

Let $B$ be a compact integral affine manifold
(and hence a tropical manifold in particular).
Let
further
$\mathcal{C}^{\infty}_{B}$
be the sheaf
of $C^{\infty}$-functions.
An element of 
$
\CDiv(B) \coloneqq H^{0}(B;\mathcal{C}^{\infty}_{B}/\mathcal{O}^{\times}_{B})
$
is called a \emph{$C^\infty$-divisor}
(\cref{definition-prepermissible-smooth-cartier-divisor,definition-C-infinity-divisor}).
The short exact sequence
\begin{align}
0
\to
\mathcal{O}^{\times}_{B}
\to
\mathcal{C}^{\infty}_{B}
\to
\mathcal{C}^{\infty}_{B} / \mathcal{O}^{\times}_{B}
\to
0
\end{align}
of sheaves induces a surjection 
$
\CDiv(B)
\to
\Pic B
$.
Any
$
s
\in
\CDiv(B)
$
gives a Lagrangian section
$L_s$
of $\check{X}(B) \to B$,
which can be turned into an object
$\mathscr{L}_{s}$ of 
$\fuk \check{X}(B)$
by choosing an additional data
(a grading,
a spin structure,
and a flat $U(1)$-bundle),
which we ignore here
for the sake of brevity of the exposition.
A pair of $C^\infty$-divisors
which differ by an element in the image of
$H^0(B; \mathcal{C}^{\infty}_{B})$
give rise to
Hamiltonian isotopic
Lagrangian sections,
so that
the isomorphism class of the object $\mathscr{L}_s$
depends only on the image $[s]$ of $s$
in $\Pic B$
(and the additional data).
The SYZ transform 
$\mathcal{L}_s$ of $\mathscr{L}_s$ is a line bundle
on the rigid analytic space $\mathcal{X}(B)$,
and
for any
$
s, s' \in \CDiv(B)
$,
one has an isomorphism 
\begin{align}
\label{equation-hf-sheaf-cohomology}
\opn{HF}^{\bullet}(\mathscr{L}_s,\mathscr{L}_{s'})
\simeq \opn{Ext}^{\bullet}_{\mathcal{O}_{\mathcal{X}(B)}}(\mathcal{L}_s,\mathcal{L}_{s'}),
\end{align}
where $\opn{HF}^{\bullet}(\mathscr{L}_s,\mathscr{L}_{s'})$
is the Floer cohomology
(over the Novikov field $\Lambda_{\mathrm{nov}}^{\mathbb{C}}$
with $\mathbb{C}$-coefficients).
In particular,
if $L_s$ intersects the zero-section $L_{s_0}$
transversely,
then the Floer complex
$\opn{CF}^{\bullet}(\mathscr{L}_{s_0},\mathscr{L}_{s})$
underlying $\opn{HF}^{\bullet}(\mathscr{L}_{s_0},\mathscr{L}_{s})$
is the graded vector space
spanned by $L_s \cap L_{s_0}$,
whose Euler characteristic coincides with
$\chi(H^{\bullet}(\mathcal{X}(B);\mathcal{L}_s))$.

An extension of the SYZ conjecture
to a curve of higher genus
is studied in
\cite{MR3570582,auroux2022lagrangian}, 
which is expected to generalize
to higher dimensions
(see \cite{MR3600059} and 
\cite[{\textsection 7}]{auroux2022lagrangian}).
Ruddat also informed us that
Mak--Matessi--Ruddat--Zharkov proved
the construction problem for
a topological SYZ torus fibration 
$\check{f}_B\colon \check{X}\to B$
from integral affine manifold with
singularities (see \cite{MR4294796}
for more detail about it).

The \emph{tropical cohomology}
$H^{\bullet,\bullet}(S;\mathbb{Z})$
introduced in
\cite{mikhalkinTropicalEigenwaveIntermediate2014a,
itenbergTropicalHomology2019b}
is defined for an arbitrary tropical manifold $S$,
and carries a piece of information
on the limiting mixed Hodge structure
if the tropical manifold comes from a degenerating family
of complex manifolds.
We follow the notations of
\cite[Definition 2.6]{MR3894860}
for tropical cohomology.
A variation of tropical homology and cohomology 
for integral affine manifold having
at worst symple singularities 
is also developed in \cite{MR4347312}
(see also \cite[\textsection 2.1]{MR4179831}
for the relations with SYZ fibrations).
If $S$ is compact and connected of dimension $n$,
then one has the trace map
$
\int_S\colon 
H^{n,n}(S;\mathbb{Z})
\to \mathbb{Z}
$
as we recall in
\cref{equation-trace-integration}.  
In this paper,
we introduce
a graded $\Z$-module 
$\opn{LMD}^{\bullet}(S;s)$
(\cref{definition-local-morse-data-divisor})
associated with a permissible
$C^\infty$-divisor $s$ 
(\cref{condition-good-divisor})
on a compact
tropical manifold $S$.
If $S$ is an integral affine manifold 
and the intersection of 
$L_{s_0}$ and $L_{s}$ is transversal,
then
$
\opn{LMD}^{\bullet}(S;s)
\otimes_{\mathbb{Z}}
\Lambda_{\mathrm{nov}}^{\mathbb{C}}
$
coincides with the Floer complex
$\opn{CF}^{\bullet}(\mathscr{L}_{s_0},\mathscr{L}_{s})$.

\begin{conjecture}
\label{conjecture-enough-prepermissible}
For any line bundle $\mathcal{L}$ on a compact 
tropical manifold $S$, there exists 
a permissible $C^{\infty}$-divisor $s$ such that
$[s]=\mathcal{L}$. 
\end{conjecture}

\cref{conjecture-tropical-MRR-preface} below
is the main conjecture
of this paper:

\begin{conjecture}
\label{conjecture-tropical-MRR-preface}
For any compact tropical manifold $S$, 
there exists an element $\opn{td}(S)$
of the total tropical cohomology group
$H^{\bullet,\bullet}(S;\mathbb{Z})$ such that, 
for any permissible $C^{\infty}$-divisors $s$, 
one has
\begin{align}
\chi(\opn{LMD}^{\bullet}(S;s))=
\int_S \opn{ch}([s])\opn{td}(S).
\end{align}
Moreover,
one has
$\chi(\opn{LMD}^{\bullet}(S;s))=\chi_{\opn{top}}(S)$
when $[s] = 0 \in \Pic S$.
\end{conjecture}

\cref{conjecture-enough-prepermissible,conjecture-tropical-MRR-preface}
imply that
the Euler characteristic of a line bundle is well-defined.
\cref{conjecture-enough-prepermissible} is obvious when 
$S$ is a compact tropical curve 
or a compact integral affine manifold.
\cref{theorem-MRR-curve-preface,theorem-MRR-hesse-preface}
below prove \cref{conjecture-tropical-MRR-preface}
for a compact tropical curve
and
a compact integral affine manifold
admitting a Hessian form respectively:

\begin{theorem}
\label{theorem-MRR-curve-preface}
If $C$ is a compact tropical curve,
then for any permissible $C^{\infty}$-divisors $s$,
one has
\begin{align}
\chi(\opn{LMD}^{\bullet}(C;s))=
\int_C c_1([s])+\frac{1}{2}c_1(-K_C)
=\opn{deg}([s])+\chi_{\opn{top}}(C). 
\end{align}
\end{theorem}

\begin{theorem}
\label{theorem-MRR-hesse-preface}
If
$B$ is an $n$-dimensional compact integral affine manifold $B$
admitting a Hessian form,
then for any permissible $C^{\infty}$-divisors $s$,
one has
\begin{align}
\chi(\opn{LMD}^{\bullet}(B;s))=
\int_B \frac{c_1([s])^{n}}{n!}.
\end{align}
\end{theorem}

\cref{tropical elliptic curve} below gives
an example of both \cref{theorem-MRR-curve-preface} and \cref{theorem-MRR-hesse-preface}:

\begin{example} \label{tropical elliptic curve}
Let
$
B = \mathbb{R}/\mathbb{Z}
$
be a tropical elliptic curve.
For any integer $n$,
the quadratic function
$f_n \colon \mathbb{R} \to \mathbb{R}$,
$x \mapsto \frac{n}{2}x^{2}$
on the universal cover of $B$
defines a $C^{\infty}$-divisor 
$s_n \in \CDiv(B)$.
The associated Lagrangian section $L_{s_n}$
is the graph of
$
B \to 
\check{X}(B) \cong B \times \mathbb{R}/\mathbb{Z}
$;
$
x \mapsto (x, nx)
$,
and for any non-zero integer $n$, one has
\begin{align}
\chi(\opn{LMD}^{\bullet}(B;s_n))=
\sharp(L_{s_0}\cap L_{s_n})=n
=\opn{deg}([s_n])+\chi_{\opn{top}}(B). 
\end{align}
\end{example}

Although \cref{theorem-MRR-curve-preface,theorem-MRR-hesse-preface}
are closely related to
\cite{auroux2022lagrangian}
and
\cite{MR4301560}
respectively,
our proofs are purely `tropical'
and does not rely on the SYZ transformation
(see
\cref{section-syz-trivalent-graph,section-floer-lmd} 
for more details).

\subsection{Outline of this paper}

In \cref{section-local-morse-data},
we define local Morse data $\opn{LMD}^{\bullet}(S;s)$
for $C^{\infty}$-divisors $s$ 
satisfying the permissibility condition
on tropical manifolds.
\cref{section-tropical-curve} contains the proof of
\cref{theorem-MRR-curve-preface}.
In \cref{section-integral-affine-manifold}, we prove 
\cref{theorem-MRR-hesse-preface}
after preparation
from the theory of tropical cohomology.
In \cref{section-more-examples}, we discuss
more examples
of Riemann--Roch type formula.

\subsection*{Acknowledgments}
I would like to express 
my deepest gratitude to my advisor, Kazushi Ueda,
for his invaluable advice and constant encouragement.
Without his extensive support,
I would not have been able to complete this paper.
I want to offer my thanks 
to Yuto Yamamoto for explaining and answering my questions 
about his works 
in \cite{yamamotoTropicalContractionsIntegral2021}
and the Gross--Siebert program.
I would like to thank Andreas Gross for answering my questions 
on \cite{gross2019sheaftheoretic}.
I am grateful to Felipe Rinc\'on 
for explaining about the unpublished recent works
on Noether formula for general tropical surfaces.
I am so thankful to Masanori Kobayashi for 
pointing out some errors for the prior version 
of the prepermissibility condition for
principal $C^{\infty}$-divisors on metric graphs.
I would like to appreciate Yoshinori Gongyo for 
giving a question about relationships 
between a type of Pick's formula for
tropical Kummer surfaces and Atiyah--Singer
index formula. This question allows me 
to discover \cite{MR2676658}.
My thanks also go to Jin Miyazawa
for explaining us a heart of localization of 
the index of Dirac operator and answering some 
questions on \cite{MR2676658}.
I am grateful to Kentaro Yamaguchi 
for explaining about \cite{MR4234675}
and results in his master thesis \cite{yamaguchimaster}. 
I would also like to 
thank Yasuhito Nakajima for letting me know 
Yamaguchi's work.
I extend my gratitude to Helge Ruddat
for pointing out several references.
This work was supported by JSPS KAKENHI 
Grant Number JP21J14529.

\subsection*{Notations and conventions} \label{notation-general}

We use the following notations in this paper:

\begin{itemize}

\setlength{\parskip}{0pt}
\setlength{\leftskip}{-20pt}
\item $\underline{\mathbb{R}}\deq 
\mathbb{R}\cup \{-\infty\}$, $\overline{\mathbb{R}}
\deq \mathbb{R}\cup \{+\infty\}$.
\item $\mathbb{T}\deq(\mathbb{R}\cup\{-\infty\},
\opn{max},+)$: the (max-plus) tropical semifield.
\item For every continuous function $f\colon X\to {\mathbb{R}}$,
\begin{align*}
\{f<f(v)\}\deq \set{x\in X\mid f(x)< f(v)},\quad
\{f\geq f(v)\}\deq \set{x\in X\mid f(x)\geq f(v)}.
\end{align*}
\item  For $n\in \mathbb{Z}_{\geq 1}$,
$\|\cdot \|_{\mathbb{R}^{n}}\colon 
\mathbb{R}^{n}\to \mathbb{R};(x_1,\ldots,x_n)\mapsto 
\sqrt{\sum_{i=1}^{n}x_i^{2}}$.
\item For simplicity, $\mathbb{R}^{0}\deq \{0\}$ 
and 
$\|\cdot\|_{\mathbb{R}^{0}}\colon \mathbb{R}^{0}\to \mathbb{R}
; x\mapsto 0$.
\item For $n\in \mathbb{Z}_{\geq 0}$, $\varepsilon,\varepsilon'\in \mathbb{R}_{\geq 0}$
and $x \in \mathbb{R}^{n}$, 
\begin{align*}
B^{n}_{\varepsilon}(x)\deq 
\{v\in \mathbb{R}^{n}\mid \|x-v\|_{\mathbb{R}^{n}}
\leq \varepsilon\},
& \quad  
S^{n-1}_{\varepsilon}(x)\deq \partial B^{n}_{\varepsilon}(x),
\quad \\
\bar{B}^{n}_{\varepsilon}(x)\deq 
B^{n}_{\varepsilon}(x)\cup \partial B^{n}_{\varepsilon}(x),  \\
B_{\varepsilon,\varepsilon'}^{n}(x)\deq 
B^{n}_{\varepsilon}(x)\setminus 
\bar{B}^{n}_{\varepsilon '}(x), 
& \quad \bar{B}_{\varepsilon,\varepsilon'}^{n}(x)\deq 
B_{\varepsilon,\varepsilon'}^{n}(x)\cup 
\partial B_{\varepsilon,\varepsilon'}^{n}(x).
\end{align*}
\item $\{\opn{pt}\}$: the topological space of 
a fixed singleton set.
\item $a_X\colon X\to \{\opn{pt}\}$:
the continuous map from a topological space $X$ to
$\{\opn{pt}\}$.
\item $A_X$: the constant sheaf on a topological space $X$ 
with fiber $A$.
\item $\catn{Mod}(\mathcal{R})$: the category of 
sheaves of $\mathcal{R}$-modules
(e.g. \cite[Definition 2.2.6]{MR1299726}).
\item $D^{b}(A_X)$: the derived category
of bounded complexes of $\catn{Mod}(A_X)$.
\end{itemize}

We also impose the following assumption for simplicity in this paper:

\begin{itemize}
\setlength{\leftskip}{-20pt}
\item Any topological space is Hausdorff and
locally compact unless otherwise specified.
\item Any ring is a ring with unity.
\item $A$ is a PID.
\item We identify any ring $A$ with a constant sheaf on
$\{\opn{pt}\}$ with fiber $A$.
\item We identify any $A$-module with a sheaf of 
$A_{\{\opn{pt}\}}$-module.
\item Any $C^{\alpha}$-manifold is paracompact for 
$\alpha=0,1,\ldots,\infty,\omega$.
\item We mainly use max-plus algebra.
\end{itemize}

\section{The cohomological local Morse data for
\texorpdfstring{$C^{\infty}$}{C-infty}-divisors}
\label{section-local-morse-data}
In this section, we recall some elementary results
of sheaf theory
and define permissible $C^{\infty}$-divisors on rational
polyhedral spaces and
the (cohomological) local Morse data 
for permissible $C^{\infty}$-divisors.

\subsection{Sheaf theory on locally compact 
Hausdorff spaces.}
In this subsection, we follow some classical results of
the theory of sheaves on locally compact topological
spaces from \cite{iversenCohomologySheaves1986a,
MR1299726,MR1269324,MR2050072}.
We mainly follow the notation of \cite{MR1299726}
unless otherwise specified.

\subsubsection{Notes for modules over PID}

Henceforth, we assume $A$ is a PID. 
Then, the (global) homological dimension of 
$\catn{Mod}(A)$ is less than or equal to $1$
\cite[Exercise I.17, I.28]{MR1299726}.
Therefore, for any $M^{\bullet}\in \opn{Ob}(D^{b}(\catn{Mod}(A)))$
there exists the following quasi-isomorphism 
\cite[Exercise I.18]{MR1299726}:
\begin{align}
M^{\bullet}\simeq 
\bigoplus_{i\in \Z}H^{i}(M^{\bullet})[-i].
\end{align}

Let $\catn{mod}(A)$ be the category of 
finitely generated $A$-modules.
Since $A$ is a PID, then the Grothendieck group 
$K_0(\catn{mod}(A))$ is isomorphic to $\Z$ by
the rank $\opn{rk}_A M$ of a finitely generated
$A$-module $M$.
If $M^{\bullet}$ is a bounded $\Z$-graded finitely generated $A$-module,
then we can define the Euler characteristic 
$\chi(M^{\bullet})\deq 
\sum_{i\in \Z}(-1)^{i}\opn{rk}_A M^{i}$.
If $(M^{\bullet},d)$ is a bounded chain complex of finitely 
generated $A$-module, then 
$\chi(M^{\bullet})=\chi(H^{\bullet}(M^{\bullet}))$.
We note $K_0(\mathcal{C})\simeq 
K_0(D^{b}(\mathcal{C}))$ for any Abelian category 
$\mathcal{C}$
\cite[Exercise I.27]{MR1299726}.

\subsubsection{Some operations of sheaves}

From now on, we recall some elementary 
operations of sheaves. 
Let $f\colon Y\to X$ be a continuous map.
Then, $f$ induces the functor $f_!$ of
the direct image with proper supports 
\cite[(2.5.1)]{MR1299726}. 
For example, $a_{X!}\mathcal{F}\simeq \Gamma_c(X;\mathcal{F})$.

Let $X$ be a topological space, $Z$ a locally closed 
subset of $X$,
and $i\colon Z\to X$ the inclusion map.
For a given sheaf $\mathcal{F}$ on $X$, we set
\begin{align}
\mcal{F}|_{Z}\deq i^{-1}\mcal{F}, \quad 
\mcal{F}_Z\deq i_! i^{-1}\mcal{F}, \quad 
(A_X)_Z\deq i_!i^{-1}A_X, \quad 
\Gamma(Z;\mathcal{F})\deq \Gamma(Z;\mathcal{F}|_{Z}).
\end{align}

From \cite[Proposition 2.5.4]{MR1299726}, $i_!$ is exact 
and $\mcal{F}_Z$ is isomorphic to the sheaf defined in
\cite[p.93]{MR1299726}.
If $Z'$ be a closed subset of $Z$, then
there exists the following exact sequence 
\cite[Proposition 2.3.6.(v)]{MR1299726}:
\begin{align}
0\to \mathcal{F}_{Z\setminus Z'} \to 
\mathcal{F}_Z \to \mathcal{F}_{Z'}\to 0.
\end{align}

Let $\Gamma_{Z}(\mcal{F})$ be the sheaf of sections of 
$\mcal{F}$ supported by a locally closed subset $Z$
\cite[Definition 2.3.8]{MR1299726} and 
$\Gamma_{Z}(X;\mcal{F})$ the global section of 
$\mathcal{F}$ supported by $Z$.
This sheaf is defined as follows:

\begin{enumerate}
\item Let $U$ be an open subset of $X$ such that $U$ contains
$Z$ as a closed subset of $U$ then we write
\begin{align}
  \Gamma_Z(U;\mcal{F})\deq \opn{Ker}(\mcal{F}(U) \to 
\mcal{F}(U\setminus Z)).
\end{align}
\item Moreover, if $V$ is an open subset of $U$ contain $Z$, then
$\Gamma_{Z}(U;\mcal{F})\simeq \Gamma_{Z}(V;\mcal{F})$
via the canonical morphism. Therefore, we can define
$\Gamma_{Z}(X;\mcal{F})\deq \Gamma_{Z}(U;\mcal{F})$.
\item The sheaf $\Gamma_{Z}(\mcal{F})$ is the presheaf 
$U\mapsto \Gamma_{Z\cap U}(U;\mcal{F})$ on $X$.
\end{enumerate}

The functor $\Gamma_Z(X;\cdot)$ is a left exact functor from 
$\opn{Mod}(A_X)$ to $\opn{Mod}(A)$,
and $\Gamma_{Z}(\cdot)$ is a left exact functor from
$\opn{Mod}(A_X)$ to $\opn{Mod}(A_X)$ 
\cite[Proposition 2.3.9 (i)]{MR1299726}.

If $j\colon V\to X$ is an open inclusion of $X$, then 
$\Gamma_V(\mathcal{F})\simeq j_*j^{-1}\mathcal{F}$ 
\cite[Proposition 2.3.9 (iii)]{MR1299726}.

Let $f\colon Y\to X$ be a continuous map,
$\mathcal{F},\mathcal{F}'\in \opn{Ob}(\catn{Mod}(A_X))$\\
and $\mathcal{G}\in\opn{Ob}(\catn{Mod}(A_Y))$.
Then, we have the following isomorphisms
\cite[(2.3.18)-(2.3.20)]{MR1299726}:
\begin{align}
\label{equation-adjoint-sheaf-cut-off}
\mathcal{H}om_{A_X}(\mathcal{F}_{Z},\mathcal{F}')
\simeq \mathcal{H}om_{A_X}(\mathcal{F},\Gamma_{Z}(\mathcal{F}'))
, \quad \\
\label{equation-property-sheaf-cut-off}
f^{-1}\mathcal{F}_Z\simeq 
(f^{-1}\mathcal{F})_{f^{-1}(Z)},\quad 
\Gamma_Zf_*\mathcal{G}\simeq f_*\Gamma_{f^{-1}(Z)}
\mathcal{G}.
\end{align}

The first isomorphism \eqref{equation-adjoint-sheaf-cut-off}
deduces that $\Gamma_{Z}(\cdot)$ preserves injectives since
$(\cdot)_Z$ is exact.

Let $X$ be a topological space, $Z$ a locally 
closed subset of $X$ and $Z'$ a closed subset of $Z$.
There exists the following distinguished triangle 
for any $\mathcal{F}^{\bullet}\in \opn{Ob}(D^{b}(A_X))$ 
\cite[(2.6.32)]{MR1299726}:
\begin{align} \label{equation-exact-local}
  R\Gamma_{Z'}(\mcal{F}^{\bullet})\to 
R\Gamma_{Z}(\mcal{F}^{\bullet})\to 
R\Gamma_{Z\setminus Z'}(\mcal{F}^{\bullet})\xto{[1]} 
R\Gamma_{Z'}(\mcal{F}^{\bullet})[1].
\end{align}

\begin{example}
\label{example-local-cohomology}
Let $Z$ be an open subset of a Hausdorff space $X$, 
$\mathcal{F}^{\bullet}=A_{X}$ 
and $U=Z\setminus Z'$.
By taking the global 
section for \eqref{equation-exact-local}, we get 
the following exact sequence:
\begin{align}
  \cdots \to H^{i}_{Z'}(X;A_X)\to
H^{i}(Z;A_{Z})\to H^{i}(U;A_U)\to 
H^{i+1}_{Z'}(X;A_X) \to \cdots.
\end{align}
The cohomology $H^{\bullet}_{\{x\}}(X;A_X)$ is called
the local cohomology of $X$ at $x$.
If $Z$ is acyclic and $x\in Z$, then $H^{\bullet}_{\{x\}}(X;A_X)$ is 
isomorphic to the $(-1)$-shift of the reduced cohomology 
$\tilde{H}^{\bullet-1}(Z\setminus \{x\};
A_{Z\setminus \{x\}})$
of the constant sheaf on $Z\setminus \{x\}$
(see 
\cite[p.199]{hatcherAlgebraicTopology2002a} for 
reduced singular cohomology).
\end{example}

\begin{example}

We mention two elementary properties 
about stalks of $R\Gamma_{Z}(\mcal{F})$ of 
a sheaf $\mathcal{F}$ on 
a topological space $X$.

\begin{enumerate}
\item Let $j:V\to X$ be an open inclusion and $Z$ be a 
closed subset of $X$. Then, 
\begin{align}
\Gamma_{Z\cap V}(\mathcal{F}|_V)=\Gamma_{Z}(\mathcal{F})|_V.
\end{align}
We can see the above from the definition or the following
equations:
\begin{align}
\Gamma_{j^{-1}(Z)} \circ j^{-1}
=j^{-1}\circ j_*\circ \Gamma_{j^{-1}(Z)} \circ j^{-1}
=j^{-1}\circ\Gamma_{Z}\circ \Gamma_V
=j^{-1}\circ \Gamma_V \circ \Gamma_Z
=j^{-1} \circ \Gamma_Z.
\end{align}
Since $j^{-1}=j^{!}$ and $j_!$ is exact, 
$j^{-1}$ preserves injectives. 
Therefore, $R\Gamma_{j^{-1}(Z)} 
\circ j^{-1}=j^{-1}\circ R\Gamma_Z$.
In particular, for any point $v\in V$ we have
\begin{align}
(R\Gamma_{Z\cap V}(\mcal{F}|_{V}))_v\simeq 
(R\Gamma_{Z}(\mcal{F}))_v.
\end{align}
Therefore, the stalk is independent
of the choice of $V (\ni v)$.
\item Let $i: S\to X$ be a closed inclusion.
Since $i_!=i_*$ and $i_!$ is exact, we have
\begin{align}
  i^{-1}R\Gamma_{Z}(i_*\mcal{F})\simeq 
i^{-1}R(\Gamma_{Z}\circ i_*)\mcal{F}
  \simeq i^{-1}R(i_*\circ \Gamma_{i^{-1}(Z)})\mcal{F}
  \simeq R\Gamma_{Z\cap S}\mcal{F}
\end{align}
In particular, $(R\Gamma_{Z}(i_*\mcal{F}))_v
  \simeq (R\Gamma_{Z\cap S}\mcal{F})_v$ for any $v\in S$.
\end{enumerate}

By combining the above two examples, for an open 
subset $V' (\ni v)$ of $S$ we get 
\begin{align} \label{equation-local-stalk}
(R\Gamma_{Z}((A_{X})_S))_v\simeq 
(R\Gamma_{Z\cap S}(A_{S}))_v\simeq
(R\Gamma_{Z\cap S\cap V'}(A_{V'}))_v.
\end{align}
We use the isomorphism
\eqref{equation-local-stalk}
in \cref{prop-local-morse-data}.
\end{example}

\subsection{Cohomological local Morse data and index}

The stalk complex and its cohomology written below are 
the most important notions
for this paper, which is naturally appeared in
microlocal sheaf theory \cite{MR1299726}, and
the complex is called ``local Morse datum'' in \cite[p.271]{MR2031639} 
or microlocal stalk in \cite{MR4132582}.
(The double quotation mark here is taken from 
\cite[p.271]{MR2031639}.)
\begin{definition}[{Cohomological local Morse data}]
\label{definition-local-morse-data}
Let $f\colon X\to {\mathbb{R}}$ be a continuous map on a 
topological space $X$ and $x\in X$. Fix 
$\mathcal{F}^{\bullet}\in 
\opn{Ob}(D^{b}(\catn{Mod}(A_X)))$.
The cohomology of the \emph{cohomological local Morse 
data} of the pair $(\mathcal{F}^{\bullet},f)$ at $x$ 
is the following graded $A$-module: 
\begin{align}
\opn{LMD}^{\bullet}(\mcal{F}^{\bullet},f,x)\deq 
H^{\bullet}(R\Gamma_{\{f\geq f(x)\}}\mathcal{F}^{\bullet})_x
=H^{\bullet}(R\Gamma_c(\{x\};R\Gamma_{\{f\geq f(x)\}}
\mathcal{F}^{\bullet})).
\end{align}

The \emph{local Morse index} of $f$ at $x$ is the 
following number if \\
$R\Gamma_{\{f\geq f(x)\}}(\mathcal{F}^{\bullet})_x
\in D^{b}(\catn{mod}(A))$:
\begin{align} \label{equation-local-index}
\opn{ind}_x^{A}(f)\deq \chi(\opn{LMD}^{\bullet}(A_X,f,x)), \quad
\opn{ind}_x(f) \deq \opn{ind}_x^{\mathbb{Z}}(f).
\end{align}

\end{definition}

In this paper, we simply say 
\emph{local Morse data} as cohomological Morse data 
if we are not afraid of confusing it with local Morse data
in the theory of stratified Morse theory 
\cite[Part I. Definition 3.5.2]{MR932724}.

\begin{remark}
\begin{enumerate}
\item We follow the above notation from 
\cite[p.271]{MR2031639} for two reasons.
One of this is for easy to see the triple 
$(\mcal{F}^{\bullet},f,x)$. 
The other one is for emphasizing on analogy to
Morse complexes, see \cref{example-poincare-hopf}.
\item The name of local Morse data may be misleading.
We don't assume $f$ satisfies some 
`Morse condition' since
 we need to use some $C^{\infty}$-function
which is \emph{not} a Morse function in the sense of 
stratified Morse theory. 
We will give an explanation about this in 
\cref{remark-nondeprave-curve}.
Moreover, 
$\opn{LMD}^{\bullet}(\mathcal{F}^{\bullet},f,x)$
is closely related with (real) Milnor fiber, so
the name \emph{(cohomological) local Milnor data} may 
be more appropriate 
than the name cohomological local Morse data. 
\item We also note that we need to use 
the graded module $\opn{LMD}^{\bullet}(\mcal{F},f,x)$ instead of 
the complex $(R\Gamma_{\{f\geq f(x)\}}\mathcal{F})_x$ in order to define
a graded module associated with line bundle on tropical
spaces. See also \cref{remark-differential-graded-module} for more about 
this reason.
\end{enumerate}
\end{remark}

Before we see some important examples of cohomological
local Morse data, we will see some elementary
properties of them.
By \eqref{equation-local-stalk}, we get the following
isomorphism:

\begin{proposition} \label{prop-local-morse-data}
Let $f\colon X\to \mathbb{R}$ be a continuous function
on $X$, $S$ a closed subset
of $X$, $V$ open subset of $S$ and $x\in V$. Then,
\begin{align} 
\opn{LMD}^{\bullet}((A_{X})_S,f,x
) &\simeq  
\opn{LMD}^{\bullet}(A_S,f|_S,x)
\simeq \opn{LMD}^{\bullet}(A_{V},f|_V,x) \\
&\simeq \varinjlim_{U\in\opn{Nbh}(x,S)} 
\tilde{H}^{\bullet-1}(S\cap \{f<f(x)\}\cap U;A),
\end{align}
where $\tilde{H}^{i-1}(X;A)$ is
the $(i-1)$-th reduced cohomology
of $A_X$ and $\opn{Nbh}(x,S)$ is 
the set of open neighborhoods of $x$ in $S$. 
\end{proposition}
\begin{proof}
The first and second isomorphism follows 
from \eqref{equation-local-stalk} 
by replacing $Z$ (resp. $V'$) 
with $\{f\geq f(x)\}$ (resp. $V$) and the isomorphism
$H^{\bullet}(\mathcal{F}^{\bullet})_x\simeq 
H^{\bullet}(\mathcal{F}^{\bullet}_x)$ for any 
$\mathcal{F}^{\bullet}\in D^{b}(A_X)$ and $x\in X$.

From now on, we may assume $X=S=V$.
From \eqref{equation-exact-local} and by taking stalks,
we have the following distinguished triangle:
\begin{align}
(R\Gamma_{\{f\geq f(x)\}}A_S)_x \to 
(A_S)_x\to 
(R\Gamma_{\{f<f(x)\}}A_S)_x\to 
(R\Gamma_{\{f\geq f(x)\}}A_S)_x[1].
\end{align}
Let $j\colon \{f<f(x)\}\to S$ be the open inclusion map. 
We can see 
$R\Gamma_{\{f< f(x)\}}\simeq Rj_*\circ j^{-1}$
since $j^{-1}=j^{!}$ preserve injectives.
From the elemental properties of higher direct image
(e.g. \cite[Chapter II. Proposition 5.11]{iversenCohomologySheaves1986a}),
we have
\begin{align}
(R^{i}\Gamma_{\{f< f(x)\}}A_S)_x\simeq 
(R^{i}j_* A_{\{f<f(x)\}})_x \simeq 
\varinjlim_{U\in \opn{Nbh}(x,S)} 
H^{i}(\{f<f(x)\}\cap U;A).
\end{align}
\end{proof}

We will see several examples of local Morse data.

\begin{example}
Let $f\colon X \to \mathbb{R}$ a continuous function 
such that $\{f<f(x_0)\}\cap S=\emptyset$
for some $x_0\in X$. 
Then,
\begin{align}
\opn{LMD}^{\bullet}(A_S,f|_{S},x_0)\simeq 
\tilde{H}^{\bullet-1}(\emptyset;A)\simeq A[0].
\end{align}

\end{example}

\begin{example} \label{example-fundamental-system}
Let $\{U_j\}_{j\in J}$ be an open fundamental system of 
$x\in X$ satisfying
$H^{\bullet}(\{f<f(x)\}\cap U_j;A)\simeq 
H^{\bullet}(\{f<f(x)\}\cap U_{j'};A)$ via the 
inclusion map $U_{j'}\hookto U_j$. Then, 
\begin{align}
\opn{LMD}^{\bullet}(A_X,f,x)\simeq H^{\bullet-1}(\{f<f(x)\}\cap U_j;A).
\end{align}

\end{example}

\begin{example}
\label{example-morse-index}
In this example, we see
the local Morse index of is a certain generalization of 
the Poincar\'e--Hopf index of the gradient 
$\opn{grad}(f)$ of a smooth function $f$ 
when $f$ is a Morse function. 
We set
\begin{align}
f_{n,m}\colon \mathbb{R}^{n}\times \mathbb{R}^{m}\to \mathbb{R};
(x,y)\mapsto \|x\|_{\mathbb{R}^{n}}^2
-\|y\|_{\mathbb{R}^{m}}^2.
\end{align}
In order to calculate 
$\opn{LMD}^{\bullet}(A_{\mathbb{R}^{n+m}},f_{n,m},0)$,
for any $\varepsilon\in \mathbb{R}_{}>0$ we get
\begin{align}
\partial (\bar{B}^{n}_{\varepsilon}(0)
\times \bar{B}^{m}_{\varepsilon}(0))=
& S^{n-1}_{\varepsilon}(0)\times B^{m}_{\varepsilon}(0) 
\cup B^{n}_{\varepsilon}(0)\times S^{m-1}_{\varepsilon}(0) 
\notag \\
& \cup 
S^{n-1}_{\varepsilon}(0)\times S^{m-1}_{\varepsilon}(0) \\
=& \{\|\cdot\|_{\mathbb{R}^{n}}=\varepsilon\} 
\times \{\|\cdot\|_{\mathbb{R}^{m}} < \varepsilon\} 
\cup \{\|\cdot\|_{\mathbb{R}^{n}} < \varepsilon\} \times 
\{\|\cdot\|_{\mathbb{R}^{m}}= \varepsilon\}
\notag \\
& \cup 
\{\|\cdot\|_{\mathbb{R}^{n}}=\varepsilon\} \times 
\{\|\cdot\|_{\mathbb{R}^{m}}= \varepsilon\}.
\end{align}

Therefore,  
\begin{align}
\partial (\bar{B}^{n}_{\varepsilon}(0)
\times \bar{B}^{m}_{\varepsilon}(0))\cap 
\{f_{n,m}<f_{n,m}(0)\}
= B^{n}_{\varepsilon}(0)\times S^{m-1}_{\varepsilon}(0).
\end{align}

The above equation and \cref{prop-local-morse-data} 
give
\begin{align}
\opn{LMD}^{\bullet}(A_{\mathbb{R}^{n+m}},f_{n,m},0)\simeq 
\tilde{H}^{\bullet-1}(S^{m-1};A_S) \simeq A[-m],& \\
\opn{ind}^{A}_{0}(f_{n,m})=(-1)^{m}=
\opn{ind}^{\opn{PH}}_{0}(\opn{grad}(f_{n,m})),&
\end{align}
where $\opn{ind}^{\opn{PH}}_{0}(\opn{grad}(f_{n,m}))$ is 
the Poincar\'e-Hopf index of the gradient 
of $f_{n,m}$ at 
the origin $0$.
From the Morse lemma, 
$\opn{ind}_x^{A}(f)=\opn{ind}^{\opn{PH}}_x(
\opn{grad}(f))$
for every critical point $x$ of a Morse function 
$f\colon M\to \mathbb{R}$ on a 
Riemannian manifold $(M,g)$. 
This equation can be generalized for more cases.
\end{example}

If $X$ is a locally compact Hausdorff space, then there exists a 
closed and compact neighborhood $S$ of $x\in X$. 
By the second isomorphism of \cref{prop-local-morse-data},
we have 
$\opn{LMD}^{\bullet}(A_X,f,x)\simeq 
\opn{LMD}^{\bullet}(A_S,f|_S,x)$, so 
we can choose a sufficiently small closed and 
compact neighborhood 
of $x$.

By \cref{prop-local-morse-data}, we can calculate
of the local Morse index of a continuous function on 
a given polyhedral complex by the theory of sheaves on 
smooth manifolds.

\begin{notation}

We define $Y\deq 
\{x\in X\mid \opn{LMD}^{\bullet}(\mathcal{F}^{\bullet},f,x)
\not \simeq 0\}$ and set

\begin{align}
\opn{LMD}^{\bullet}(\mathcal{F}^{\bullet},f)\deq 
\bigoplus_{x\in Y}
\opn{LMD}^{\bullet}(\mathcal{F}^{\bullet},f,x).
\end{align}
\end{notation}

The graded module $\opn{LMD}^{\bullet}(\mathcal{F}^{\bullet},f)$
behaves like a Morse complex of a Morse function
as graded modules.

Next, we recall about the micro-support of sheaves 
from \cite{MR1299726}.
Let $X$ be a $C^{\alpha}$-manifold for
$\alpha\in\{1,2,\ldots,\infty,\omega\}$ and let
$\mcal{F}^{\bullet}\in\opn{Ob}(D^{b}(A_X))$ and
$\opn{SS}(\mcal{F}^{\bullet})$ 
the micro-support of $\mcal{F}^{\bullet}$:
\begin{definition}[{\cite[Proposition 5.1.1, Definition 5.1.2]{MR1299726}}]
The element $p\in T^{*}X$ 
is  $p\notin \opn{SS}(\mathcal{F}^{\bullet})$
if there exists an open neighborhood $U$ of $p$ such 
that for any $x_1\in X$
and any $C^{\alpha}$-function
$\psi$ defined 
in a neighborhood $V$ of $x_1$, with 
$(x_1;d\psi(x_1))\in U$, we have:
\begin{align}
(R\Gamma_{\{\psi \geq \psi(x_1)\}}(\mathcal{F}^{\bullet}|_V))_{x_1}
\simeq 0,\, \mathrm{i.e.} \, , 
\opn{LMD}^{\bullet}(\mathcal{F}^{\bullet}|_V,\psi,x_1)
\simeq 0.
\end{align}

\end{definition}

As remarked in \cite[Remark 5.1.6]{MR1299726},
the dependence of the micro-support 
$\opn{SS}(\mcal{F}^{\bullet})$ is solely 
on the $C^{1}$-structure of $X$, and the forgetful functor 
$\phi_X \colon D^{b}(A_X)\to D^{b}(\Z_X)$
preserves the micro-support of $\mcal{F}^{\bullet}$ 
\cite[Remark 5.1.5]{MR1299726}.
For example, if $M$ is a closed submanifold of $X$, then
$\opn{SS}((A_{X})_M)=T^{*}_M X$ 
\cite[Proposition 5.3.2]{MR1299726}.

For simplicity,
$\opn{SS}(\mcal{F})_x\deq \opn{SS}(\mcal{F}) 
\cap T^{*}_x X$ 
and $\check{\pi}_{X}\colon T^{*}X\to X$
is the cotangent bundle.
From the definition of the micro-support of sheaves, 
$\opn{SS}(\mcal{F}|_{U})_x=\opn{SS}(\mcal{F})_x$
for any open neighborhood $U$ of $x$.
The following condition and proposition are also important.
\begin{condition}[{\cite[Proposition 5.4.20]{MR1299726}}]
\label{condition-global-morse}
Let $X$ be a smooth manifold,\\ $A$ a PID, 
$\mathcal{F}^{\bullet}\in D^{b}(A_X)$,
and $f\colon X\to \mathbb{R}$ a smooth function.
The pair $(\mathcal{F}^{\bullet},f)$ satisfies
\begin{enumerate}
\item for all $t\in \mathbb{R}$, 
$\{x\in\opn{supp}(\mathcal{F}^{\bullet})\mid f(x)\leq t\}$ 
is compact,
\item $\sharp (\Lambda_f\cap \opn{SS}(\mathcal{F}^{\bullet}))
<\infty$ where $\Lambda_f$ is the image of 1-form 
$df\colon X\to T^{*}X$,
\item the graded module 
$\opn{LMD}^{\bullet}(\mathcal{F}^{\bullet},f)$ 
is finitely generated (as an $A$-module).

\end{enumerate}

\end{condition}

\begin{proposition}[{\cite[Remark 1.2]{MR1160840},
\cite[Proposition 5.4.20]{MR1299726}}]
Under \cref{condition-global-morse}, 
\begin{align}
\label{equation-global-morse}
\chi(H^{\bullet}(X;\mathcal{F}^{\bullet}))=
\chi (\opn{LMD}^{\bullet}(\mathcal{F}^{\bullet},f)).
\end{align}

\end{proposition}
The RHS of \eqref{equation-global-morse}
can be considered 
as the intersection number 
of two Lagrangian cycles in $T^{*}X$ when 
$\mathcal{F}^{\bullet}$ is a complex of 
constructible sheaves,
and $f$ is an analytic function on an analytic manifold,
see \cite[Corollary 9.5.2 and Theorem 9.5.6]{MR1299726}.
This is called the microlocal index formula or Kashiwara's 
index formula.

\begin{remark}
We can change \cref{condition-global-morse} (2) 
to a weaker version \cite[(H.2)]{MR1160840}:
\begin{enumerate}
\setlength{\leftskip}{0pt}
\item[(2)'] $\check{\pi}_X(\Lambda_f\cap \opn{SS}(\mathcal{F}^{\bullet}))$
is a finite disjoint union of compact subsets of $X$ and
$f|_{\check{\pi}_X(\Lambda_f\cap \opn{SS}(\mathcal{F}^{\bullet}))}$ 
is locally constant.
\end{enumerate}
If $X$ is compact and $f$ is constant, then 
the pair $(A_X,f)$ satisfies 
\cref{condition-global-morse} (1), (3) and (2)',
but not (2).
\end{remark}

\begin{example}
Let $M$ be a compact integral affine manifold and $S$  
a compact rational polyhedral subspace in $M$ and 
$\iota\colon S\hookto M$ the inclusion map of $S$. 
Let $f\colon M \to \mathbb{R}$ be a smooth function such
that the pair $(\iota_*\mathbb{R}_S,f)$ satisfies 
\cref{condition-global-morse}:
\begin{align}
\chi(H^{\bullet}(M;\iota_*\mathbb{R}_S))
=\chi(\opn{LMD}(\iota_*\mathbb{R}_S,f))
\end{align}
for any smooth function $f\colon M\to \mathbb{R}$.
Therefore, this is a certain sheaf theoretic version of 
Poincar\'e--Hopf theorem for some polyhedral spaces. 
The cohomology 
$H^{\bullet}(M;\iota_*\mathbb{R}_S)$
 should be considered as a
certain analog of 
$\chi(H^{\bullet}(X;\iota_*\mcal{O}_Z))$ for the closed embedding
$\iota\colon Z\to X$ of a scheme.
\end{example}

\begin{example}[{cf. \cite[Exercise V.12]{MR1299726}}]
\label{example-poincare-hopf}
Let $U$ be an open neighborhood of the origin 
$0\in \mathbb{R}^{n}$,
$f\colon U\to \mathbb{R}$ be a smooth 
function 
with a unique isolated critical point at $0$. 
From \cref{prop-local-morse-data},
$\opn{LMD}^{\bullet}(A_U,f,0)\simeq 
\opn{LMD}^{\bullet}(A_{B_{\varepsilon}(0)},
f|_{B_{\varepsilon}(0)},0)\simeq 
\opn{LMD}^{\bullet}(A_{\bar{B}_{\varepsilon}(0)},
f|_{\bar{B}_{\varepsilon}(0)},0)$ 
for a sufficiently small $\varepsilon \in \mathbb{R}_{>0}$.
Fix $\varepsilon$ and an open smooth 
inclusion $j\colon U \to M$ of $U$
where $M$ is a compact smooth manifold,
$x_0\deq j(0)$ and $K\deq j(\bar{B}_{\varepsilon}(0))$. 
The sheaf $\mathcal{C}^{\infty}_M$ of smooth functions 
on $M$ is fine, so we can extend
$f|_{\bar{B}_{\varepsilon}(0)}$ to 
$f_0\in C^{\infty}(M)$.
From \cref{prop-local-morse-data}, 
$\opn{LMD}^{\bullet}(A_M,f_0,x_0)
\simeq \opn{LMD}^{\bullet}(A_U,f,0)$
(see \cite[Definition 3.3]{MR2359489} 
for the definition of fine sheaves in this paper).
By a Morse perturbation (see e.g. 
\cite[(2.57)]{MR1873233}), we can deform 
$f_0$ to a smooth function $f_1$ such that 
$f_0|_{K}=f_1|_K$ and $f_1$ is nondegenerate except at 
$x_0$.
Let $\opn{Crit}(f)$ be the set of critical points of $f$.
If $\opn{LMD}^{\bullet}(A_U,f,0)$ is 
a finitely generated graded $A$-module,
then we have
\begin{align} \label{equation-local-poincare-hopf}
\chi(H^{\bullet}(M;A))=
\chi(\opn{LMD}^{\bullet}(A_M,f_1,x_0))+
\sum_{x\in \opn{Crit}(f_1)\setminus \{x_0\}}
\chi(\opn{LMD}^{\bullet}(A_M,f_1,x))
\end{align}
by \cite[Proposition 5.4.20]{MR1299726}.
From the Poincar\'e--Hopf theorem and 
\cref{example-morse-index}, we have
\begin{align} \label{equation-poincare-hopf}
\opn{ind}_0^{A}(f)=
\opn{ind}_x^{A}(f_1) =
\opn{ind}_x^{\opn{PH}}(\opn{grad}(f_1))=
\opn{ind}_0^{\opn{PH}}(\opn{grad}(f)), 
\end{align}
where $\opn{ind}^{\opn{PH}}_x(\opn{grad} f)$ is the 
Poincar\'e--Hopf index at $x$ of the gradient
vector field of $f$ 
for a fixed Riemannian metric.
We note
$\opn{LMD}^{\bullet}(A_M,\tilde{f})=
\bigoplus_{x\in \opn{Crit}(\tilde{f})}
\opn{LMD}^{\bullet}(A_M,\tilde{f},x)$ is isomorphic
to the graded module of an $A$-valued Morse complex of 
$\tilde{f}$ when $\tilde{f}$ is a Morse function.
\end{example}

The following example is a $C^{\infty}$-function on a 
smooth manifold with an isolated critical point
at $x$ such that $\opn{LMD}^{\bullet}(A_M,f,x)$
is finitely generated.

\begin{example}
Let $f\colon M\to \mathbb{R}$ be a 
$C^{\infty}$-function on an $n$-dimensional
smooth Riemannian manifold 
$(M,g)$ and 
$\varepsilon$ a 
sufficiently small positive real number. Then,
for $x\in M$ we have
\begin{align}
\opn{LMD}^{\bullet}(A_{M},f,x)\simeq 
\opn{LMD}^{\bullet}(A_{B_{\varepsilon}^{n}(x)},
f|_{B_{\varepsilon}^{n}(x)},x) \simeq 
\opn{LMD}^{\bullet}(A_{\bar{B}_{\varepsilon}^{n}(x)},
f|_{\bar{B}_{\varepsilon}^{n}(x)},x).
\end{align}

If $f$ is smooth, has an isolated \emph{nondepraved} 
critical point at $x$, 
and $f^{-1}(f(x))\cap \opn{Crit}(f)=\{x\}$, then 
$\{f=f(x)\}\setminus \{x\}$ is a smooth submanifold 
of $M$ and the set 
$\{\{f=f(x)\}\setminus \{x\},\{x\} \}$ induces 
a stratification satisfying Whitney's condition B
\cite[Part I. Lemma 2.5.1]{MR932724}.
For instance, every real analytic function is nondepraved 
at any isolated critical point
~\cite[Part I.2.4]{MR932724}.

For a sufficiently small $\varepsilon\in \mathbb{R}_{> 0}$, 
we can assume that
$S_{\varepsilon}^{n-1}(x)$ and $\{f=f(x)\}$ 
intersect transversely \cite[Part I. Lemma 3.5.1]{MR932724},
and thus $S_{\varepsilon}^{n-1}(x) \cap \{f=f(x)\}$
is also a closed submanifold of $M$ and compact.
A relative version of Ehresmann's lemma for 
$\|\cdot\|_{\mathbb{R}^{n}}^2$ induces the 
following homeomorphism of pairs of smooth manifold 
with boundary:
\begin{align}
&(\bar{B}_{\varepsilon}^{n}(x)\setminus \{x\}, 
 \bar{B}_{\varepsilon}^{n}(x)\cap \{f=f(x)\} \setminus \{x\}) \\ 
\simeq & ((0,\varepsilon] \times S_{\varepsilon}^{n-1}(x),
(0,\varepsilon] \times
(S_{\varepsilon}^{n-1}(x) \cap \{f=f(x)\})). \notag
\end{align}
A proof of
a relative version of Ehresmann's lemma is 
in \cite{359708}, for example.
Besides, there exists a homeomorphism of pairs of
topological spaces:
\begin{align}
\label{equation-cone-pair}
(\bar{B}_{\varepsilon}^{n}(x),
\bar{B}_{\varepsilon}^{n}(x)\cap \{f=f(x)\})
\simeq (\opn{Cone}(S_{\varepsilon}^{n-1}(x)),
\opn{Cone}(S_{\varepsilon}^{n-1}(x) \cap \{f=f(x)\})),
\end{align}
where $\opn{Cone}(X)$ is the topological cone over
$X$
(see e.g. \cite[Example 4.4]{hatcherAlgebraicTopology2002a}).
This homeomorphism is classically proved when 
$f$ is real analytic function
\cite[Theorem 2.10]{MR0239612}.
The homeomorphism \eqref{equation-cone-pair} and
\cref{prop-local-morse-data} induces 
\begin{align}
\opn{LMD}^{\bullet}(A_{\bar{B}_{\varepsilon}(x)},
f|_{\bar{B}_{\varepsilon}(x)},x)\simeq 
\tilde{H}^{\bullet-1}(S_{\varepsilon}^{n-1}(x)
\cap \{f<f(x)\};A).
\end{align}

By the Alexander duality, 
$\tilde{H}^{\bullet}(S_{\varepsilon}^{n-1}(x)\setminus \{f=f(x)\};A)$ 
is finitely generated.
Therefore, $\opn{LMD}^{\bullet}(A_{\bar{B}_{\varepsilon}^{n}(x)},
f|_{\bar{B}_{\varepsilon}^{n}(x)},x)$ is
also finitely generated.

We also add an explanation of real Milnor fibers from
\cite[\textsection 3.2]{MR3779558}. 
If $f$ is real analytic, then 
$S_{\varepsilon}^{n-1}(x)\cap \{f\leq f(x)\}\simeq 
f^{-1}(-\delta)\cap \bar{B}_{\varepsilon}^{n}(x)$
for sufficiently small $\vep>\delta>0$.
The fiber 
$f^{-1}(-\delta)\cap \bar{B}_{\varepsilon}(x)$ is called 
the (negative) \emph{real Milnor fiber}.
By Khimshiashvili's [Him\v{s}ia\v{s}vili] formula \cite{MR0458467}, we have
\begin{align} \label{equation-Khimshiashvili}
  \opn{ind}_x^{A}(f)
=1-\chi(f^{-1}(-\delta)\cap \bar{B}_{\varepsilon}(x))
=\opn{ind}^{\opn{PH}}_{x}(\opn{grad}(f)).
\end{align}

\end{example}

\begin{remark}[{A vanishing properties
for some polyhedral cone}]
Fix $X\deq \mathbb{R}^{n}$.
Let $f\colon X\to {\mathbb{R}}$ be 
a convex $C^{1}$-function $X$ with the minimum point 
at the origin
and $S$ a polyhedral cone in $X$, i.e.,
a polyhedral subset such that 
$rS=S$ for all $r\in \mathbb{R}_{> 0}$.
One of trivial but remarkable point is that
\begin{align}
|\opn{ind}_0^{A}(f|_S)|\neq |\opn{ind}_0^{A}(-f|_S)|
\end{align}
except for special cases such like
\cref{example-poincare-hopf}.
From \cref{example-local-cohomology},
\cref{prop-local-morse-data} and
\cref{example-fundamental-system}, we have
\begin{align}
\opn{LMD}^{\bullet}(A_S,-f|_S,0)=
\tilde{H}^{\bullet-1}(S\setminus \{0\};A_{S}).
\end{align}

If $S$ is a purely $n$-dimensional Bergman fan, then 
$\opn{LMD}^{\bullet}(A_S,-f|_S,0)\simeq 
H^{n-1}(S\setminus\{0\};A)[-n]$ since every Bergman fan has a good subdivision 
\cite[Theorem 1]{MR2185977} and its link is shellable
(e.g. \cite[Theorem 7.9.1]{MR1165544}). Hacking also
studied homotopy types of the link of
the tropicalization of a very affine algebraic variety
at the origin,
and the vanishing property of its homology over
$\mathbb{Q}$ for generic cases
\cite[Theorem 2.5, Remark 2.11]{MR2452307}.
Since any tropical variety is connected
through codimension $1$ \cite[Theorem 3.3.5]{MR3287221},
the vanishing property works when $\dim S \leq 2$.
This phenomenon gives a certain analog of the Kodaira
vanishing theorem for 
positive line bundles, see 
also \cref{example-tropical-kodaira-vanishing}. 
Amini and Piquerez also discuss a certain shellability 
for tropical fans \cite{amini2021homology}.
\end{remark}

\subsection{$C^{\infty}$-divisors
on boundaryless rational 
polyhedral spaces}

In this subsection,
we recall about line bundles on 
rational polyhedral spaces and discuss the notion of
\emph{$C^{\infty}$-divisor}
for boundaryless rational polyhedral spaces
\cite[\textsection 4.2]{MR4246795}.
We mainly follow 
from \cite{MR3903579,gross2019sheaftheoretic},
but our notation are slightly different from
that of them. 

Let $S$ be a rational polyhedral space
\cite[Definition 2.2]{gross2019sheaftheoretic}, 
$\mathcal{O}^{\times}_S$ the sheaf of 
locally integral affine linear function on $S$
and $\mathcal{A}^{0,0}_S$ the sheaf of 
$(0,0)$-superforms on $S$ 
\cite[Definition 2.24]{MR3903579}. 
\begin{remark}
Jell, Shaw and Smacka defined the sheaf of superforms 
for polyhedral spaces 
with a local face structure in \cite{MR3903579}
(see also \cite[Definition 2.4.(c)]{gross2019sheaftheoretic}
for the definition of face structures 
of rational polyhedral space).
We can also define the sheaf of $(0,0)$-superforms on 
a rational polyhedral space as the same way,
at least.
A $(0,0)$-superform on $S$ are also called as 
a smooth function on $S$. In fact, 
the sheaf $\mathcal{A}^{0,0}_S$ of $(0,0)$-superforms on $S$
is equal to the sheaf $\mathcal{C}_S^{\infty}$ of 
smooth functions on $S$ if $S$ is an integral 
affine manifold.
\end{remark}

In various scenes, we assume rational polyhedral spaces
satisfy the following \emph{boundaryless condition}:

\begin{condition}[{\cite[\textsection 4.2]{MR4246795}}]
\label{condition-Rn}
Let $S$ be a rational polyhedral space admitting 
a coordinate system $\{(U_i,\phi_i)\}_{i\in I}$
such that the image of each chart 
$\phi_i\colon U_i \to \mathbb{T}^{n_i}$
is in $\mathbb{R}^{n_i}$.
\end{condition}

The \cref{condition-Rn} is strict for rational polyhedral
spaces, but it is enough to
give a test play for our approach.

As like tropical line bundles on tropical curves
\cite[Definition 4.4]{mikhalkinTropicalCurvesTheir2008a},
there exists the notion of tropical line bundle on 
rational polyhedral spaces
\cite[Definition 3.12]{gross2019sheaftheoretic}.
The set of isomorphism classes
of tropical line bundles
on a rational polyhedral space $S$ is 
bijective with the cohomology $H^{1}(S;\mcal{O}^{\times}_S)$
\cite[Proposition 3.13]{gross2019sheaftheoretic}.
The cohomology $H^{1}(S;\mcal{O}^{\times}_S)$ is called 
the \emph{Picard group} of $S$.
There exists the following exact sequence:
\begin{align}
\label{equation-smoooth-cartier-divisor-sequence}
0\to \mathcal{O}^{\times}_S \to
\mathcal{A}^{0,0}_{S} \to
\mathcal{A}^{0,0}_{S}/\mathcal{O}^{\times}_S\to 0.
\end{align}

This is a $C^{\infty}$-version of the following exact
sequences associated with the sheaf 
$\mcal{M}_S^{\times}$
of (real valued) \emph{tropical rational functions}
\cite[Definition 4.1]{MR3894860}:
\begin{align}
\label{equation-cartier-divisor-sequence}
0 \to \mcal{O}^{\times}_S \to \mcal{M}_S^{\times} 
\to \mcal{M}_S^{\times}/\mathcal{O}^{\times}_S\to 0.
\end{align}
We also note there exist different definitions of  
tropical rational function 
(for example, see \cite{mikhalkin2018tropical}).
By applying the global section functor for the
short exact sequence 
\eqref{equation-cartier-divisor-sequence},
we get a connecting homomorphism 
$\delta \colon 
H^{0}(S;\mcal{M}_S^{\times}/\mathcal{O}^{\times}_S)
\to H^{1}(S;\mathcal{O}^{\times}_S)$.
We note $H^{1}(S;\mcal{M}_{S}^{\times})=0$, 
and thus $\delta$ is surjective
\cite[Lemma 4.5 and Proposition 4.6]{MR3894860}.
An element in $H^{0}(S;\mcal{M}_S^{\times}/\mathcal{O}^{\times}_S)$
is called a \emph{Cartier divisor} on $S$.
Therefore, for every line bundle $\mathcal{L}$ on a rational polyhedral 
space, there exists a Cartier divisor $D$ 
which represents $\mathcal{L}$.

By taking global sections of the
short exact sequence
\eqref{equation-smoooth-cartier-divisor-sequence},
there also exists a connecting homomorphism 
$[\cdot] \colon 
H^{0}(S;\mathcal{A}^{0,0}_{S}/\mathcal{O}^{\times}_S)
\to H^{1}(S;\mathcal{O}^{\times}_S);s \mapsto [s]$. 
From definition, every rational polyhedral space 
$S$ is paracompact and
$\mathcal{A}^{0,0}_S$ is a fine sheaf 
\cite[Lemma 2.7, Proposition 2.26]{MR3903579}, so 
$H^{1}(S;\mathcal{A}^{0,0}_S)=0$ and 
so $[\cdot]$ is also surjective.
\begin{definition}
Let $S$ be a boundaryless rational polyhedral space.
A \emph{$C^{\infty}$-divisor} on 
$S$ is an element of
$H^{0}(S;\mathcal{A}^{0,0}_{S}/\mathcal{O}^{\times}_S)$.
Let 
$(\cdot)^{\mathrm{sm}}\colon 
H^{0}(S;\mathcal{A}^{0,0}_S)\to 
H^{0}(S;\mathcal{A}^{0,0}_{S}/\mathcal{O}^{\times}_S)$ 
be a morphism defined from 
\eqref{equation-smoooth-cartier-divisor-sequence}.
The \emph{principal $C^{\infty}$-divisor} of a $(0,0)$-superform 
$f$ is the image $(f)^{\mathrm{sm}}$ of $f$.
\end{definition}
By $\CDiv(S)$, we mean 
$H^{0}(S;\mathcal{A}^{0,0}_{S}/\mathcal{O}^{\times}_S)$
when $S$ is a boundaryless rational polyhedral space.
We will define $\CDiv(S)$
for any general rational polyhedral space $S$ by replacing 
the sheaf $\mathcal{A}^{0,0}_{S}$ with the sheaf 
$\mathcal{A}^{\mathrm{weak}}_S$ of weakly-smooth 
functions on $S$ in \cref{section-weakly-smooth-functions}.
We will also see a $C^{\infty}$-divisor is
an analog of Lagrangian section later
(\cref{proposition-cartier-lagrangian}).

Furthermore, we use the following representations of 
$C^{\infty}$-divisors like 
\cite[Definition 4.0.12]{MR2810322}.
\begin{definition}[{$C^{\infty}$-Cartier local data}]
\label{definition-Cartier-local-data}
Let $s$ be a $C^{\infty}$-divisor
on a (boundaryless) rational polyhedral space $S$, 
$\{U_i\}_{i\in I}$ an open cover of $S$ and 
$f_i\colon U_i \to \mathbb{R}$ a $(0,0)$-superform
on $U_i$. The set 
$\{(U_i,f_i)\}_{i\in I}$ is 
a \emph{local data} for $s$ if 
$s|_{U_i}=(f_i)^{\mathrm{sm}}$ for all 
$i\in I$. 
\end{definition}

Local data for a $C^{\infty}$-divisor $s$ are also useful for
calculating the divisor class
$[s]\in H^{1}(S;\mathcal{O}^{\times}_S)$
of $s$
by using the \v{C}ech cohomology 
$\check{H}^{1}(S;\mathcal{O}^{\times}_S)$
(see e.g. \cite[Remark 7.21]{MR3560225}).

In the next subsection, 
we will define an analog of Floer complexes of 
Lagrangian submanifolds
(without differential) for $C^{\infty}$-divisors
on boundaryless rational polyhedral spaces
satisfying the permissibility 
condition (\cref{condition-good-divisor}).

\subsection{Permissibility condition}

In this subsection, we define 
the \emph{permissibility condition}
(\cref{condition-good-divisor}) for
$C^{\infty}$-divisors on
boundaryless rational polyhedral spaces.
We will define an analog of a Morse complex
without differential for permissible 
$C^{\infty}$-divisors later.
In order to define the prepermissibility condition,
we recall some necessary sheaves
on rational polyhedral spaces from 
~\cite{gross2019sheaftheoretic}.

Fix a rational polyhedral space $S$.
We set
$\tform^{1}_{\mathbb{Z},S}
\deq \mathcal{O}^{\times}_S/\mathbb{R}_{S}$,
$\tform^{1}_{\mathbb{R},S}\deq 
\tform^{1}_{\mathbb{Z},S}
\otimes_{\mathbb{Z}_S} \mathbb{R}_S$ 
and $T_x S\deq \opn{Hom}_{\mathbb{Z}}(
(\tform^{1}_{\mathbb{Z},S})_x,\mathbb{R})$.
From the definition of rational polyhedral spaces,
the stalk $(\tform^{1}_{\mathbb{Z},S})_x$
at $x$ is a free 
$\mathbb{Z}$-module of finite rank.

In order to define the prepermissibility condition,
we define the subset $\opn{SS}(S)_x$ of cotangent
space $(\tform^{1}_{\mathbb{R},S})_x$ at $x\in S$ for 
a (boundaryless) rational polyhedral
space $S$.

Let $S$ be a rational polyhedral space, 
$x\in S$ and $\opn{LC}_x S$
the local cone of $S$ at $x$ 
\cite[2.2]{gross2019sheaftheoretic}.
This is a subset of
$T_{x}S$ which spans $T_{x}S$
\cite[Corollary 2.6]{gross2019sheaftheoretic}.
Then, there exists an open neighborhood 
$V$ of $0$ in $\opn{LC}_x S$ and
a canonical morphism $\phi \colon V\to S$ 
such that $\phi(0)=x$ and
$d_{0}\phi\colon T_0(\opn{LC}_xS) \simto T_x S$
\cite[Proposition 2.5]{gross2019sheaftheoretic}.
The $d_0 \phi$ does not depend on the choice of 
$V$. 
We frequently identify $T_0^{*}(\opn{LC}_x S)$ with 
$T_x^{*}S=(\tform_{\mathbb{R},S}^{1})_x$ for simplicity.
Under this identification, we can define the 
$\opn{SS}(S)_x$ as follows:
\begin{definition}
For the above notation, we set
\begin{align}
\opn{SS}(S)_x\deq 
\opn{SS}((\mathbb{Z}_{T_x S})_{\opn{LC}_x S})_0.
\end{align}

\end{definition}

In general, $\opn{SS}(S)_x$  
is complicated, but this is related with the 
following vector 
spaces:

\begin{definition}[{Lineality subspace 
\cite[\textsection 2.1]{MR4246795}}]
Let $Z$ be a subset of
$\mathbb{R}^{n}$. A \emph{lineality subspace} of $Z$ is 
a linear subspace $V$ of
$\mathbb{R}^{n}$ such that $x+s\in Z$ for all 
$s\in Z$ and $x\in V$, i.e.,
\begin{align}
Z+V=Z,
\end{align}
where $Z+V$ is the Minkowski sum of $Z$ and $V$.
\end{definition}
If $0\in Z$, then 
every lineality subspace of $Z$ is a subset of $Z$.
If $V$ and $W$ are a lineality subspace of $Z$, 
then the Minskowski sum $V+W$ is 
also a lineality space.
We write $\opn{lineal}(Z)$ for the maximal lineality 
subspace of $Z$. If $P$ is a polyhedron in 
$\mathbb{R}^{n}$, $\opn{lineal}(P)$ is the lineality space
in the sense of \cite[p.60]{MR3287221}, see also 
\cite[\textsection 5]{MR3041763}.

\begin{proposition}
\label{condition-good-linearity-space}
Let $S$ be a boundaryless rational polyhedral space 
and $x\in S$. Then,
\begin{align} \label{equation-W-SS-Lin}
(\mathcal{W}_{1,S}^{\mathbb{R}})_{x}^{\bot}\subset 
\opn{span}_{{\mathbb{R}}}\opn{SS}(S)_x\subset
\opn{lineal}(\opn{LC}_xS)^{\bot},
\end{align}
where $(V)^{\bot}$ is the orthogonal complement of 
$V$ for the standard pairing of dual vector spaces
and $\mathcal{W}_{1,S}^{\mathbb{R}}=
\mathcal{H}om(\tform^{1}_{\mathbb{Z},S};\mathbb{R}_S)$ is 
the sheaf of (real) wave tangent 
vector spaces of $S$ at $x$
(see \cite[1.3]{mikhalkinTropicalEigenwaveIntermediate2014a} 
and 
\cite[Definition 2.16, Proposition 2.20, Remark 2.21]{yamamotoTropicalContractionsIntegral2021}).
\end{proposition}
\begin{proof}
Since
$\opn{LC}_x S\simeq 
\opn{lineal}(\opn{LC}_x S)\times 
(\opn{LC}_x S/\opn{lineal}(\opn{LC}_x S))$
and 
\cite[Proposition 5.4.1]{MR1299726}, we have
\begin{align} 
\opn{SS}(\opn{LC}_x S)_0\simeq \{0\}\times 
\opn{SS}(\opn{LC}_x S/\opn{lineal}(\opn{LC}_x S))_0
\subset \opn{lineal}(\opn{LC}_xS)^{\bot}.
\end{align}

Next, we prove the following inclusion:
\begin{align}
(\mathcal{W}_{1,S}^{\mathbb{R}})_{x}^{\bot}=
(\bigcap_{y \in (\opn{LC}_xS)_{\reg}} 
\opn{span}_{\mathbb{R}} U_y)^{\bot}=
\sum_{y \in (\opn{LC}_xS)_{\reg}} 
(\opn{span}_{\mathbb{R}} U_y)^{\bot} \subset 
\opn{span}_{\mathbb{R}}\opn{SS}(\opn{LC}_x S)_0,
\end{align}
where $(\opn{LC}_xS)_{\reg}$ is
the set of generic points in $\opn{LC}_xS$ 
\cite[\textsection 4.B]{MR3894860},
and $U_y$ is the intersection of
$(\opn{LC}_xS)_{\reg}$ and a sufficiently 
small open ball $V_y$ of $y$ in $T_0(\opn{LC}_xS)$.
If the diameter of $V_y$ is sufficiently small, 
then $\opn{span}_{\mathbb{R}} U_y$ is independent of 
the choice of $V_y$.
We can prove the first equation since 
$(\opn{LC}_x S)_{\reg}$ is open dense and 
$T_xS$ is generated by $\{\opn{span}_{\mathbb{R}}U_y\}_{y\in(\opn{LC}_x S)_{\reg}}$.

Fix $y\in (\opn{LC}_x S)_{\mathrm{gen}}$ and
$v\in (\opn{span}_{\mathbb{R}} U_y)^{\bot}$.
In order to prove $v\in \opn{SS}(\opn{LC}_x S)_0$, 
we will find a $C^{\infty}$-function 
$f\colon T_x S \to \mathbb{R}$ 
such that $f(0)=0$ $df(0)=v$ and 
\begin{align}
(R_{\{f\geq 0\}}(A_{T_xS})_{\opn{LC}_xS})_0
\not \simeq 0.
\end{align}
From now on, we assume $T_x S=\mathbb{R}^{n}$.

By \cref{prop-local-morse-data}, 
we only need to find a $C^{\infty}$-function 
$f\colon T_x S \to \mathbb{R}$ such that 
\begin{align}
\label{equation-non-acyclic}
1-\chi_{\mathrm{top}}(\{x\in \opn{LC}_x S\cap B_{\varepsilon}^{n}(0)
\mid f(z)<0\})\ne 0
\end{align}
for any sufficiently small $\varepsilon >0$.
 
Let 
$\rho_{\mathbb{R}_{\geq 0}}\colon 
\mathbb{R}\to \mathbb{R}$ be the following
well-known $C^{\infty}$-function:
\begin{align}
\label{equation-elementary-bump}
\rho_{\mathbb{R}_{\geq 0}}(x)=
\begin{cases}
e^{-1/x},\quad  (x> 0) \\
0, \quad (x\leq 0).
\end{cases}
\end{align}

Let $l\colon \mathbb{R}^{n}\to \mathbb{R}$
be an affine map. Then, the pullback 
$l^{*}\rho_{\mathbb{R}_{\geq 0}}$ is a 
$C^{\infty}$-function such that 
\begin{align}
\label{equation-pullback-support}
\{l^{*}\rho_{\mathbb{R}_{\geq 0}}\ne 0\}
=\{l^{*}\rho_{\mathbb{R}_{\geq 0}}> 0\}
=\{l>0\}.
\end{align}
Let $P$ be an $n$-dimensional convex polyhedron 
in $\mathbb{R}^{n}$.
If $P=\bigcap_{i=1}^{m}\{l_i\geq 0\}$
for a family $\{l_i\}_{i=1,\ldots,m}$
of affine linear maps on $\mathbb{R}^{n}$,
then $\rho_P\deq\prod_{i=1}^{m}l_i^{*}\rho_{\mathbb{R}_{\geq 0}}$
is a $C^{\infty}$-function satisfies
\begin{align}
\{\rho_P
\ne 0\}
=\bigcap_{i=1}^{m}\{l_i>0\}=\opn{int}(P),
\end{align}
where $\opn{int}(P)$ is the interior of $P$.
Besides, $d\rho_P(x)=0$ if $x\notin \opn{int}(P)$.

Let $K\subset V_y$ be an $n$-dimensional convex polytope
containing $y$ and 
$P\deq \mathbb{R}_{\geq 0}K$. 
Then, 
\begin{align}
\label{equation-bump-set}
\{z\in\opn{LC}_xS\mid v(z)-\rho_P(z)<0\}=
\{z\in\opn{LC}_xS\mid v(z)<0\}\sqcup (\opn{LC}_x S \cap \opn{int}(P)).
\end{align}

\eqref{equation-bump-set} follows from $v|_{\mathbb{R}_{\geq 0}U_y}=0$ 
and $P\subset \mathbb{R}_{\geq 0}U_y$. 
The RHS of \eqref{equation-bump-set} is a cone and 
$\opn{LC}_x S \cap \opn{int}(P)$ is convex.
Hence, we have
\begin{align}
\chi_{\mathrm{top}}(
\{z\in \opn{LC}_x S\cap B^{n}_{\varepsilon}(0)
\mid v(z)-\rho_{P}(z)<0\}) \notag \\
-\chi_{\mathrm{top}}(
\{z\in \opn{LC}_x S \cap B^{n}_{\varepsilon}(0)\mid v(z)<0
\})=1
\end{align}
for any sufficiently small $\varepsilon>0$.
Therefore, $v-\rho_{P}$ or $v$ satisfies 
\eqref{equation-non-acyclic}.
\end{proof}

\begin{remark}
In many cases, the three vector spaces in 
\eqref{equation-W-SS-Lin} are equal.
For instance, if $S$ is a tropical curve 
or integral affine manifold, 
then $S$ satisfies this condition.
Let $S$ be the tropical hypersurface of
$f=\max\{x_1,\ldots,x_n,1\}$ in $\mathbb{R}^{n}$. 
We can see
$((\mathcal{W}_{1,S}^{\mathbb{R}})_{x})^{\bot}=\opn{span}_{{\mathbb{R}}}(\opn{SS}(S)_x)=\tform_{S,x}^{1}$
when $x\deq (0,\ldots,0)$. 
Of course, the other points in $S$ satisfies 
that the three vector space in 
\eqref{equation-W-SS-Lin} is equal. If 
$S=\mathbb{R}_{\geq 0}$, then 
$((\mathcal{W}_{1,S}^{\mathbb{R}})_{0})^{\bot}=\{0\}$ and
$\opn{span}_{\mathbb{R}}(\opn{SS}(S)_0)=
(\opn{lineal}(\opn{LC}_0 S))^{\bot}=\mathbb{R}$. 
We expect the middle and right vector space in 
\eqref{equation-W-SS-Lin} is equal for various 
important cases.
\end{remark}

From now on, we define the prepermissibility conditions 
for $C^{\infty}$-divisors on boundaryless
rational polyhedral spaces.
This condition is one of necessary conditions
to define the local Morse data
for $C^{\infty}$-divisor.
From the definition of superforms on 
polyhedral spaces 
\cite[Definition 2.10, 2.24]{MR3903579}, we 
can define the set of \emph{critical points} 
of $(0,0)$-superforms on
rational polyhedral spaces
like that of $C^{\infty}$-function on differential
manifolds.
Let $f\colon S \to \mathbb{R}$ be a $(0,0)$-superform on 
a boundaryless rational polyhedral space $S$. 
The differential $d''f\in \mathcal{A}^{0,1}_S$ of $f$ has the 
differential vector $df (x)\in (\tform_{S}^{1})_x$ for any 
$x\in S$ from the definition of $(0,0)$-superforms
and properties of the local 
cone \cite[Corollary 2.6]{gross2019sheaftheoretic}.
We write $\opn{Crit}(f)$ for the set
of critical points of $f$.
For a general rational polyhedral space $S$,
we can also define the 
differential of $f$ at $x$ by the pullback of 
the germ of morphisms 
$(\opn{LC}_x S,0)\to (S,x)$, and we will use this
in \cref{definition-prepermissible-weakly-smooth}.

\begin{condition}[{Prepermissiblity condition}] \label{cond: prepermissible}
Let $S$ be a boundaryless rational polyhedral space.
A $(0,0)$-superform $f\colon S \to {\mathbb{R}}$ is \emph{prepermissible} at $x$ 
if 
\begin{align}
df(x)\notin (\opn{span}_{{\mathbb{R}}}
\opn{SS}(S)_x+\tform_{\mathbb{Z},S,x}^{1})
\setminus \tform_{\mathbb{Z},S,x}^{1}.
\end{align}
\end{condition}

From definition, if $f$ is prepermissible at $x$ then
$f|_{U_x}+m$ is also prepermissible at $x$ for any 
open neighborhood $U_x$ of $x$ and 
$m\in \mathcal{O}_{S}^{\times}(U_x)$.
Therefore, we can define the prepermissibility condition
for principal $C^{\infty}$-divisors directly.
\begin{remark}
Suppose $\opn{span}_{\mathbb{R}}(\opn{SS}(S)_x)
=(\tform_{\mathbb{R},S}^{1})_x$. 
Then, $f$ is prepermissible if and only if 
$df(x)\in (\tform_{\mathbb{Z},S}^{1})_x$.
Therefore, the prepermissibility
condition may be not a generic condition.
If $\dim_{\mathbb{R}} \opn{span}_{{\mathbb{R}}}
(\opn{SS}(S)_x)<\dim_{\mathbb{R}} (\tform_{\mathbb{R},S}^{1})_x$,
then the 
set of prepermissible cotangent vector at $x$ is dense and 
connected.
\end{remark}

\begin{definition}
\label{definition-prepermissible-smooth-cartier-divisor}
A $C^{\infty}$-divisor $s$ is \emph{prepermissible} 
if there exists a local data $\{(U_i,f_i)\}_{i\in I}$ 
for $s$
such that $f_i$ is prepermissible for all $i\in I$. 
In general, a pair $(s',s)$ of smooth 
Cartier divisors $s,s'$ is 
\emph{prepermissible} if $s-s'$ is 
prepermissible. 
\end{definition}

Let $s_0$ be the zero $C^{\infty}$-divisor on $S$. 
Then, a pair $(s_0,s)$ of $C^{\infty}$-divisors 
is prepermissible if and only if $s$ is prepermissible.
If a $C^{\infty}$-divisor $s$ is prepermissible, 
then every local data for $s$ satisfies the above 
prepermissibility condition.

\begin{definition}
\label{definition-divisor-intersection}
Let $(s,s')$ be a prepermissible pair of 
$C^{\infty}$-divisors and 
$\{(U_i,f_i)\}_{i\in I}$ a local data
for $s-s'$. The intersection set of the pair 
$(s,s')$ is the following set
\begin{align}
s'\cap s\deq \set{x\in S \mid  df_i(x)\in 
\tform^{1}_{\mathbb{Z},S,x} 
\text{ for some } i\in I}.
\end{align}
\end{definition}

The set $s'\cap s$ is independent of the choice of 
a local data for $s-s'$. 

\subsection{Stratified torus fibration}
\label{section-stratified-torus-fibration}
We will explain the set $s'\cap s$ is the 
intersection of two `Lagrangian sections' of
a stratified torus fibration in this subsection.
We can skip this subsection for following 
our main theorem.

Let $S$ be a rational polyhedral space.
We can define a naive analog of Lagrangian torus 
fibration 
$\bigcup_{x\in S}(\tform^{1}_{\mathbb{R},S})_x/
(\tform^{1}_{\mathbb{Z},S})_x \to S$
for $S$ (as a set). Instead of this fibration, 
we define (the underlying set of) 
another torus fibration as follows. 
\begin{definition}[{Stratified torus fibration}]
\label{definition-stratified-torus-fibration}
Let $S$ be a compact rational polyhedral space.
Let $\check{X}(S)_x$ be 
the real torus defined as follows:
\begin{align}
\check{X}_0(S)_x\deq  
(\tform_{\mathbb{Z},S,x}^{1}/\opn{span}_{{\mathbb{R}}}
(\opn{SS}(S)_x)\cap \tform_{\mathbb{Z},S,x}^{1})\otimes_{\Z} {\mathbb{R}}/\Z.
\end{align}
We call the union $\check{X}_0(S)\deq \bigcup_{x\in S}
\check{X}_0(S)_x$
the underlying set of the \emph{stratified torus 
fibration} of $S$.
\end{definition}
We won't pursue about
topological structures of $\check{X}_0(S)$ 
in this paper.

\begin{example}
{\setlength{\leftmargini}{22pt}
\begin{enumerate}
\item If $B$ is an integral affine manifold, then
$\check{X}_0(B)$ is equal to $\check{X}(B)$ as a set.
\item 
Let $C$ be a compact trivalent tropical curve.
The stratified torus fibration $\check{X}_0(C)$ is equal to 
the stratified space associated
with a trivalent metric graph defined in \cite{auroux2022lagrangian}
as a set.
\end{enumerate}
}
\end{example}

Let $\check{f}_S\colon \check{X}_0(S)\to S$ be the 
canonical projection onto $S$.
We recall the following elementary proposition 
in order to define `Lagrangian sections' for $\check{f}_S$.

\begin{proposition}

Let $M$ (resp. $N$) be a $\mathbb{Z}$-module and
$i\colon M'\to M$ (resp. $j\colon N'\to N$) be the 
inclusion map of submodules, and 
$\opn{pr}\colon M\otimes_{\mathbb{Z}}N \to M/M'\otimes_{\mathbb{Z}} N/N'$
is the canonical projection map induced from the following
commutative diagram:
\begin{equation}
\begin{tikzcd}
M\otimes_{\mathbb{Z}}N \arrow[d] \arrow[r] & M/M'\otimes_{\mathbb{Z}}N \arrow[d] \\
M\otimes_{\mathbb{Z}}N/N' \arrow[r]        & M/M'\otimes_{\mathbb{Z}}N/N'.       
\end{tikzcd}
\end{equation}
Then, 
\begin{align}
\opn{pr}^{-1}(0)=\opn{Im}(i\otimes_{\mathbb{Z}}\opn{id}_N)+
\opn{Im}(\opn{id}_M\otimes_{\mathbb{Z}}j).
\end{align}
\end{proposition}

From the above proposition,
the canonical projection
$\opn{pr}_{S,x}\colon 
(\tform_{\mathbb{R},S}^{1})_x\simeq 
(\tform_{\mathbb{Z},S}^{1})_x\otimes_{\mathbb{Z}}\mathbb{R}
\to \check{X}_0(S)_x$ satisfies 
\begin{align}
\label{equation-torus-fiber}
\opn{pr}_{S,x}^{-1}(0)
=\opn{span}_{\mathbb{R}}(\opn{SS}(S)_x)+
(\tform^{1}_{\mathbb{Z},S})_x
\end{align}
for all $x\in S$.

\begin{definition}
Let $s$ be a $C^{\infty}$-divisor 
on a (boundaryless) rational polyhedral space $S$ 
and $\{(U_i,f_i)\}_{i\in I}$ a local data for $s$. 
The \emph{stratified Lagrangian section} associated with $s$ is the following
section of $\check{f}_S$:
\begin{align}
\check{s}\colon S \to \check{X}_0(S); x \mapsto \opn{pr}_{S,x} (df_i(x)).
\end{align}
\end{definition}

The $\check{s}$ is independent of the choice 
of local data for $s$ since 
$(\tform^{1}_{\mathbb{R},S})_x/(\tform^{1}_{\mathbb{Z},S})_x$
is naturally isomorphic to
$(\tform^{1}_{\mathbb{Z},S})_x
\otimes_{\mathbb{Z}}\mathbb{R}/\mathbb{Z}$.

From the prepermissibility condition and
\eqref{equation-torus-fiber}, we get the following proposition.

\begin{proposition}
Let $(s',s)$ be a prepermissible pair of
$C^{\infty}$-divisors
on a (boundaryless) rational polyhedral space, 
$\{(U_i,f_i)\}_{i\in I}$ a local data 
for $s-s'$ and $L_s\deq \opn{Im}(\check{s})$.
Then,
\begin{align}
\check{f}_S(L_{s'}\cap L_s)=s'\cap s.
\end{align}

\end{proposition}

\subsection{Cohomological local Morse data for permissible divisors}

In this subsection, we
define the graded module of a given prepermissible pair 
$(s',s)$ of $C^{\infty}$-divisors such that 
$\sharp (s'\cap s)<\infty$.
In order to define the graded module
$\opn{LMD}^{\bullet}(s',s)$, we define
intersection data of 
the prepermissible pair $(s',s)$.

\begin{definition}[{intersection data}]
\label{definition-intersection-data}
Fix a (boundaryless) 
rational polyhedral space $S$.
Let $s$ be a prepermissible $C^{\infty}$-divisor 
on $S$ such that $\sharp (s_0\cap s)<\infty$. 
An \emph{intersection data} for $s$ is the set 
$\{(U_p,f_p)\}_{p\in s_0\cap s}$ of pairs of open subsets
$U_i$ of $S$ and (0,0)-superforms
$f_p\colon U_p \to \mathbb{R}$ indexed by 
$ s_0\cap s$ 
such that
\begin{enumerate}
\item $s|_{U_p}=(f_p)^{\mathrm{sm}}$
for all $p\in s_0\cap s$,
\item $\opn{Crit}(f_p)=\{p\}$ for all $p\in s_0\cap s$,
\item $f_p(p)=0$ for all $p\in s_0\cap s$.
\end{enumerate}
An intersection data for a prepermissible pair 
$(s',s)$ of $C^{\infty}$-divisors such that 
$\sharp (s'\cap s)<\infty$ is 
an intersection data for $s-s'$.
\end{definition}

\begin{definition}[{Cohomological local Morse data for 
$C^{\infty}$-divisors}]
\label{definition-local-morse-data-divisor}
Let $S$ be a (boundaryless) rational polyhedral space 
and $(s',s)$ be a prepermissible pair of
$C^{\infty}$-divisors such that
$s'\cap s$ is finite and 
$\{(U_p,f_p)\}_{p\in s'\cap s}$ 
be an intersection data for $s$.
We set
\begin{align}
\opn{LMD}^{\bullet}_{A}(s',s)\deq 
\bigoplus_{p\in s' \cap s} 
\opn{LMD}^{\bullet}(A_{U_p},f_{p},p), \quad 
\opn{LMD}^{\bullet}(s',s)
\deq \opn{LMD}^{\bullet}_{\mathbb{Z}}(s',s).
\end{align}
When $s'=s_0$, we also define 
\begin{align}
\opn{LMD}^{\bullet}_A(S;s)
\deq \opn{LMD}^{\bullet}_{A}(s_0,s), \quad
\opn{LMD}^{\bullet}(S;s)\deq
\opn{LMD}^{\bullet}_{\mathbb{Z}}(S;s).
\end{align}
\end{definition}

\begin{remark}
\label{remark-differential-graded-module}
We stress that we should not consider 
$\opn{LMD}^{\bullet}_{A}(S;s)$ as the cohomology of $s$,
but just
a graded module 
since we need to define the differentials 
$\mathfrak{m}_1$ of 
$\opn{LMD}^{\bullet}_{A}(S;s)$ as an analog of
Floer cohomology.
We don't have, however,
the definition of the true differential of 
$\opn{LMD}^{\bullet}_A(S;s)$ in general, except when 
$\opn{LMD}^{\bullet}_{A}(S;s)$ is equal to a Floer complex
of Lagrangian submanifolds
as graded modules (e.g. \cite[5.2]{MR1882331}).
\end{remark}

It is important that 
$\opn{LMD}^{\bullet}_{A}(S;s)\simeq 
\opn{LMD}^{\bullet}(A_{\coprod_{p\in s_0\cap s} U_p},
\coprod_{p\in s_0\cap s} f_p)$ is
for a given intersection data 
$\{(U_p,f_p)\}_{p\in s_0\cap s}$ for $s$, and thus
we can use topological tools for computations of 
$\opn{LMD}^{\bullet}_{A}(S;s)$.

\begin{example}
\label{example-tropical-kodaira-vanishing}
Let $S$ be a boundaryless compact tropical manifold 
of pure dimension $n$, 
$s$ a prepermissible $C^{\infty}$-divisor on $S$ satisfying 
$\sharp (s_0\cap s)<\infty$,
and $\{(U_p,f_p)\}_{p\in s_0\cap s}$
an intersection data for $s$. 
If $\{-f_p<-f_p(p)\}=U_p\setminus\{p\}$
for all $p\in s_0\cap s$, then
\begin{align}
\label{equation-kodaira-vanishing}
\opn{LMD}^{\bullet}_{A}(S;s)\simeq 
\bigoplus_{p\in s_0 \cap s} A[0], \quad
\opn{LMD}^{\bullet}_{A}(S;-s)\simeq 
\bigoplus_{p\in s_0\cap -s} 
H^{n-1}(S\setminus\{p\};A)[-n].
\end{align}
In this case, the `cohomology' of 
$\opn{LMD}^{\bullet}_{A}(S;s)$ is independent of the choice of 
differential of $\opn{LMD}^{\bullet}_{A}(S;s)$.

Such a $C^{\infty}$-divisor naturally appears when
we consider an analog of positive forms
on complex manifolds. A \emph{Hessian form} on
an (integral) affine manifold is a certain analog
of a K\"ahler form, and has been studied
in Hessian geometry (see 
\cref{definition-hessian-metric,example-hessian-metric-from}).
Lagerberg \cite{MR2892935} defined a tropical analog of
positive $(p,p)$-forms and currents in $\mathbb{C}^{n}$.

As like the above previous research, we can also define
the positivity condition for $C^{\infty}$-divisors
tentatively. A $C^{\infty}$-divisor $s$ on
a boundaryless rational polyhedral space $S$ is 
\emph{positive} if there exists
a coordinate system 
$\{(U_i,\psi_i\colon U_i\to \Omega_i 
(\subset \mathbb{R}^{n_i}))\}_{i\in I}$ and 
a family $\{f_i\}_{i\in I}$
of strongly convex $C^{\infty}$-functions on a convex
open subset of $\mathbb{R}^{n_i}$ such that 
$\{(U_i,\psi_i^{*}f_i)\}_{i\in I}$ is a local data
for $s$.

If $s$ is a positive and prepermissible
$C^{\infty}$-divisor on a rational polyhedral space $S$
such that $s_0\cap s$ is finite, then 
$s$ satisfies \eqref{equation-kodaira-vanishing}.
This is a tropical analog of the Kodaira vanishing theorem
for a (positive power of) positive line bundle
$\mathcal{L}$ on a complex manifold $M$: 
\begin{align}
H^{\bullet}(M;\mcal{L}^{-1})\simeq H^{n}(M;\mcal{L}^{-1})[-n].
\end{align}
Positive $C^{\infty}$-divisors on integral affine 
manifolds are also related
with \\
Bohr--Sommerfeld points on integral affine manifolds
(see \cref{sec: BSRR}).
\end{example}

The following is the \emph{permissibility condition} for 
$C^{\infty}$-divisors.

\begin{condition}[{Permissibility condition}]
\label{condition-good-divisor}
$s$ is a prepermissible $C^{\infty}$-divisor on 
a (boundaryless) rational polyhedral space 
$S$ such that $s_0\cap s$ is finite and 
$\opn{LMD}^{\bullet}(S;s)$ is 
finitely generated.
\end{condition}

With this, all preparations which are necessary 
to formulate \cref{conjecture-tropical-MRR-preface} 
have been completed when $S$ is boundaryless.

The following proposition is an analog of \'{e}tale multiplicity of Euler characteristics
(cf. \cite[Proposition 1.1.28]{MR2095471}).
Let's see \cref{definition-etale-covering} for 
the definition of tropical \'etale covering 
map in this paper.
\begin{proposition}
\label{proposition-euler-number-etale}
Let $S$ be a compact (boundaryless) rational polyhedral 
space and $s$ be a permissible 
$C^{\infty}$-divisor.
For a tropical \'etale covering 
map $q \colon S' \to S$ having
the finite topological degree 
$\opn{deg}_{\opn{top}}(q)$,
\begin{align}
\chi(\opn{LMD}^{\bullet}(S';q^{*}(s)))=
\opn{deg}_{\opn{top}}(q)\chi(\opn{LMD}^{\bullet}(S;s)).
\end{align}
\end{proposition}

The proof of \cref{proposition-euler-number-etale} is trivial from definition.
Besides, we can apply the Meyer--Vietoris exact sequence 
for $\opn{LMD}^{\bullet}(S;s)$ under a suitable condition.
We will use this feature for tropical curves
(\cref{proposition-gluing-formula}).

\begin{remark}
Obviously, $\opn{LMD}^{\bullet}(S;s)$ is not usually 
finitely generated and this is not invariant 
for the divisor class $[s]$ of $s$
if $S$ is not compact.
This corresponds that Morse homology does not work directly
for non-compact manifolds unlike for compact manifolds.
Besides, $\chi(\opn{LMD}^{\bullet}(S;s))$ is \emph{not}
an invariant for line bundles for general compact rational 
polyhedral spaces. See also 
\cref{remark-rotation-closed-interval}.  
\end{remark}

\subsection{\texorpdfstring{$C^{\infty}$}{C-infty}-divisors 
for general rational polyhedral spaces}
\label{section-weakly-smooth-functions}

In this section, we will define
$C^{\infty}$-divisors 
for general rational polyhedral spaces 
and finish preparations which is necessary to formulate
\cref{conjecture-tropical-MRR-preface} for any 
compact tropical manifolds.
We recall \cref{condition-good-divisor}
and the definition of $(0,0)$-superforms
on a tropical affine space $\underline{\mathbb{R}}^{n}$
\cite[Definition 2.4]{MR3903579}.
As explained in \cite[Example 2.6]{MR3903579},
every $(0,0)$-superform on $\underline{\mathbb{R}}$ 
 is constant at $-\infty$.
For instance, a continuous function
$f\colon \underline{\mathbb{R}} \to \mathbb{R}; x\mapsto 
\opn{log}(1+e^{x})$ is not a $(0,0)$-superform
on $\underline{\mathbb{R}}$.
This is contrast with \cref{condition-good-divisor}.
In order to overcome this point, we define 
the sheaf $\mathcal{A}^{\mathrm{weak}}_S$ of 
weakly-smooth functions for rational polyhedral 
spaces $S$.

\begin{definition} \label{definition-weakly-smooth}
Let $\opn{exp}\colon \underline{\mathbb{R}}^{l}\to 
\mathbb{R}_{\geq 0}^{l}$ be the coordinate-wise 
exponential map of 
$\underline{\mathbb{R}}^{l}$ and $\Omega$
an open subset of a rational polyhedral subset 
in $\underline{\mathbb{R}}^{l}$.
A continuous function $f_0\colon \Omega\to \mathbb{R}$
is a \emph{primitively weakly-smooth} if 
$f_0=\opn{log}\circ h\circ \opn{exp}$ for some 
smooth function $h\colon W\to \mathbb{R}_{>0}$ on 
an open subset $W$ of $\mathbb{R}^{l} 
(\supset \mathbb{R}_{\geq 0}^{l})$:
\[\begin{tikzcd}
	W & {\mathbb{R}_{> 0}} \\
	\Omega & {\mathbb{R}}
	\arrow["\log", from=1-2, to=2-2]
	\arrow["f_0"', from=2-1, to=2-2]
	\arrow["h", from=1-1, to=1-2]
	\arrow["{\operatorname{exp}}", from=2-1, to=1-1].
\end{tikzcd}\]
Let $f$ be a continuous function on a rational polyhedral 
space $S$ and $x\in S$ such that 
there exists a chart 
$\psi\colon U \simto \Omega$ satisfying $f=f_0 \circ \psi$.
The data $(x,U,\psi,W,h)$ is 
called a \emph{primitive data} of $f$ at $x$.
\end{definition}

\begin{example}
Let 
$g\in \mathbb{R}_{\geq 0}[x_1,\ldots,x_n]$ 
be a polynomial 
function with nonnegative coefficients such that $g(0)\ne 0$. 
Then, $f=\opn{log}\circ g \circ \opn{exp}$ is 
primitively weakly-smooth 
on $\underline{\mathbb{R}}^{n}$.
\end{example}

\begin{definition}
Let $f$ is a continuous function on a rational polyhedral 
space $S$.
The function $f$ is \emph{weakly-smooth} at $x$ if 
there exists a chart 
$\psi \colon U\simto \Omega (\subset \underline{\mathbb{R}}^{l})$ 
of $S$ and a primitively weakly-smooth 
function $f_0$ such that $x\in U$ and $f=f_0\circ \psi$.

\end{definition}

\begin{definition}
Let $S$ be a rational polyhedral space and 
$U$ an open subset of $S$:
\begin{align}
\label{equation-weakly-smooth-definition}
\mathcal{A}_S^{\mathrm{weak}}(U)
\deq \{ f\in C^0(U)\mid
f \text{ is weakly-smooth at } x 
\text{ for all } x\in U\}.
\end{align}
The sheaf $\mathcal{A}_S^{\mathrm{weak}}$ 
of weakly-smooth functions on $S$ is the sheaf defined from 
\eqref{equation-weakly-smooth-definition}. 
\end{definition}

\begin{proposition}
\label{proposition-sheaf-weak-smooth}
The $\mathcal{A}_S^{\mathrm{weak}}$ is the sheaf of Abelian groups.
\end{proposition}

\begin{proof}
The condition is local, and thus $\mathcal{A}_S^{\mathrm{weak}}$ 
is a subsheaf of the sheaf $\mathcal{C}^{0}_S$ of 
continuous functions on $S$. From now on, we prove 
$\mathcal{A}^{\mathrm{weak}}_S$ is Abelian.
Fix an open subset $U$ of $S$, a point $x$ in $U$ 
and $f_1,f_2\in \mathcal{A}_S^{\mathrm{weak}}(U)$.
Each $f_i$ ($i=1,2$) has the primitive data
$(x,U_i,\psi_i,W_i,h_i)$ of
weakly-smooth function $f$ at $x$
respectively. We can assume $U_1=U_2$.
The charts 
$\psi_1\colon U_1 \to \Omega_1$ and 
$\psi_2\colon U_2\to \Omega_2$ induce
the diagonal morphism
$\psi\colon 
U_1\cap U_2\to \Omega_1\times \Omega_2; x\mapsto 
(\psi_1(x),\psi_2(x))$.
The $\psi$ is also a chart.
Let $\pi_1\colon W_1\times W_2\to W_1$ 
and $\pi_2\colon W_1\times W_2\to W_2$
the canonical projections.
Then, we get the following commutative diagram:  
\[\begin{tikzcd}
	{W_1\times W_2} & {\mathbb{R}_{> 0}} \\
	{\psi(U_1\cap U_2)} & {\mathbb{R}}
	\arrow["\log", from=1-2, to=2-2]
	\arrow["f_1+f_2"', from=2-1, to=2-2]
	\arrow["{\operatorname{exp}}", from=2-1, to=1-1]
	\arrow["{\pi_1^{*}h_1\cdot\pi_2^{*}h_2}", from=1-1, to=1-2]
\end{tikzcd}\]
The continuous map
$\pi_1^*h_1\cdot \pi_2^{*}h_2$ is also smooth
on $W_1\times W_2$.

Therefore, the data
$(x,U_1\cap U_2,\psi,W_1\times W_2,
\pi_1^*h_1\cdot \pi_2^{*}h_2)$ forms a primitive data
of $f_1+f_2$ at $x$. Hence, 
the $f_1+f_2$ is also weakly-smooth at $x$ for all
$x\in U$. Of course, $-f_1$ is also weakly-smooth
at $x$ for all $x\in U$.
\end{proof}

\begin{remark}
Since $\opn{log}\colon \mathbb{R}_{>0}\simeq \mathbb{R}$ is 
smooth, we can also prove $\mathcal{A}^{\mathrm{weak}}_S$ is 
the sheaf of $\mathbb{R}$-algebras by the same way
as \cref{proposition-sheaf-weak-smooth}. 
\end{remark}

From definition of the sheaf of $(0,0)$-superforms, 
$\mathcal{A}_S^{(0,0)}=\mathcal{A}_S^{\mathrm{weak}}$
if $S$ is boundaryless. 
As like $\mathcal{A}_S^{(0,0)}$, 
$\mathcal{A}_S^{\mathrm{weak}}$ is also fine
(cf. \cite[Lemma 2.7, Proposition 2.26]{MR3903579}).
There exists a canonical injective morphism 
$\mathcal{O}_S^{\times}\to \mathcal{A}_S^{\mathrm{weak}}$ 
since every 
affine map $x\mapsto \langle u,x \rangle+b$ 
with integer slope on $\mathbb{R}^{n}$
is equal to $\opn{log} \circ h\circ \opn{exp}$ for 
some Laurent monomial $h$ over the semifield 
$\mathbb{R}_{\geq 0}$ of nonnegative real numbers, 
see also \eqref{equation-log-polynomial}.

\begin{definition}
\label{definition-C-infinity-divisor}
A \emph{$C^{\infty}$-divisor} on 
a rational polyhedral space $S$ is an 
element of $H^{0}(S;\mathcal{A}_S^{\opn{weak}}/
\mathcal{O}^{\times}_S)$. 
A \emph{principal Cartier divisor} of $f$ is 
the value $(f)^{\mathrm{sm}}$ 
of the canonical morphism
$(\cdot)^{\mathrm{sm}}\colon 
H^{0}(S;\mathcal{A}_S^{\opn{weak}})\to 
H^{0}(S;\mathcal{A}_S^{\opn{weak}}/
\mathcal{O}^{\times}_S)$.
\end{definition}

We write 
$\CDiv(S)\deq 
H^{0}(S;\mathcal{A}_S^{\opn{weak}}/
\mathcal{O}^{\times}_S)$ for simplicity.
We can define a \emph{local data} for 
$s$ like when $S$ is boundaryless naturally
(\cref{definition-Cartier-local-data}).
Moreover, we can extend the prepermissibility condition for
weakly-smooth functions and $C^{\infty}$-divisors on
general rational polyhedral spaces by the following lemma.  

\begin{lemma}
Let 
$(x,U,\psi_1,W_1,h_1)$ and $(x,U,\psi_2,W_2,h_2)$
be primitive data of a weakly-smooth function $f$ at $x$.
Then, there exists an open neighborhood $V_0$ at the 
origin in $T_x S$ and the following natural diagram of 
morphisms:
\begin{equation}
\begin{tikzcd}
	& {\psi_1(U)} & {\opn{exp}^{-1}(W_1)} \\
	U & {V_0\cap\operatorname{LC}_xS} & {V_0} & {\mathbb{R}} \\
	& {\psi_2(U)} & {\opn{exp}^{-1}(W_2)}
	\arrow[hook, from=2-3, to=1-3]
	\arrow[hook', from=2-3, to=3-3]
	\arrow[hook, from=3-2, to=3-3]
	\arrow[hook, from=1-2, to=1-3]
	\arrow[hook, from=2-2, to=1-2]
	\arrow[hook', from=2-2, to=3-2]
	\arrow[hook, from=2-2, to=2-3]
	\arrow["{\opn{log}\circ h_1\circ \opn{exp}}", from=1-3, to=2-4]
	\arrow["{\opn{log}\circ h_2\circ \opn{exp}}"', from=3-3, to=2-4]
	\arrow[hook, from=2-2, to=2-1]
	\arrow["{\psi_{1}}"', from=2-1, to=3-2]
	\arrow["{\psi_0}", from=2-1, to=1-2],
\end{tikzcd}
\end{equation}
where $\phi \colon V_0\cap \opn{LC}_x S\hookto U$ is the 
canonical morphism induced from a morphism of germs 
in \cite[Proposition 2.5]{gross2019sheaftheoretic}.
In particular, the pullback $\phi^{*}f$ of $f$ is also 
weakly-smooth at $0$.
\end{lemma}
Every local cone $\opn{LC}_x S$ is boundaryless, and thus
we can define the prepermissible condition for any 
weakly-smooth functions.
The \emph{differential vector} 
$df(x)\in (\tform^{1}_{\mathbb{R},S})_x$ at $x$
is that of $\phi^{*}f$ via the identification 
$T_0(\opn{LC}_x S)=T_xS$.
We can also define the set $\opn{Crit}(f)$ of
\emph{critical points} of a weakly-smooth function $f$
on a rational polyhedral space naturally.
\begin{definition}[{Prepermissibility condition}]
\label{definition-prepermissible-weakly-smooth}
Let $f$ be a weakly-smooth function on a rational polyhedral 
space $U$. The $f$ is \emph{prepermissible} at $x\in U$ if 
the pullback $\phi^{*}f$ is prepermissible at 
$0\in \opn{LC}_x U$.
The weakly-smooth function $f$ is prepermissible if 
$f$ is prepermissible at $x$ for all $x\in U$.
Let $s\in H^{0}(S;\mathcal{A}^{\mathrm{weak}}_S/
\mathcal{O}^{\times}_S)$ be a $C^{\infty}$-divisor on a rational polyhedral space $S$.
The $s$ is \emph{prepermissible} if there exists 
a local data $\{(U_i,f_i)\}_{i\in I}$ for $s$
such that every $f_i$ is prepermissible.
\end{definition}

\begin{example}
\label{exmaple-weakly-smooth-line}
Let $f\colon \underline{\mathbb{R}}\to \mathbb{R}$ be a 
weakly-smooth function. The local cone 
$\opn{LC}_{-\infty} \underline{\mathbb{R}}$ 
at $-\infty$ is trivial, 
and thus every $f$ is prepermissible.
Fix a primitive data 
$(-\infty,[-\infty,\varepsilon),\opn{id},W,g)$ for $f$
at $-\infty$.
By the Jacobi rule, we have
\begin{align}
\frac{\partial f}{\partial x}(z)=(g\circ \opn{exp}(z))^{-1} 
\frac{\partial g}{\partial y}(\opn{exp}(z))\opn{exp}(z)
\end{align}
for all $z\in (-\infty,\varepsilon)$.
We can assume $\frac{\partial g}{\partial y}$ and 
$g$ are bounded on $(0,\opn{exp}(\varepsilon))$
for sufficiently small $\varepsilon$.
Therefore, $\lim_{z\to -\infty}\frac{\partial f}{\partial x}=0$.
\end{example}

As like the prepermissibility condition,
we can define a \emph{prepermissible pair} $(s',s)$
of $C^{\infty}$-divisors, the
set $s'\cap s$ of $(s',s)$ 
(\cref{definition-divisor-intersection}), 
an intersection data for $s$ 
(\cref{definition-intersection-data}), 
and cohomological local Morse data of $s$
(\cref{definition-local-morse-data-divisor}).
With them, all preparations which are necessary to 
formulate \cref{conjecture-tropical-MRR-preface}
have been completed for any compact 
tropical manifolds.

\section{Proof for compact tropical curves}
\label{section-tropical-curve}
In this section, we prove \cref{theorem-MRR-curve-preface}. 
See \cite[Definition 3.1]{mikhalkinTropicalCurvesTheir2008a} 
for the definition of tropical curves.
\cref{theorem-MRR-curve-preface} and
our proof of it are highly effected from 
\cite{MR2676658,knill2012graph,auroux2022lagrangian}.
Some propositions which have already proved as a special case
of the previous section will appear, but we write short
proofs of them 
for simplicity.

\subsection{Proof for \cref{theorem-MRR-curve-preface}}
The goal of this subsection is to prove 
\cref{theorem-MRR-curve-preface}.

Let $\Gamma_{n}$ be the projectivization of 
the tropical module defined in 
\cite[\textsection 2.2]{mikhalkinTropicalCurvesTheir2008a}.
Each $\Gamma_{n}$ has an embedding 
into the $(n-1)$-dimensional tropical projective space
$\mathbb{T}P^{n-1}\deq \mathbb{P}(\mathbb{T}^{n})$.
The $\mathbb{T}P^{n-1}$ has a canonical inclusion
$\iota_n\colon \mathbb{R}^{n-1}\to 
\mathbb{T}P^{n-1}$ 
of a tropical algebraic torus.
There exists the following isomorphism as rational 
polyhedral spaces:
\begin{align}
\Gamma_{n}^{\circ}\deq \Gamma_{n}\cap \iota_n(\mathbb{R}^{n-1}) \simeq \mathbb{R}_{\geq 0}(1,0,\ldots,0)\cup \cdots
\cup {\mathbb{R}}_{\geq 0}(0,\ldots,0,1)\cup 
{\mathbb{R}}_{\geq 0}(-1,\ldots,-1).
\end{align}
In particular, $\Gamma_{n}$ is isomorphic to the closure of
$\Gamma_{n}^{\circ}$ in $\mathbb{T}P^{n-1}$.
Every point in a tropical curve has an open neighborhood 
which is isomorphic to an open subset of 
$\Gamma_n$ for some $n\in \mathbb{Z}_{\geq 1}$.
Let $C$ be a tropical curve,
and $C^{\circ}$ the complement
of the set of $1$-valent vertices in $C$. 
Then, every point in $C^{\circ}$ has 
an open neighborhood which is isomorphic to 
an open subset of some $\Gamma_{n}^{\circ}$.
Besides, every compact tropical curve $C$ 
has a canonical metric structure $d_{C^{\circ}}$
on $C^{\circ}$
associated with the lattice structure of the tangent 
vector space of $C$ 
\cite[Proposition 3.6]{mikhalkinTropicalCurvesTheir2008a}.

Next, we see the prepermissibility condition 
for $C^{\infty}$-divisors on tropical curves.

\begin{proposition}
\label{proposition-prepermissible-curve}
Let $|\Sigma|$ be the support of a pure $1$-dimensonal 
finite convex polyhedral fan $\Sigma$ such that
$|\Sigma|$ spans $\mathbb{R}^{n}$.
Let $f\colon |\Sigma|\to \mathbb{R}$ be a 
$(0,0)$-superform on $|\Sigma|$. Then, 
$f$ is prepermissible at $x=(0,\ldots,0)$ if and only if 
$df(x)\in (\tform^{1}_{\mathbb{Z},|\Sigma|})_x\simeq (\mathbb{Z}^{n})^{\vee}$
or $\opn{lineal}(|\Sigma|,x)\simeq \mathbb{R}$.
\end{proposition}
\begin{proof}
Obviously, 
$\opn{dim}_{\mathbb{R}}\opn{lineal}(|\Sigma|,x)\leq 1$.
If $\opn{dim}_{\mathbb{R}}\opn{lineal}(|\Sigma |,x)=1$, 
then $|\Sigma | \simeq \mathbb{R}$. 
From now on, we assume $\dim_{\mathbb{R}}\opn{lineal}(|\Sigma|,x)=0$.
Fix a vector $v_{\sigma}$ for $\sigma\in \Sigma \setminus\{0\}$ 
such that $\opn{span}_{\mathbb{R}}(v_{\sigma})=\sigma$.
We can choose a basis $\{v_{\sigma_i}\}_{i=1,\ldots, n}$ of
$\mathbb{R}^{n}$
from $\{v_{\sigma}\}_{\sigma \in \Sigma}$,
and a linear form
$h=\sum_{i=1}^{n} 
a_{\sigma_i} v^{*}_{\sigma_i}$
such that $a_{\sigma_i}<0$ for any 
$i=1,\ldots,n$.
Then, the lower set $|\Sigma|\cap \{h< 0\}$ is disconnected
when $n\geq 2$. If $n=1$, $|\Sigma|\cap \{-h<0\}$ is 
empty.
Thus, 
$\opn{span}_{\mathbb{R}}(\opn{SS}((\mathbb{Z}_{X})_S)_x)
=\mathbb{R}^{n}$ where $X=\mathbb{R}^{n}$ and $S=|\Sigma|$.
\end{proof}

\cref{proposition-n-valent,proposition-1-valent}  
classify 
the type of intersection data of 
$C^{\infty}$-divisors which appear
in tropical curves.  

\begin{proposition}[{cf. \cite{knill2012graph}}]
\label{proposition-n-valent}
Let $|\Sigma|$ be the support of a pure $1$-dimensonal 
finite convex polyhedral fan $\Sigma$ such that
$|\Sigma|$ spans $\mathbb{R}^{n}$ and
$s_{p,q}\colon |\Sigma|\to \mathbb{R}$
a prepermissible $(0,0)$-superform on $|\Sigma|$ 
such that $\opn{Crit}(s_{p,q})=\{x\}$ where
$x=(0,\ldots,0)$. We set
\begin{align}
p\deq\sharp(\pi_0(\{s_{p,q}>0\})), \quad 
q\deq \sharp(\pi_0(\{s_{p,q}<0\})).
\end{align}

Then, 
\begin{align}
\opn{ind}_x(s_{p,q})=1-q ,\quad \opn{ind}_x(-s_{p,q})=1-p.
\end{align}

\end{proposition}

\begin{proof}
From definition, $s_{p,q}$ is a monotone function on 
$|\Sigma|\setminus\{x\}$.
\end{proof}

\begin{example}
When $p+q=3$, we have
\begin{align}
\opn{ind}_x(s_{3,0})=1, \quad 
\opn{ind}_x(s_{2,1})=0, \quad \opn{ind}_x(s_{1,2})=-1, 
\quad \opn{ind}_x(s_{0,3})=-2.
\end{align}
\end{example}

\begin{remark}
\label{remark-nondeprave-curve}
Every $|\Sigma|$ of \cref{proposition-n-valent}
has a coarse polyhedral cone structure 
$\mathscr{C}$. 
For this structure, $s_{p,q}$ has a unique critical 
point at the origin with respect to 
$\mathscr{C}$ in the theory of stratified 
Morse theory \cite[Part I.2.1]{MR932724}
but does not satisfy the nondegeneracy condition.
In particular, $s_{p,q}$ is not nondepraved 
at the origin \cite[Part I.2.3. Definition (c)]{MR932724}, 
and thus we cannot use the theory of local Morse 
data for nondepraved functions directly 
\cite[Part I. Definition 3.5.2]{MR932724}. 
We expect, however, that there exists a 
good method to connect our approach with
stratified Morse theory.
\end{remark}

\begin{proposition}
\label{proposition-1-valent}
Let $U\deq [-\infty,a) (\subset \underline{\mathbb{R}})$ be a connected open neighborhood of 
$-\infty$ and $f\colon U\to \mathbb{R}$ a weakly-smooth
function on $U$ such that $\opn{Crit}(f)=\{-\infty\}$.
Then, $\{f< 0\}$ is equal to $\emptyset$ or 
$U\setminus\{-\infty\}$.
\end{proposition}

\begin{remark}
By using \cref{proposition-1-valent}, we can generalize
\cref{proposition-n-valent} for any point in 
a pure $1$-dimensonal rational polyhedral space.
\end{remark}

Let $C$ be a pure $1$-dimensional
rational polyhedral space and 
$C_{\reg}$ the set of generic points in $C$
(see \cite[\textsection 4.B]{MR3894860} and
\cite[Definition 2.7]{gross2019sheaftheoretic}).
In the rest, we assume every $C$ satisfies
the following condition for simplicity 
(see also \cref{remark-classification-curve}):
\begin{condition}
\label{condition-simplicity-curve}
By a \emph{$1$-dimensional rational polyhedral space},
we mean a pure $1$-dimensional rational polyhedral
space $C$ such that every connected component
of $C_{\reg}$ is isomorphic to 
an open interval of $\mathbb{R}$ or a 
circle $\mathbb{R}/c\mathbb{Z}$ ($c\in \mathbb{R}_{>0}$) 
with the standard structure induced from $\mathbb{R}$. 
In particular,
\begin{align}
  C_{\reg}\simeq \left( \coprod_{e\in E} (a_e,b_e)\right)\amalg
\left(\coprod_{l\in L} \mathbb{R}/c_l\mathbb{Z} \right), \quad
(a_e\in \underline{\mathbb{R}}, b_e\in 
\overline{\mathbb{R}}, c_l\in \mathbb{R}_{>0})
\end{align}
for some indexed set $E$ and $L$.
\end{condition}

\begin{remark}
\label{remark-classification-curve}
To be honest, \cref{condition-simplicity-curve}
is unnecessary. This remark explains that.
Let $C$ be a pure $1$-dimensional
rational polyhedral space.
Then, $C_{\reg}$ is also a rational polyhedral space,
and isomorphic to the disjoint union of 
$1$-dimensional integral affine manifolds
(see \cite[\textsection 2.1]{MR2181810}
for three different definition of
integral affine manifolds).
Since $C_{\reg}$ is paracompact,
every connected component of $C_{\reg}$
is diffeomorphic to an open interval or a circle.
Every connected $n$-dimensional
integral affine manifold $B$ has 
the developing map $\opn{dev}_{B}\colon\tilde{B}\to \mathbb{R}^{n}$
where $\tilde{B}$ is the universal cover of $B$
(see e.g. \cite[\textsection 2.3]{goldmanRadianceObstructionParallel1984a}).
Besides, every developing map is an immersion and a submersion, 
and thus $\opn{dev}_{B}$ is an embedding onto 
an open subset of $\mathbb{R}$ when $\dim B=1$.
\end{remark}

Since $C_{\reg}$ is an integral affine manifold, 
$C_{\reg}$ has a standard Lagrangian torus fibration
$
\check{f}_{C_{\reg}}
\colon \check{X}(C_{\reg})\to C_{\reg}
$ 
and a dual torus fibration 
$f_{C_{\reg}}\colon X(C_{\reg})\to C_{\reg}$
(see \cref{definition-integral-affine-manifold}
and \cref{definition-SYZ-torus-fibration}).

\begin{definition}[{cf. \cref{definition-stratified-torus-fibration}}]
\label{definition-continuous-section}
Let $C$ be a tropical curve or a closed interval in 
$\mathbb{T}P^{1}$.
We set the following spaces:
\begin{align}
\check{X}(C_{\reg})\deq 
T^{*}C_{\reg}/\mathcal{T}_{\mathbb{Z},C_{\reg}}^{\vee}, \quad 
\check{X}_0(C)\deq 
\check{X}(C_{\reg})\sqcup_{i,\check{s}_0}C,
\end{align}
where $\check{X}(C_{\reg})\sqcup_{i,\check{s}_0}C$ is the pushout of 
the inclusion $i\colon C_{\reg} \to C$ and 
$\check{s}_0\colon C_{\reg}\to \check{X}(C_{\reg})$ is the zero 
section of $\check{f}_{C_{\reg}}\colon \check{X}(C_{\reg})
\to C_{\reg}$.
\end{definition}

Obviously, $\check{X}_0(C)$ has a canonical projection
$\check{f}_C \colon \check{X}_0(C)\to C$.
By the universal property of the pushout,
$\check{f}_C$ is continuous.

\begin{example}
Let $C=[a,b]$ be a closed interval in $\mathbb{T}P^{1}$, then 
$\check{X}(C_{\reg})$ is just a cylinder
and $\check{X}(C)$ has a canonical embedding 
into 
$[a,b]\times \mathbb{R}^{\vee}/\mathbb{Z}^{\vee}$.
\begin{align}
\check{X}_0([a,b])\simeq 
\{(a,0)\}\cup((a,b)\times 
\mathbb{R}^{\vee}/\mathbb{Z}^{\vee})
\cup \{(b,0)\}.
\end{align}
The negation map
$\iota \colon\mathbb{T}P^{1}\to \mathbb{T}P^{1}; x\mapsto -x$
induces the canonical automorphism induced from the pullback of cotangent bundles:
\begin{align}
\iota^{*}\colon \check{X}_0(\mathbb{T}P^{1})
\to \check{X}_0(\mathbb{T}P^{1}); (x;y)\mapsto (-x;-y). 
\end{align}
This is a symplectic automorphism on 
$\check{X}(\mathbb{R}) (\subset \check{X}_0(\mathbb{T}P^{1}))$ since 
$dx\wedge dy=d(-x)\wedge d(-y)$.
This automorphism can be visualized as the 180-degree rotation of the 
cylinder.
\end{example}

The following proposition is a repetitive work of
\cref{section-stratified-torus-fibration}.

\begin{proposition}[{continuous sections}]
\label{proposition-prepermissible-section}
Let $s$ be a prepermissible $C^{\infty}$-divisor
$s=\{(U_i,f_i)\}_{i\in I}$ on a tropical curve 
$C$ or a closed interval of $\mathbb{T}P^{1}$. 
Then, 
{\setlength{\leftmargini}{22pt}
\begin{enumerate}
\item There exists the following continuous map:
\begin{align}
\label{equation-continuous-section}
\check{s}\colon C\to \check{X}_0 (C); x
\mapsto (x,\opn{pr}_{C,x}(df_{i}(x))).
\end{align}
\item If $\check{s}_0\colon C\to \check{X}_0 (C)$ be the 
zero section, then $s_0\cap s=\check{f}_C(\check{s}_0\cap \check{s})$.
\end{enumerate}
}
\end{proposition}
\begin{proof}
(1): 
Choose $(U_i,f_i)$ and $(U_j,f_j)$ such that
$U_i\cap U_j\neq \emp$. 
Then, $df_i(x)-df_j(x)\in 
(\tform_{\mathbb{Z},C}^{1})_x$.
Therefore, the construction of $\check{s}$ is well-defined.
Obviously, $\check{s}$ is continuous by gluing lemma 
for each edge of $C$.

(2): It follows directly from \eqref{equation-continuous-section}.
\end{proof}

By \cref{proposition-prepermissible-section},
$s_0\cap s$ is compact when $C$ is compact.

Let $C$ be a tropical curve and 
$\mathcal{A}_C^{\opn{weak}}$ be the sheaf 
of weakly-smooth functions on $C$. From definition,
$\mathcal{A}_C^{0,0}\subset \mathcal{A}_C^{\opn{weak}}$
and $\mathcal{A}_C^{\opn{weak}}$ is also a fine sheaf.
Since $\mathcal{A}_C^{\opn{weak}}$ is fine, 
the connecting homomorphism 
$[\cdot]\colon H^{0}(C;\mathcal{A}_C^{\opn{weak}}/
\mathcal{O}^{\times}_C)\to \opn{Pic}(C);s\mapsto [s]$ is 
surjective.

From now on, we will prove
\cref{conjecture-enough-prepermissible}
for compact tropical curves
(\cref{proposition-prepermissible-section}).
Geometrically, this is clear,
but in order to solve it,
some preparations are necessary.

\begin{definition}
An \emph{open interval} in
a $1$-dimensional rational polyhedral space $C$ is
an open subset of $C_{\opn{reg}}$ which is homeomorphic
to an open interval.
\end{definition}

\begin{proposition}
\label{proposition-bump-interval}
Let $U$ be a finite union of open intervals in a 
compact tropical curve $C$.
Then, there exists a prepermissible nonnegative weakly-smooth function
$f\colon C\to \mathbb{R}_{\geq 0}$ satisfying
\begin{enumerate}
\item $\{f>0\}\subset U$
\item $U\cap(\bigcup_{m\in \mathbb{Z}}\{f'=m\})$ is finite.
\end{enumerate}
\end{proposition}
\begin{proof}
We may assume $C$ is connected.
If $U$ is isomorphic to a bounded open interval in
$\mathbb{R}$,
then we can create such an $f$ easily by using 
the following bump function
$f_{a,b}\colon\mathbb{R}\to \mathbb{R}$
\begin{align}
f_{a,b}(x)\deq 
\begin{cases}
\opn{exp}\left(-\frac{1}{a^{2}-(x-b)^{2}}\right),
\quad \text{ if } (|x-b|<a), \\
0,\quad \text{ o.w. }
\end{cases}
\end{align}
for all $a\in \mathbb{R}_{>0}$ and $b\in \mathbb{R}$.
Then, $f=f_{a,b}$ is a $C^{\infty}$-function on such that
$\{f>0\}=(b-a,b+a)$ and has a unique maximal point.
Of course, $f_{a,b}'$ is bounded.
We can extend $f|_U$ to $C$ by zero outside $U$
by the pullback of $f$ by an appropriate linear function 
on $\Gamma_{n}^{\circ}$
(see also \cref{equation-pullback-support}).

Suppose $U=\mathbb{R}_{>0}$.
Then, 
$f=f_{\frac{1}{2},\frac{1}{2}}\circ \opn{exp}
\colon \underline{\mathbb{R}}\to \mathbb{R}$ is a
nonnegative weakly-smooth function which is strictly monotone 
except at $x=-\opn{log}2$ and $\{f>0\}=(-\infty,0)$.
By \cref{exmaple-weakly-smooth-line}, $f$ is also
bounded. In this case, we can also extend 
$f|_U$ to $C$ by zero outside $U$.

If $U=\mathbb{R}$, then $C=\mathbb{T}P^{1}$.
$f\colon \mathbb{T}P^{1}\to\mathbb{R};x\mapsto 
f_{\frac{1}{2},\frac{1}{2}}\circ \opn{exp}(x)
+f_{\frac{1}{2},\frac{1}{2}}\circ \opn{exp}(-x)$
satisfies the conditions $(1)$ and $(2)$.

By combining with the above three cases, we can 
construct $f$ satisfying the conditions
$(1)$ and $(2)$ for every $U$.
\end{proof}

\begin{proposition}
\label{proposition-permissible-divisor}
Let $s$ be a $C^{\infty}$-divisor on a compact
tropical curve $C$.
Then, there exists
a prepermissible $C^{\infty}$-divisor $s'$ which is linearly
equivalent to $s$ and
$s_0\cap s'$ is a finite set. 
\end{proposition}
\begin{proof}
We may assume $C$ is connected.
Fix a local data $\{(U_i,f_i)\}_{i\in I}$
for $s$. We can assume $I$ is finite
since $C$ is compact.
Let $K_v$ be a compact and connected neighborhood 
of $v (\in C\setminus C_{\reg})$ such that
$K_{v}\subset U_i$ for some $i\in I$ and 
$K_v\cap K_{v'}=\emptyset$ when $v\ne v'$ for
all $v,v'\in C\setminus C_{\reg}$.
Let $h\colon C\to \mathbb{R}$ be a weakly-smooth function
such that $h|_{K_v}=f_i|_{K_v}$.
Such an $h$ exists since $\mathcal{A}^{\mathrm{weak}}_C$
is fine. 
The set $\{(U_i,h|_{U_i})\}_{i\in I}$ is a local data for 
$(h)^{\mathrm{sm}}$.
From construction,
$s'\deq s-(h)^{\mathrm{sm}}$ is linearly-equivalent 
to $s$, prepermissible 
and has a local data 
$\{(U_i,\tilde{f}_i)\}_{i\in I}$ for $s'$
satisfying $\tilde{f}_i|_{K_v}$ is constant for all 
$i\in I$ and $v\in C\setminus C_{\reg}$. 

From now on, we assume $s$ is prepermissible.
We may assume $s\ne s_0$.
By \cref{proposition-prepermissible-section},
$s_0\cap s$ is closed and compact.
Therefore, $(s_0\cap s)\cap C_{\reg}$ is 
a finite union of closed subset of $C_{\reg}$.

If $s$ is not trivial, then each connected
component of the interior 
$\opn{int}((s_0\cap s)\cap C_{\reg})$ is 
isomorphic to an open interval.
Let $U\deq \opn{int}((s_0\cap s)\cap C_{\reg})$ and 
$f$ be a weakly-smooth function on $C$
satisfies the conditions of 
\cref{proposition-bump-interval}.
Then, $s'=s+(f)^{\mathrm{sm}}$ is also
prepermissible and $\sharp(s_0\cap s')<\infty$.
\end{proof}

Let $C$ be a $1$-dimensional rational polyhedral space.
$\check{X}_0(C_{\reg})$ has a canonical symplectic structure,
and thus there exists a canonical orientation induced
from the symplectic form on $\check{X}(C_{\reg})$.

\tikzset{every picture/.style={line width=0.75pt}} 

\begin{figure}
  \centering
  \begin{tikzpicture}[x=0.75pt,y=0.75pt,yscale=-1,xscale=1]

    \draw   (220,160.5) .. controls (220,138.13) and (230.07,120) .. (242.5,120) .. controls (254.93,120) and (265,138.13) .. (265,160.5) .. controls (265,182.87) and (254.93,201) .. (242.5,201) .. controls (230.07,201) and (220,182.87) .. (220,160.5) -- cycle ;
    \draw   (385,160.5) .. controls (385,138.13) and (395.07,120) .. (407.5,120) .. controls (419.93,120) and (430,138.13) .. (430,160.5) .. controls (430,182.87) and (419.93,201) .. (407.5,201) .. controls (395.07,201) and (385,182.87) .. (385,160.5) -- cycle ;
    \draw    (242.5,120) -- (407.5,120) ;
    \draw    (242.5,201) -- (407.5,201) ;
    \draw  [dash pattern={on 15pt off 1.5pt}]  (220,160.5) -- (385,160.5) ;
    \draw    (220,160.5) .. controls (263,97.5) and (348,224.5) .. (385,160.5) ;
  \end{tikzpicture}

\caption{A continuous section $s\colon [0,1]\to [0,1]\times \mathbb{R}/\mathbb{Z}$}
\end{figure}

As like \cite{auroux2022lagrangian},
we also use the notion of the \emph{rotation number} of 
$C^{\infty}$-divisors for the proof of 
\cref{theorem-MRR-curve-preface}. 
From now on, we start to prepare for defining
the rotation number of a permissible
$C^{\infty}$-divisor on a $1$-dimensional rational
polyhedral space.

We recall there exists a canonical 
embedding $\check{X}_0([a,b])\to [a,b]\times 
\mathbb{R}^{\vee}/\mathbb{Z}^{\vee}$ for 
any closed interval $[a,b]$ in $\mathbb{T}P^1$.
We usually identify $\check{X}_0([a,b])$ with 
the image of the above embedding.
\begin{definition}
\label{definition-rotation-interval}
Fix a closed interval $[a,b]$ in $\mathbb{T}P^{1}$.
Let $s\colon [a,b]
\to [a,b]\times \mathbb{R}^{\vee}/\mathbb{Z}^{\vee}$ be 
a continuous section of the canonical projection
$p\colon [a,b]\times \mathbb{R}^{\vee}/\mathbb{Z}^{\vee} \to [a,b]$
such that $s(a)=(a,0)$ and $s(b)=(b,0)$. 
The rotation number $\opn{rot}(s)$
is the homotopy class of the closed path
$q\circ s\colon [a,b]\to \mathbb{R}^{\vee}/\mathbb{Z}^{\vee}$
induced from $s$ and the canonical projection
$q\colon [a,b]\times \mathbb{R}^{\vee}/\mathbb{Z}^{\vee}
\to \mathbb{R}^{\vee}/\mathbb{Z}^{\vee}$.
Here, 
we identify $\pi_1(\mathbb{R}^{\vee}/\mathbb{Z}^{\vee},0)\simeq \mathbb{Z}$
by the canonical orientation of
$\mathbb{R}^{\vee}/\mathbb{Z}^{\vee}$.
\end{definition}

The canonical orientation on 
$[a,b]\times \mathbb{R}^{\vee}/\mathbb{Z}^{\vee}$
is compatible with the symplectic orientation 
of $\check{X}((a,b))$.

\begin{lemma}
\label{lemma-rotation-slope}
Let $s$ be a prepermissible $C^{\infty}$-divisor on 
a closed interval $[a,b]$ in $\mathbb{T}P^{1}$ 
and
$f\colon (a,b)\to \mathbb{R}$ be 
a weakly-smooth function
such that $s|_{(a,b)}=(f)^{\mathrm{sm}}$.
Then, 
\begin{align}
\label{equation-degree-rot-line}
\lim_{x\to b-0}f'(x)-\lim_{x\to a+0} f'(x)
=\opn{rot}(\check{s}).
\end{align}

\end{lemma}
\begin{proof}
From \cref{exmaple-weakly-smooth-line}
and $s|_{(a,b)}=(f)^{\mathrm{sm}}$,
the RHS of \cref{equation-degree-rot-line} is an
integer.
A continuous function
$g\colon [a,b]\to \mathbb{R}^{\vee};z\mapsto 
\lim_{x\to z-0}f'(x)-\lim_{x\to a+0} f'(x)$
is just a lift of the
rotation map $q\circ\check{s}$ defined in 
\cref{definition-rotation-interval}.
Therefore, we can get our assertion from definition 
of fundamental group. 
\end{proof}

\begin{definition}
Let $s$ be a prepermissible $C^{\infty}$-divisor on 
a closed interval $[a,b]$ in $\mathbb{T}P^{1}$.
We define $\opn{rot}_0(s)\deq\opn{rot}(\check{s})$.
\end{definition}

\cref{proposition-rotation-number}
and \cref{remark-rotation-slope} 
show that the rotation number of 
a permissible $C^{\infty}$-divisor is
independent of the choice of embeddings
of closed intervals into $\mathbb{T}P^{1}$. 

\begin{proposition} \label{proposition-rotation-number}
Let 
$\phi \colon [a,b] \to [a,b]$ be an automorphism of 
a closed interval in $\mathbb{T}P^{1}$, 
$s\colon [a,b]\times \mathbb{R}^{\vee}/\mathbb{Z}^{\vee}$
a section of the canonical projection 
$q\colon [a,b]\times \mathbb{R}^{\vee}/\mathbb{Z}^{\vee} \to [a,b]$ and 
$\varphi \colon [a,b]\times \mathbb{R}^{\vee}/\mathbb{Z}^{\vee} \to 
[a,b]\times \mathbb{R}^{\vee}/\mathbb{Z}^{\vee}$ 
the symplectic automorphism induced from $\phi$.
Then, $\opn{rot}(\varphi\circ s \circ \phi^{-1})=\opn{rot}(s)$.
\end{proposition}
\begin{proof}
Since every automorphism of 
$[a,b]$ is a composition of the negation map 
with a translation map, we only need to consider the case
when $\phi$ is the negation map.
In this case, $\varphi(x,y)=(b+a-x,-y)$ for 
$(x,y)\in [a,b]\times \mathbb{R}^{\vee}/\mathbb{Z}^{\vee}$ if 
$a,b\notin \{-\infty,\infty\}$.
Since $s(x)=(x,q\circ s(x))$ for all $x\in [a,b]$,
we have
\begin{align}
q\circ \psi \circ s \circ \phi^{-1}(x)=
-q\circ s(\phi^{-1}(x)))=-q\circ s(b+a-x).
\end{align}
From definition of fundamental group, 
$\opn{rot}(\varphi\circ s \circ \phi^{-1})$
should be $\opn{rot}(s)$.
The same method works for $[a,b]=\mathbb{T}P^{1}$.
\end{proof}

\begin{remark}
\label{remark-rotation-slope}
Let $s$ be
a prepermissible $C^{\infty}$-divisors
on a closed interval $[a,b]$ in $\mathbb{T}P^{1}$,
$f\colon (a,b)\to \mathbb{R}$
a permissible weakly-smooth function such that 
$s|_{(a,b)}=(f)^{\mathrm{sm}}$
and $\phi\colon [a,b]\to [a,b]\colon x\mapsto b+a-x$.
Then,
\begin{align}
\lim_{x->b-0}\phi^{*}(f)'(x)-\lim_{x->a+0}\phi^{*}(f)'(x)
=-\lim_{x->a+0}f'(x)+\lim_{x->b-0}f'(x)
\end{align}
Therefore, $\opn{rot}_0(s)=
\opn{rot}_0(\phi^{*}s)$.
This is a special case of \cref{proposition-rotation-number}.
\end{remark}

From \cref{proposition-prepermissible-section} and 
\cref{proposition-rotation-number}
(or \cref{remark-rotation-slope}), we can define the 
rotation number for every permissible
$C^{\infty}$-divisor (\cref{{definition-rotation-number}}).

\begin{definition}
Let $C$ be a $1$-dimensional compact rational polyhedral space
and $s$ a prepermissible $C^{\infty}$-divisor on $C$.
We set
\begin{align}
I(C,s)\deq \{
e\in\pi_0(C\setminus (s_0\cap s))\mid e 
\text{ is an open interval}\}.
\end{align}
\end{definition}

\begin{definition}[{cf. \cite{auroux2022lagrangian}}]
\label{definition-rotation-number}
Let $s$ be a prepermissible $C^{\infty}$-divisor 
and a $1$-dimensional compact rational polyhedral space $C$ 
and 
$\check{s}$ its associated Lagrangian section.
The \emph{rotation number} of $s$ is 
\begin{align}
\opn{rot}(s)\deq \sum_{e\in I(C,s)}
\opn{rot}_0(\psi_{\bar{e}}^{*}(s)),
\end{align}
where  
$\psi_{\bar{e}}\colon [a_e,b_e]\to C$ is 
a continuous map such that the restriction of $\psi_{\bar{e}}$ 
on $(a_e,b_e)$ is an isomorphism onto $e$.
\end{definition}

The rotation number for prepermissible 
$C^{\infty}$-divisors is equal to 
the degree of the divisor class $[s]$ of 
prepermissible $C^{\infty}$-divisor when
$C$ is a compact tropical curve
(\cref{proposition-rotation-degree}). 

From now on, we start the proof of 
\cref{theorem-MRR-curve-preface}.

\begin{proposition}
\label{proposition-trivial-case-curve}
Let $s$ be a prepermissible $C^{\infty}$-divisor 
and a $1$-dimensional compact rational polyhedral space $C$
such that $s_0\cap s=\emptyset$. Then, 
\begin{align}
  \opn{LMD}^{\bullet}(C;s)=\chi_{\opn{top}}(C)=0.
\end{align}
\end{proposition}
\begin{proof}
The equation $\opn{LMD}^{\bullet}(C;s)=0$ is trivial.
From definition, $C\setminus C_{\reg} \subset s_0\cap s=\emptyset$.
Then, $C$ must be the disjoint union of
tropical elliptic curves. 
\end{proof}

\begin{proposition}
\label{proposition-simple-interval-rr}
Let $C=[a,b]$ be a closed interval in $\mathbb{T}P^{1}$
and $s$ a prepermissible $C^{\infty}$-divisor on 
$C$
such that $s_0\cap s=\{a,b\}$. Then,
\begin{align}
\chi(\opn{LMD}^{\bullet}(C;s))=\opn{rot}(s)+
\chi_{\opn{top}}(C).
\end{align}
\end{proposition}
\begin{proof}
Since there are a finite number of types of 
intersection data and 
$\opn{rot}(s)=-1,0,1$,
we can directly compute it using
\cref{proposition-n-valent} and 
\ref{proposition-1-valent} directly.
\end{proof}

Let $C$ be a $1$-dimensional rational polyhedral space,
$C'$ be a rational
polyhedral subspace of $C$ and $i_{C'}$
the inclusion map of $C'$. Then, the pullback of 
a weakly-smooth function is also weakly smooth from
definition, so there exists a morphism 
$\iota_{C'}^{-1}\mathcal{A}^{\mathrm{weak}}_C\to 
\mathcal{A}^{\mathrm{weak}}_{C'}$.
This also induces the pullback 
$\iota_{C'}^{*}\colon \CDiv(C')\to \CDiv(C)$
(see also \cite[Proposition 3.15]{gross2019sheaftheoretic}).
A good point of the rotation number
$\opn{rot}(s)$ and the 
Euler characteristic of 
$\opn{LMD}^{\bullet}(C;s)$ is that
they have the
following Meyer--Vietoris type lemma.

\begin{proposition}
\label{proposition-gluing-formula}
Let $s$ be a permissible $C^{\infty}$-divisor
on a $1$-dimensional compact rational space $C$.
Fix two closed $1$-dimensional compact rational
polyhedral
subspaces $C'$, $C''$ of $C$ such that
$C' \cup C''=C$ and $C'\cap C''\subset s_0 \cap s$.
Then, 
the inclusion $i_{C'}\colon C'\to C$ 
and $i_{C''}\colon C''\to C$ induces
\begin{align}
\label{equation-meyer-vietoris}
\chi(\opn{LMD}^{\bullet}(C;s))
&=\chi(\opn{LMD}^{\bullet}(C';i_{C'}^{*}(s)))
+\chi(\opn{LMD}^{\bullet}(C'';i_{C''}^{*}(s)))
-\sharp(C'\cap C''), \\
\label{equation-rotation-sum}
\opn{rot}(s)&=\opn{rot}(i_{C'}^{*}(s))
+\opn{rot}(i_{C''}^{*}(s)).
\end{align}
\end{proposition}
\begin{proof}
From assumption, $i^{*}_{C'}(s)$ and 
$i^{*}_{C''}(s)$ are prepermissible.
The \eqref{equation-rotation-sum} is true from
the definition of rotation number. 
 Besides,
\begin{align}
s_0\cap i^{*}_{C'}(s)=i^{-1}_{C'}(s_0\cap s), \, 
s_0\cap i^{*}_{C''}(s)=i^{-1}_{C''}(s_0\cap s), \,
s_0\cap i^{*}_{C'\cap C''}(s)=
i^{-1}_{C'\cap C''}(s_0\cap s).
\end{align}
By \cref{proposition-n-valent} 
and the Meyer--Vietoris sequence, we get 
the \eqref{equation-meyer-vietoris}.
\end{proof}

We can generalize
\cref{proposition-simple-interval-rr} by
\cref{proposition-gluing-formula} as follows.

\begin{proposition}
\label{proposition-MRR-1-dim-poly-space}
Let $C$ be a $1$-dimensional compact rational 
polyhedral space and
$s$ a permissible $C^{\infty}$-divisor on $C$.
Then,
\begin{align}
\label{equation-rotation-rr}
\chi(\opn{LMD}^{\bullet}(C;s))=\opn{rot}(s)+
\chi_{\opn{top}}(C).
\end{align}
\end{proposition}

\begin{proof}
Without loss of generality, we can assume $C$ is 
connected and $\sharp (s_0\cap s)\geq 1$ 
by \cref{proposition-trivial-case-curve}.
Since $s_0\cap s\supset (C\setminus C_{\reg})$ 
and $C$ is connected, $C\setminus (s_0\cap s)$ has no 
connected components which 
are isomorphic to tropical elliptic curves.
The set $s_0\cap s$ induces a finite cell decomposition
$\mathscr{C}$ of 
$C$.
Let $\mathscr{C}_1=\{e_{\alpha}\}_{\alpha \in \pi_0(C\setminus (s_0\cap s))}$
be the set of (closed) $1$-cells of $\mathscr{C}$. 
Fix a total ordering $<$ on $\pi_0(C\setminus (s_0\cap s))$.
By gluing $1$-cells in that order, we get a filtering;
\begin{align}
\emptyset=C_0\subsetneqq C_1\subsetneqq \cdots
\subsetneqq C_m \subsetneqq C_{m+1}=C,
\end{align}
where $C_{i+1}\deq C_{i}\cup e_{i+1}$
($e_i\in \mathscr{C}_1$)
for $i=0,\ldots,m$.
Let $i_{C'}\colon C'\to C$ be the inclusion map
of a cell subcomplex $C'$ of $C$.
By \cref{proposition-simple-interval-rr,proposition-gluing-formula},
for $i=0,\ldots,m$ we have
\begin{align}
\label{equation-1st}
\chi(\opn{LMD}^{\bullet}(C_1;i_{C_1}^{*}(s)))
=& \opn{rot}(i_{C_1}^{*}(s))+\chi_{\opn{top}}(C_1), \\
\chi(\opn{LMD}^{\bullet}(C_{i+1};i_{C_{i+1}}^{*}(s)))
=& \chi(\opn{LMD}^{\bullet}(C_i;i_{C_i}^{*}(s)))
+ \chi(\opn{LMD}^{\bullet}(e_{i+1};i_{e_{i+1}}^{*}(s)))\\
-&\chi_{\opn{top}}(C_i\cap e_{i+1}), \\
\opn{rot}(i_{C_{i+1}}^{*}(s))
=& \opn{rot}(i_{C_{i}}^{*}(s))
+ \opn{rot}(i_{e_{i+1}}^{*}(s)).
\end{align}

If $C_{j}$ ($j\in \{0,\ldots,m\}$)
satisfies \eqref{equation-rotation-rr}, then

\begin{align}
\chi(\opn{LMD}^{\bullet}(C_{j+1};i_{C_{j+1}}^{*}(s)))
=&\opn{rot}(i_{C_{j}}^{*}(s))+\opn{rot}(i_{e_{j+1}}^{*}(s))
+\chi_{\opn{top}}(C_j)+\chi_{\opn{top}}(e_{j+1}) \\
-&\chi_{\opn{top}}(C_j\cap e_{j+1}) \notag \\
=&\opn{rot}(i_{C_{j+1}}^{*}(s))
+\chi_{\opn{top}}(C_{j+1}). \label{equation-j-th}
\end{align}
\eqref{equation-1st} and \eqref{equation-j-th}
complete the proof of \eqref{equation-rotation-rr}.
\end{proof}

\begin{proposition}
\label{proposition-rotation-degree}
Let $s$ be a permissible $C^{\infty}$-divisor on
a compact tropical curve $C$. Then,
\begin{align}
\label{equation-rotation-number}
\opn{rot}(s)=\opn{deg}([s])(=\int_{C} c_1([s])).
\end{align}
\end{proposition}
For proving \cref{equation-rotation-number},
we use the following elementary lemma.
\begin{lemma}
\label{lemma-csoft-tropical-rational-function}
Let $(a,b)$ be an open interval in
$\mathbb{T}P^{1}$ and 
$l\colon \mathbb{R}\to \mathbb{R}$ be an 
affine function with an integer slope.
Then, there exists a tropical rational function
$f \colon\underline{\mathbb{R}}\to \mathbb{R}$
such that the support $\opn{supp}((f))$ of 
the principal divisor $(f)$ is in $(a,b)$, 
$\lim_{x\to a}f'(x)=0$, and $f(z)=l(z)$ for all 
$z\geq \opn{max}\{x\mid x\in\opn{supp}((f))\}$.
\end{lemma}
\begin{proof}
This is a special case of \cite[Lemma 4.4]{MR3894860}.
\end{proof}

\begin{proof}[{Proof of \cref{proposition-rotation-degree}}]
We recall that a local data $\{(U_i,f_i)\}_{i\in I}$ for $s$
defines an \v{C}ech $1$-cocycle 
$(f_{j}-f_i)_{i,j}
\in C^{1}(\{U_i\}_{i\in I},\mathcal{O}_C^{\times})$
whose cohomology class is equal to $[s]$
(see also \cite[Remark 7.2.1]{MR3560225}).
There exists a family 
$\{U_v\}_{v\in C\setminus C_{\reg}}$ of 
contractible open neighborhoods $U_v$ of $v$
such that $U_v\cap U_{v'}\ne \emptyset $ if $v\ne v'$.
So we may assume 
$\{U_i\}_{i\in I}=\{U_v\}_{v\in C\setminus C_{\reg}}
\cup\{e\}_{e\in\pi_0(C_{\reg})}$.
Then, we can find an ordinary Cartier divisor 
$D\in H^{0}(C;\mathcal{M}_C^{\times}/
\mathcal{O}_C^{\times})$
and a local data $\{(U_i,g_i)\}_{i\in I}$ for $D$
such that $(g_j-g_i)_{i,j}=(f_j-f_i)_{i,j}$ for any 
$i,j\in I$
and $\opn{supp}(D)\subset C_{\reg}$
by \cref{lemma-csoft-tropical-rational-function}.
From \cite[Theorem 4.15]{MR3894860}
(or \cite[Proposition 5.12]{gross2019sheaftheoretic}),
we can calculate the degree of $D$ by the sum of
the weights on $\opn{supp}(D)$,
which are determined by 
the divisor map $\opn{div}\colon 
H^{0}(C;\mathcal{M}_C^{\times}/
\mathcal{O}_C^{\times})\to \mathrm{Z}_{0}(C)$
\cite[Defintion 4.14]{MR3894860}.
By \cref{lemma-csoft-tropical-rational-function}
and \cref{lemma-rotation-slope}, we get
$\opn{deg}([s])=\opn{deg}(D)=\opn{rot}(s)$.
\end{proof}

\begin{remark}
\label{remark-rotation-closed-interval} 
The rotation number of a permissible
$C^{\infty}$-divisor on $1$-dimensional
compact rational polyhedral space is 
\emph{not} an invariant for its divisor class.
Let $C\deq [0,1]$ be a closed interval
and $f_n\colon C\to \mathbb{R};x \mapsto nx^{2}$.
Any $f_n$ is a prepermissible $(0,0)$-superform on $C$.
From the definition
of principal $C^{\infty}$-divisor,
$s_n\deq (f_{n})^{\mathrm{sm}}$ is always trivial,
but $\opn{rot}(\check{s}_n)=n$. 
A philosophical reason of this is
that a closed interval in $\mathbb{R}$
is a tropical analog of a closed annulus 
in $\mathbb{C}$ but not for
a complex projective line (cf. 
\cite[Definition 3.7]{MR3903579}).
In fact, $X(C_{\reg})$ has a standard complex structure
which is biholomorphic to
an open annulus in $\mathbb{C}$
(see \eqref{definition-SYZ-torus-fibration}
for the definition of $X(C_{\reg})$).
\end{remark}

\begin{proof}[{Proof of \cref{theorem-MRR-curve-preface}}]
We can prove it from \cref{proposition-MRR-1-dim-poly-space,proposition-rotation-degree}.
\end{proof}

\begin{remark}
\label{remark-non-prepermissible-divisor}
In this remark, we explain
the reason why we need the prepermissibility condition for 
$C^{\infty}$-divisors. 
One of reasons is that there exists 
a weakly-smooth function $f$
on a polyhedral space $S$ and 
$x\notin\opn{Crit}(f)$ such that
$\opn{LMD}^{\bullet}(A_S,f,x)$ is nontrivial.
We will explain about this problem by the following 
example:
Let 
$g$ be a weakly-smooth function on 
$\Gamma_3$
such that 
\begin{enumerate}
\item $dg=(\frac{1}{2},\frac{1}{2})$,
\item $\opn{Crit}(g)=\Gamma_3\setminus \Gamma_3^{\circ}$,
\item $\sum_{p\in \opn{Crit}(g)}\opn{ind}_p(g)=2$.
\end{enumerate}
$g$ is prepermissible except $x=(0,0)$.
If we generalize 
\cref{theorem-MRR-curve-preface} for 
a principal non-prepermissible
$C^{\infty}$-divisors $(g)^{\mathrm{sm}}$, 
we need to add the index $\opn{ind}_x(g)$ at $x$ satisfying
\begin{align}
\sum_{p\in \opn{Crit}(g)}\opn{ind}_p(g)+
\opn{ind}_x(g)=1=\chi_{\opn{top}}(\Gamma_3).
\end{align}
At least, \cref{theorem-MRR-curve-preface} is not
true without correction of the effects at 
non-prepermissible points
of $g$.
In this case, $\opn{ind}_x(g)$ should be $-1$.
This condition is independent of the choice
of $g$ satisfying the above condition.
On the other hand, the index of a given 
$C^{\infty}$-divisor should be 
independent of the choice of 
local data.
Since $\opn{ind}_x(g+m)$ depends on the choice
of monomials $m$, we need to choose a monomial $m_0$
with integer slope to 
define the index of $g$ at $x$ as a divisor.
In our knowledge, there is no canonical choice
of $m_0$ (up to constant) for $g$ as a principal divisor without
assumption prepermissibility.
Every prepermissible function $f$ has a canonical choice
of monomials $m$ which is one such that 
$x\in \opn{Crit}(f+m)$ (up to constant).
We could not use this condition for $g$ since 
$x\notin \opn{Crit}(g+m)$ for every monomial $m$ with
integer slope.
\end{remark}

\begin{remark}
Our definition of $C^{\infty}$-divisor 
has an analog of positive divisor by using 
the local data $\{(U_i,f_i)\}_{i\in I}$ as 
explained in \cref{example-tropical-kodaira-vanishing}.
However, this positive $C^{\infty}$-divisor is different 
from ample line bundles for tropical curves. 
If $s$ is a positive $C^{\infty}$-divisor, then
$\chi(\opn{LMD}(C;s))=\sharp(s_0\cap s)\geq \sharp 
(C\setminus C_{\reg})$. 
Thus, $\chi(\opn{LMD}(C;s))\geq 2g(C)-2$
when $C$ is a compact trivalent
graph. This means that
$\opn{deg}(s)\geq 3g(C)-3$.
On the other hand, every divisor $D$ satisfying
$\opn{deg}(D)>0$ is ample \cite[Corollary 43]{MR2892941}.
Roughly speaking, the positivity condition is 
stricter than the ampleness condition.
\end{remark}

\subsection{Relationships to previous results}
\label{section-tropical-curve-note}

\subsubsection{Relationships to
homological mirror symmetry of trivalent graphs}
\label{section-syz-trivalent-graph}

Our approach is mostly affected from 
\cite{auroux2022lagrangian}.
When $C$ is a compact trivalent metric graph, 
our $\check{X}_0(C)$ are 
almost the same with 
the trivalent configuration $M$ of 
$2$-spheres in 
\cite[\textsection 3.1]{auroux2022lagrangian}.
Here, the space $M$ is the
union of $\C P^{1}$'s whose intersection complex is $C$, 
i.e., every edge of $C$ corresponds to a
$\C P^{1}$ and every trivalent vertex of $C$ corresponds
to the intersection point of some three $\C P^{1}$'s
(see also \cite{MR3570582}).
The mirror manifold of $M$ is an algebraic curve $X_K$
defined from Mumford's construction for the dual intersection
complex defined from $C$.

The similarity of $\check{X}_0(C)$ and $M$ is evident 
since $M$ is a compactification of $\check{X}_0(C)$ 
as a topological space,
so below we will write the main differences
between  
we will list the differences between 
$X_0(C)$ and $M$ below:
{\setlength{\leftmargini}{22pt}
\begin{enumerate}
  \item The local model of $M$ comes from the critical locus
of the potential function of a (A-side) Landau--Ginzburg 
model $(\C^{3},-xyz)$.
  \item If a $C^{\infty}$-divisor $s$ on $C$ 
satisfies 
\cref{condition-good-divisor} and has an intersection 
data $\{(U_p,f_p)\}_{p\in s_0\cap s}$ such that
$f_p$ is strictly convex at all $p\in C\setminus C_{\reg}$, 
then \\
$\opn{LMD}_{\Lambda_{\mathrm{nov}}^{\mathbb{C}}}
(C;s)$ corresponds to the 
Floer complex \cite[(3.2)]{auroux2022lagrangian}
as a graded module, but 
the both are different if not. 
In particular, the graded modules associated with 
an intersection point $p\in s_0\cap s$ may not 
be a shift of $1$-dimensional vector space.
\item We don't know the differential $\mathfrak{m}_1$ 
for $\opn{LMD}_{\Lambda_{\mathrm{nov}}^{\mathbb{C}}}
(C;s)$ in general.
\end{enumerate}
}

Auroux--Efimov--Katzarkov  
expected a generalization of their results
shown in \cite{auroux2022lagrangian}
to projective hypersurfaces
\cite[\textsection 7]{auroux2022lagrangian}.
We hope our $\check{X}_0(S)$ for a 
rational polyhedral space $S$ give a hint for 
a generalization of their expectation.  
\subsubsection{Relationships to the Poincar\'e--Hopf 
theorem for graphs}

\cref{theorem-MRR-curve-preface} can be considered 
as a $C^{\infty}$-version of Knill's graph 
theoretical Poincar\'e--Hopf theorem \cite{knill2012graph}
or 
Banchoff's critical point theorem \cite{MR225327}
when the line bundle on a tropical curve is trivial.
We can easily see the index $i(v)$ defined in 
\cite[\textsection 3]{knill2012graph}
corresponds to $\opn{ind}_v(f)$ defined in 
\eqref{equation-local-index} (see also 
\cref{proposition-n-valent} and 
\cite[\textsection 7]{knill2012graph}).

Banchoff proved a polyhedral complex version of Poincar\'e--Hopf theorem
for height functions \cite[Theorem 1]{MR225327}.
The definition of the index $a(v,f)$ 
\cite[p.246]{MR225327} which 
he used for the theorem may seem different 
from our index but essentially equal to it
(see also \cite[p.143-144]{grunert2017piecewise} 
for more detail about it).

\subsubsection{Relationships to a localization of 
index on compact Riemann surfaces.}

We mention relationships between 
\cref{theorem-MRR-curve-preface} and
\cite[6]{MR2676658}.
For simplicity,
we assume $C$ is a compact trivalent tropical curve, i.e.,
a compact tropical curve which only has $3$-valent
vertex in $C\setminus C_{\reg}$.  
Then, $C_{\reg}$ has a canonical dual torus fibration 
$\pi_{C_{\reg}}\colon X(C_{\reg})\to C_{\reg}$ of
$\check{\pi}_{C_{\reg}} \colon \check{X}(C_{\reg})\to C_{\reg}$.
We can define $X_0(C)\deq C\sqcup_{i,s} X(C_{\reg})$ like
$\check{X}_0(C)$.
Besides, $C$ has a canonical lattice length metric $d_{C}$, and
thus we can define the following set for 
a sufficiently small
$\vep >0$: 
\begin{align}
C_{0,\vep}\deq \{p\in C\mid d_C(p,v)>\vep 
\text{ for all } v\in C\setminus C_{\reg}\}.
\end{align}

For a given $\vep$ (and a \emph{signed tropical structure}
on $C$ \cite{MR3076066}), 
$X_0(C)$ can be considered as a subset of a compact 
Riemann surface $X$ whose complex structure is 
compatible with $X(C_{0,\vep})$. 
Such an $X$ is given by gluing triangles for 
each open neighborhood of $p\in X_0(C)\setminus X(C_{\reg})$
appropriately.
Let $s$ be a permissible $C^{\infty}$-divisor on $C$.
The restriction $s|_{C_{\reg}}$ corresponds to a Lagrangian section of 
$\check{\pi}_{C_{\reg}}$ from 
\cref{proposition-cartier-lagrangian}.
By the semi-flat SYZ transformation 
(see \cite{MR1876073}
or \cite[\textsection 9.1]{MR1882331}),
$s|_ {C_{\reg}}$ induces a family of $U(1)$-holonomoy for
each torus fiber $\pi_{C_{\reg}}^{-1}(p)$ and defines a 
complex line bundle on $X(C_{0,\varepsilon})$. From now on, 
we suppose $s_0\cap s<\infty$. 
Let $\{(U_p, f_p)\}_{p\in s_0\cap s}$ be
an intersection data 
for $s$ such that 
$\{x\in C\mid d_C(p,x)\leq \vep\}\subset U_p$ for
all $p\in C\setminus C_{\reg}$.
Besides, we assume 
$(s\cap s_0)\setminus C_{0,\vep}=C\setminus C_{\reg}$ and $s$ satisfies the following cocycle
condition for each $p\in C\setminus C_{\reg}$:
\begin{align}
df_p(p_1)+df_p(p_2)+df_p(p_3)=0,\quad 
(\{p_1,p_2,p_3\}\deq \{x\in C\mid d_C(p,x)=\vep\}).
\end{align}

From the above condition, $0<|df_p(p_i)| \ll 1$ and
$f_p$ is of type $s_{2,1}$ or $s_{1,2}$ discussed
in \cref{proposition-n-valent}.
From the above condition, we can define 
a flat $U(1)$-bundle on each pants of a 
connected component of $X\setminus X(C_{0,\vep})$
as like \cite[6.1.3]{MR2676658}.
We remark the isomorphism
$(\tform_{\mathbb{Z},C}^{1})_p\otimes_{\mathbb{Z}}
\mathbb{R}/\mathbb{Z}\simeq 
H^1(\mathbb{P}^{1}\setminus 
\{0,1,\infty\};\mathbb{R}/\mathbb{Z})$
is important and no coincidence for this construction.
The $s_{2,1}$
corresponds a small pants and the $s_{1,2}$
corresponds a large pants in 
\cite[Definition 6.3]{MR2676658}.

Each small pants and large pants have a local
Riemann--Roch number defined the theory of 
localization of Dirac-type operator and its 
number is equal 
to our index by the above correspondence.
The following equation is a correspondence about 
the local index $\opn{ind}_v$ and the local 
Riemann--Roch number under the above 
identification \cite[Theorem 6.7]{MR2676658}.
\begin{align} \label{equation-local-RR}
& [BS^{+}]=\opn{ind}_v(s_{2,0})=1, 
&& [BS^{-}]=\opn{ind}_v(s_{0,2})=-1,\notag \\
& [D^{+}]=\opn{ind}_v(s_{1,0})=1,
&& [D^{-}]=\opn{ind}_v(s_{0,1})=0, \\
& [P^{S}]=\opn{ind}_v(s_{2,1})=0,
&& [P^{L}]=\opn{ind}_v(s_{1,2})=-1 \notag.
\end{align}
The origin of our approach firstly comes from the 
above correspondence.
We stress that the both methods of computation
of our index and the local Riemann--Roch number
defined in \cite{MR2676658} are different 
each other, even though both indexes have various 
similar features.
We can also consider the following question as 
like \cite{auroux2022lagrangian}.

\begin{question} \label{question-tropical-complex-rr}
Can we generalize the 
correspondence \eqref{equation-local-RR}
for more general the complement of
projective hyperplane arrangements and flat $U(1)$-bundles on it?
How about for pants decomposition of
  complex algebraic hypersurfaces in \cite{MR2079993} and line bundles?
\end{question}

We add two comments about the above question. 
First, the Bergman fan $\Sigma$ of a given complex 
(projective) hyperplane arrangement $\mathcal{A}$ does \emph{not} 
determine 
the homotopy type of the complement of (projective) 
hyperplane
arrangements $U(\mathcal{A})$ uniquely in general.
Second, every flat $U(1)$-bundle on 
$U(\mathcal{A})$ is, 
however, given
from a group homomorphism 
$\rho\colon T_{|\Sigma|,0}^{\mathbb{Z}}\to 
\mathbb{R}/\mathbb{Z}$ since 
$T_{|\Sigma|,0}^{\mathbb{Z}}\simeq 
H_1(U(\mathcal{A});\mathbb{Z})$ (see \cite[Theorem 4]{MR3153919}).
Therefore, we hope that we can get an affirmative answer 
for \cref{question-tropical-complex-rr}.

\section{Proof for integral affine manifold with Hessian form.}
\label{section-integral-affine-manifold}
In this section,
we recall about affine manifolds from
\cite{MR2293045,
goldmanRadianceObstructionParallel1984a,
MR2181810,
grossMirrorSymmetryLogarithmic2006a,
sepethesis}
and 
\cite[Chapter 6]{MR2567952}.

\begin{definition} \label{definition-integral-affine-manifold}
An $n$-dimensional \emph{integral affine 
(resp. strongly integral affine, affine)} manifold
is a pair of an $n$-dimensional differential manifold $B$ 
and an atlas $\{(U_i,\psi_i)\}_{i \in I}$ of $B$ such that 
$\psi_i \circ \psi_j^{-1}$ is a restriction of 
an element in $\opn{GL}_n(\mathbb{Z})\ltimes \mathbb{R}^{n}$
(resp. $\opn{GL}_n(\mathbb{Z})\ltimes \mathbb{Z}^{n}$,
$\opn{GL}_n(\mathbb{R})\ltimes \mathbb{R}^{n}$) for any $i,j\in I$.
\end{definition}

\begin{remark}
An (integral) affine manifold is a special case of 
$(G,X)$-manifolds \cite[3.3]{MR1435975}.
The terminology integral affine manifold has 
different terminologies for different authors. 
For instance, an integral (resp. strongly integral) affine manifold 
is called a tropical affine manifold (resp. integral affine manifold)
in \cite[Definition 1.22]{MR2722115}. 
The terminology \emph{strongly integral affine manifold} 
comes from \cite[Remark 5.10]{MR3079343}.
\end{remark}

We can consider a sheaf theoretic definition of 
integral affine manifold as explained 
in \cite[2.1]{MR2181810} 
and every integral affine manifold is a special case of 
tropical manifold. 
Every integral affine structure of $B$ induces a 
canonical local system $\mathcal{T}_{\mathbb{Z},B}$ 
of integer valued tangent vectors in $TB$ and the dual 
local system 
$\mathcal{T}_{\mathbb{Z},B}^{\vee}
(\simeq \tform_{\mathbb{Z},B}^{1})$ of it.

Let $(B,\{(U_i,\psi_i)\}_{i \in I})$ be an affine manifold.
Every  $U_i$ has a coordinate system $\{x_k\}_{k=1,\ldots,n}$ 
induced from a coordinate system formed 
by affine functions on $\mathbb{R}^{n}$.
We call this coordinate system as an \emph{affine coordinate system} 
of $U_i$. 

\begin{definition}
\label{definition-hessian-metric}
A \emph{Hessian metric} or \emph{metric of a Hessian form}
on an affine manifold 
$(B,\{(U_i,\psi_i)\}_{i \in I})$
is a Riemannian metric $g$ such that 
there exists a multivalued function $K$ on $B$ satisfying
$g=\sum_{k,l}\frac{\partial K}{\partial x_k \partial x_l}dx_k dx_l$ for 
an affine coordinate system of $U_i$. The triple
$(B,\{(U_i,\psi_i)\}_{i \in I},g)$ is called a \emph{Hessian manifold}.
\end{definition}
\begin{remark}
We follow the terminology Hessian metric and Hessian manifold 
from \cite{MR2293045}.
The terminology 
Hessian manifold is called
\emph{K\"ahler affine manifold}
in \cite{MR714338}, \emph{affine K\"ahler manifold}
or AK-manifold \cite{MR1882331} for short.
To our knowledge, the notion of Hessian metric have already
appeared in \cite[p.213]{MR413010}.
\end{remark}

A Hessian metric can be considered as a $(1,1)$-superform 
naturally and this superform is a certain analog of 
a K\"ahler form, so a Hessian metric itself is called a Hessian form.
We can see this the following pair of torus fibrations induced 
from the local system $\mathcal{T}_{B}$ on an integral 
affine manifold:

\begin{equation} \label{definition-SYZ-torus-fibration}
\begin{tikzcd}
T^{*}B/\mathcal{T}_{\Z , B}^{\vee} =:\hspace{-35pt} 
&\check{X}(B) \arrow[rd,"\check{f}_{B}"'] 
&   & X(B) \arrow[ld,"f_{B}"] 
&\hspace{-35pt} \deq TB/\mathcal{T}_{\Z, B} \\
& &  B & &
\end{tikzcd}.
\end{equation}

In some literature, the symbols $\check{X}$ and
$X$ are reversed. 
The manifold $\check{X}(B)$
has a canonical symplectic structure 
induced from $T^{*}B$, and each fiber of 
$\check{f}_{B}$ is a Lagrangian torus.
On the other hand, $X(B)$ has a canonical complex
structure induced from $TB$ and integral affine 
structure of $B$.
Since each fiber of the fibrations $f_{B}$ and 
$\check{f}_{B}$ are compact,
$f_{B}$ and $\check{f}_{B}$ are proper maps.
In particular, if $B$ is compact then 
$\check{X}(B)$ and $X(B)$ are compact.
If $B$ has a Hessian form $g$, 
then $\check{X}(B)$ and $X(B)$ have
the standard K\"ahler structure
induced from $g$ 
\cite[Proposition 6.14 and 6.15]{MR2567952}.
 
\begin{remark} \label{remark-compact-hessian}
We recall a well-known result for 
compact special affine manifolds admitting
Hessian forms which is mentioned 
in some literature
(e.g. \cite[\textsection 5.2]{MR1882331}).
By \cite[Theorem 2.1 and Corollary 2.3]{MR714338},
 every closed special Hessian manifold $(B,g)$ has 
a flat Riemannian and Hessian metric $\tilde{g}$ such 
that its volume form $\opn{vol}_{\tilde{g}}$ is parallel 
with respect to the
affine connection induced from the affine structure 
on $B$. Moreover, $\tilde{g}$ satisfies 
real Monge--Amp\`ere equation, and thus
the associated first and second Koszul 
form of $\opn{vol}_{\tilde{g}}$ is trivial 
\cite[Definition 3.1.2]{MR2293045}.
If $B$ is a special integral affine manifold with 
a Hessian metric 
$g=\sum_{k,l}\frac{\partial K}{\partial x_k \partial x_l}dx_k dx_l$, 
then the K\"ahler metric $\omega=2\sqrt{-1}\partial \bar{\partial}
(K\circ f_{B})$ on $X(B)$ is Ricci--flat if 
and only if $K$ satisfies real Monge--Amp\`ere equation
\cite[Proposition 6.14]{MR2567952}. 
By applying the proof of \cite[Theorem 8.3.3]{MR2293045}
for $\tilde{g}$,
the Levi--Civita connection of $\tilde{g}$ is 
equal to the affine connection of $B$ 
\cite[Corollary 8.3.7]{MR2293045}. 
By the Bieberbach theorem, every closed flat Riemannian manifold 
is covered by a Riemannian flat torus 
(see also \cite[Theorem 5.3]{MR862114}), and thus every closed 
integral Hessian manifold is an unramified finite cover of 
a tropical torus.
\end{remark}

\begin{remark}
The integral affine structure itself is a very strict condition.
\emph{Chern's conjecture} states 
every closed affine manifold $B$ has zero topological 
Euler characteristic.
If $B$ is a closed flat Riemannian manifold, Chern's conjecture is true from Chern--Gauss--Bonnet formula.
This conjecture is true for 
\emph{special affine manifold} \cite{MR3665000},
i.e., an affine manifold $(B, \{U_i,\psi_i\}_{i\in I})$ 
whose transition map is 
in $\opn{SL}_n(\mathbb{R})\ltimes \mathbb{R}^{n}$.
Therefore, every closed integral affine 
manifold satisfies Chern's conjecture 
since the orientable double cover of an integral affine manifold has a compatible special integral affine structure.
\myfootnote{We learned this conjecture 
from \cite{goldmanRadianceObstructionParallel1984a}
and \cite{MR3665000}.
}
\end{remark}

\begin{remark}[{Relationships with symplectic geometry}]
Integral affine manifolds naturally appear 
as base spaces of Lagrangian torus fibration 
\cite{duistermaatGlobalActionangleCoordinates1980a}. 
Besides, a Hessian manifold naturally appears as
the moduli space of special Lagrangian deformations of
a compact special Lagrangian submanifold of
a Calabi-Yau manifold by McLean's theorem 
\cite[Theorem 3.6 and Corollary 3.10]{MR1664890}
and Hitchin's theorem
\cite[Theorem 2]{MR1655530} (see also 
\cite[Chapter 6.1]{MR2567952}).
In \cref{appendix-geometric-quantization}, we write
about the relationships between integral affine manifolds 
and geometric prequantization. 
\end{remark}

\subsection{Tropical homology, Tropical superform and Cartier data}
From now on, we follow about tropical cohomology and superforms 
for tropical manifolds
from \cite{mikhalkinTropicalEigenwaveIntermediate2014a,
MR3903579,gross2019sheaftheoretic}.
In the case of affine manifolds, some parts of them has already
studied in the field of Hessian geometry by the language 
of differential geometry
(see e.g. ~\cite[Chapter 7]{MR2293045}).
Let $X$ be a rational polyhedral space and 
$\Omega_{\mathbb{Z},X}^{p}$ the sheaf
of tropical $p$-forms on $X$
~\cite[Definition 2.7]{gross2019sheaftheoretic}.
Let $\mform^{p}_{\mathbb{Z},X}$ be 
the sheaf defined in 
\cite[2.4]{mikhalkinTropicalEigenwaveIntermediate2014a}
where $p\in \mathbb{Z}_{\geq 0}$.
If $X$ is a tropical manifold, then 
$\tform^{p}_{\mathbb{Z},X}\simeq 
\mform_{\mathbb{Z},X}^{p}$ for all 
$p\in \mathbb{Z}_{\geq 0}$
(see \cite[Remark 2.8]{gross2019sheaftheoretic}).
In the rest of this paper,
we assume every rational polyhedral space $X$
has a local face structure and
satisfies $\Omega_{\mathbb{Z},X}^{p}\simeq \mform_{\mathbb{Z},X}^{p}$
for all $p$.
We set 
$\tform^{p}_{\mathbb{R},X}\deq 
\tform^{p}_{\mathbb{Z},X}\otimes_{\mathbb{Z}_X}\mathbb{R}_X$.

We recall some results about tropical superforms from
\cite{MR3903579,smacka2017differential}.
Let $\mathcal{A}^{p,q}_X$ be the sheaf of 
$(p,q)$-superforms on a rational polyhedral space
$X$ with a local face structure. 
The bigraded sheaf $\mathcal{A}_X^{\bullet,\bullet}$ has 
a canonical bigraded
complex structure
$(\mathcal{A}_X^{\bullet,\bullet},d',d'')$.
The sheaf $\tform^{p}_{\mathbb{R},X}$ has the following acyclic resolution
\cite[Corollary 3.18, Lemma 3.21]{MR3903579}:
\begin{align}
  0 \to \tform^{p}_{\mathbb{R},X} \to \mcal{A}^{p,0}_{X}\xto{d''} 
\mcal{A}^{p,1}_{X} \xto{d''}\cdots.
\end{align}

\begin{example}
\label{example-hessian-metric-from}
If $B$ is an integral affine manifold, then 
$\tform^{p}_{\mathbb{Z},B}\simeq 
\bigwedge^{p}_{i=1} \mathcal{T}_{\mathbb{Z},B}^{\vee}$ for 
all $p\in \mathbb{Z}_{\geq 0}$.
Every Riemannian metric $g$ on $B$ can be considered
as an element of $\mathcal{A}_{B}^{1,1}$.
The metric $g$ is $d''$-closed if and only 
if $g$ is a Hessian metric (\cite[Lemma 7.4.1]{MR2293045}).
\end{example}

The complex $(\tform_{\mathbb{R},X}^{\bullet},d')$
has a canonical dga structure induced from 
$(\mathcal{A}_X^{\bullet,\bullet},d',d'')$, 
and thus its hypercohomology 
$\mb{H}^{\bullet}(X;\tform_{\mathbb{R},X}^{\bullet})$ is a 
graded-commutative algebra.
This is a certain tropical analog of the singular cohomology
$H^{\bullet}(X;\C)$ for a complex manifold $X$ since 
the analytic de Rham theorem $\C_X \simeq \Omega_X^{\bullet}$ 
gives isomorphism 
$H^{\bullet}(X;\C)\simeq \mb{H}^{\bullet}(X;\Omega_X^{\bullet})$
of graded algebras. 

An elementary but remarkable fact of 
$(\tform_{\mathbb{R},X}^{\bullet},d')$ is that 
$\tform_{\mathbb{R},X}^{\bullet}\simeq 
\bigoplus_{i\in \Z}\tform_{\mathbb{R},X}^{i}[-i]$, i.e., the differential 
of $\tform_{\mathbb{R},X}^{\bullet}$ is trivial unlike the analytic de Rham complex
of complex manifolds, see \cite[Corollary 2.15]{smacka2017differential}.
Therefore, we can calculate the multiplication of 
$\mb{H}^{\bullet}(X;\tform_{\mathbb{R},X}^{\bullet})$ by 
the cup products of each $H^{q}(X;\tform_{\mathbb{R},X}^{p})$ easily.
We can also construct the subcomplex
$(\tform_{\mathbb{Z},X}^{\bullet},d'')$ of 
$(\tform_{\mathbb{R},X}^{\bullet},d'')$ easily,
and the differential of 
$(\tform_{\mathbb{Z},X}^{\bullet},d'')$
is also trivial. 
Therefore, the hypercohomology
$\mathbb{H}^{\bullet}(X;\tform_{\mathbb{Z},X}^{\bullet})$
also has
a very simple graded-commutative ring structure.  

Let $f\colon X\to Y$ be a morphism of rational polyhedral 
spaces. Then, this induces a canonical morphism
\begin{align}
b_f^{p}\colon \tform^{p}_{\mathbb{Z},Y}\to 
f_* \tform^{p}_{\mathbb{Z},X} \to
Rf_*\tform^{p}_{\mathbb{Z},X}
\end{align}
for all $p\in \mathbb{Z}_{\geq 0}$ and the pullback 
$f^{*}\colon 
\mb{H}^{\bullet}(Y;\tform_{\Z, Y}^{\bullet})\to 
\mb{H}^{\bullet}(X;\tform_{\Z, X}^{\bullet})$
\cite[Proposition 4.17]{gross2019sheaftheoretic}.
The pullback $f^{*}$ is a graded ring homomorphism. 

The canonical monomorphism $\mcal{O}^{\times}_X \to \mcal{A}^{0,0}_X$
induces the following commutative diagram 
\myfootnote{Several similar diagrams for some special tropical spaces 
and similar spaces appear
in literature, e.g. \cite[p.468]{MR2567952}
and \cite[Definition 1.45]{grossMirrorSymmetryLogarithmic2006a}.}:

\begin{equation} \label{equation-smooth-cartier-diagram}
  \begin{tikzcd}
    & 0 \arrow[d]    & 0 \arrow[d]           &                      &   \\
    & \mb{R}_{X} \arrow[r,equal] \arrow[d]   & \mb{R}_{X} \arrow[d]\\
0\arrow[r] & \mathcal{O}_X^{\times} \arrow[r] \arrow[d] &
\mcal{A}^{0,0}_X \arrow[r] \arrow[d] & \mcal{A}^{0,0}_X /\mcal{O}_{X}^{\times}  \arrow[r] \arrow[d,equal] & 0 \\
    0 \arrow[r] & 
\tform_{\mathbb{Z},X}^{1} \arrow[r] \arrow[d] & 
\mcal{Z}^{1}_{X} \arrow[r] \arrow[d]  & \mcal{Z}^{1}_{X}/
\tform_{\mathbb{Z},X}^{1} \arrow[r]   & 0 \\
    & 0 & 0 &  &
  \end{tikzcd}
\end{equation}
Here $\mcal{Z}^{q}_{X}\deq
  \opn{Ker}(d'': \mcal{A}^{0,q}_X\to \mcal{A}^{0,q+1}_X)$.
Every row and column of the diagram 
\eqref{equation-smooth-cartier-diagram} are exact
and $\mcal{A}_{X}^{0,0}/\mcal{O}^{\times}_X
  \simeq \mcal{Z}^{1}_X/\tform_{\Z,X}^{1}$ from the snake lemma.
The left column short exact sequence of 
\eqref{equation-smooth-cartier-diagram} is called
the 
\emph{tropical exponential sequence}
(e.g. \cite[\textsection 3]{MR3894860}).

By taking the long exact sequence for 
\eqref{equation-smooth-cartier-diagram}, we get a canonical morphism:
\begin{align} \label{equation-tropical-cartier}
c_1\colon H^{0}(X;\mathcal{A}^{0,0}_X/\mathcal{O}_X^{\times})\to \opn{Pic}(X)\to 
H^{1}(X;\tform_{\mathbb{Z},X}^{1}); s\mapsto [s] 
\mapsto c_1([s]).
\end{align}
The second morphism of
\eqref{equation-tropical-cartier} is 
called the Chern class map of 
$\opn{Pic}(X)$ in \cite[5]{mikhalkinTropicalCurvesTheir2008a}.
Let $f\colon X\to Y$ be a morphism of 
two boundaryless rational polyhedral spaces.
From definition of morphism of rational polyhedral space
and \cite[Lemma 2.21]{MR3903579},
$f$ induces the pullback 
$f^{*}\colon \CDiv(Y)\to \CDiv(X)$ 
and $c_1(f^{*}s)=f^{*}(c_1(s))$ from the diagram
\eqref{equation-smooth-cartier-diagram}. 
In particular, we have
$c_1(f^{*}(s)^{n})=f^{*}(c_1(s)^{n})$ for every 
$s\in\CDiv(Y)$
by \cite[Proposition 4.17]{gross2019sheaftheoretic}.

\subsection{Verdier duality and Borel--Moore homology
for tropical manifolds}
In this subsection, we recall the Verdier duality for
tropical manifolds from \cite{gross2019sheaftheoretic}.
See \cref{section-verdier-dual} if you are not familiar 
with the classical Verdier duality.
Every rational polyhedral space $X$ is locally compact, 
and thus we can use the Verdier duality for tropical manifolds.
We set $H^{p,q}(X;\mathbb{Z})\deq 
H^{q}(X;\tform_{\mathbb{Z},X}^{p})$. 
The total cohomology 
$H^{\bullet,\bullet}(X;\mathbb{Z})$ has a 
canonical ring structure induced from the 
hypercohomology 
$\mathbb{H}^{\bullet}(X;
\tform_{\mathbb{Z},X}^{\bullet})$.
Let $\upomega_X^{\bullet}$ be the dualizing complex 
of $D^{b}(\mathbb{Z}_X)$.
The $(p,q)$-th Borel--Moore homology of $X$
\cite[Definition 4.3]{gross2019sheaftheoretic} is
\begin{align}
H^{\opn{BM}}_{p,q}(X;\Z)\deq 
H^{-q}R\opn{Hom}(\tform_{\mathbb{Z},X}^{p},\upomega_X^{\bullet})\simeq 
\opn{Hom}_{D^{b}(\mathbb{Z}_X)}(\tform_{\mathbb{Z},X}^{p}[q],\upomega_X^{\bullet}).
\end{align}
$H_{0,q}^{\opn{BM}}(X;\Z)=
H^{-q}R\opn{Hom}(A_X,\upomega_X^{\bullet})$ 
is equal to the classical $p$-th Borel--Moore homology
\cite[Lemma 4.8]{gross2019sheaftheoretic}.
We note $H_{0,0}^{\opn{BM}}(X;\mathbb{Z})\simeq 
 \mathbb{Z}$ if 
$X$ is compact and path-connected.
This isomorphism follows from the universal 
coefficient theorem and 
$H^{1}_c(X;\mathbb{Z})$ is finitely generated and
torsion-free 
\cite[Chapter VI. Proposition 5.3]{iversenCohomologySheaves1986a}.

Let $f\colon X \to Y$ be a proper morphism
of rational polyhedral spaces.
Then, there exists the pushforward
$f_*\colon H_{p,q}^{\mathrm{BM}}(X;\mathbb{Z})
\to H_{p,q}^{\mathrm{BM}}(Y;\mathbb{Z})$
of tropical Borel--Moore homology
\cite[Definition 4.9]{gross2019sheaftheoretic}:
\begin{align}
f_*(\psi)\deq \opn{Tr}_{f,\upomega_Y^{\bullet}}\circ 
Rf_!\psi \circ b_{f}^{p}[q]\in 
\opn{Hom}_{D^{b}(\mathbb{Z}_X)}(
\tform^{p}_{\mathbb{Z},Y}[q],\upomega_Y^{\bullet}),
\end{align}
where 
$\opn{Tr}_{f,\upomega_Y^{\bullet}}$ is the counit 
$Rf_!f^{!}\upomega_Y^{\bullet}\to \upomega_Y^{\bullet}$.
\begin{example}

If $p=0$, $q=0$ and $f=a_X$, 
then $b^{0}_{a_X}$ is just a counit for 
the adjoint $a^{-1}_X\dashv Ra_{X*}$.
Therefore, we can identify
$a_{X*}\colon H^{\opn{BM}}_{0,0}(X;\mathbb{Z})
\to H^{\opn{BM}}_{0,0}(\{\opn{pt}\};\mathbb{Z})$
with $\opn{Hom}_{D(\mathbb{Z})}(\mathbb{Z},\opn{Tr}_{a_X,\mathbb{Z}})$
by the adjoint isomorphism\\
$\opn{Hom}_{D(\mathbb{Z}_X)}(a_X^{-1}\mathbb{Z},
\upomega_X^{\bullet})\simeq 
\opn{Hom}_{D(\mathbb{Z})}(\mathbb{Z},Ra_{X*}\upomega_X^{\bullet})$
if $a_X$ is proper.

\end{example}

Let $f\colon X\to Y$ and $g\colon Y\to Z$ be proper 
morphisms of rational polyhedral spaces.
Then, $f_*\circ g_*=(f\circ g)_*$ from 
\eqref{equation-trace}.

The tropical cohomology and the tropical Borel--Moore homology 
also have the cap product 
\cite[\textsection 4.6]{gross2019sheaftheoretic}:
\begin{align}
\cdot \frown \cdot \colon
H^{i,j}(X;\mathbb{Z}) \times 
H_{p,q}^{\mathrm{BM}}(X;\mathbb{Z})\to 
H_{p-i,q-j}^{\mathrm{BM}}(X;\mathbb{Z});(c,\alpha) 
\mapsto c\frown \alpha.
\end{align}

If $p=i$ and $q=j$, then the cap product $c\frown \alpha$ 
is just the composition $\alpha \circ c$ of 
morphisms:
\begin{align} \label{equation-composition-cap}
\cdot \frown \cdot =\cdot \circ \cdot \colon
\opn{Hom}_{D^{b}(\mathbb{Z}_X)}(
\mathbb{Z}_X,\tform^{p}_{\mathbb{Z},X}[q])\times &
\opn{Hom}_{D^{b}(\mathbb{Z}_X)}
(\tform^{p}_{\mathbb{Z},X}[q],\upomega_X^{\bullet}) \notag \\
\to &
\opn{Hom}_{D^{b}(\mathbb{Z}_X)}(\mathbb{Z}_X,
\upomega_X^{\bullet}).
\end{align}

Let $f\colon X \to Y $ be a proper morphism of 
rational polyhedral spaces.
Then, there exists the projective formula 
for $\alpha\in H_{p,q}^{\mathrm{BM}}(X)$
and $c\in H^{i,j}(Y)$ 
\cite[Proposition 4.18]{gross2019sheaftheoretic}:
\begin{align}
  f_*(f^{*}c\frown \alpha)=
c\frown f_*\alpha\in H_{p-i,q-j}^{\opn{BM}}(Y).
\end{align}

By an \emph{n-dimensional tropical manifold},
we mean a tropical manifold of pure dimension $n$,
for simplicity in this paper
(see \cite[\textsection 2.A]{MR3894860}).
Let $X$ be a compact
$n$-dimensional tropical manifold. 
Then, there exists the following natural isomorphism, 
which is called the Poincar\'e--Verdier duality
for tropical manifolds
\cite[Theorem 6.2]{gross2019sheaftheoretic}:
\begin{align}
\opn{PVD}^{(n-p)}_X\colon \tform_{\mathbb{Z},X}^{n-p}[n]
\simeq 
\mathcal{D}_{\mathbb{Z}_X}(\tform_{\mathbb{Z},X}^{p}),
\end{align}
where $\mathcal{D}_{\mathbb{Z}_X}(\mathcal{F}^{\bullet})
\deq R\mathcal{H}om_{\mathbb{Z}_X}(\mathcal{F}^{\bullet}
,\upomega^{\bullet}_X)$ for 
$\mathcal{F}^{\bullet}\in D^{b}(\mathbb{Z}_X)$.
The Poincar\'e--Verdier duality 
and the tensor-hom adjunction induce
the following isomorphism
(see also \eqref{equation-verdier-dual}):
\begin{align}
H^{n-p,n-q}(X;\mathbb{Z})=
R^{0}\opn{Hom}(\mathbb{Z}_X,
\tform_{\mathbb{Z},X}^{n-p}[n-q])\simeq 
R^{0}\opn{Hom}(\tform_{\mathbb{Z},X}^{p}[q],
\upomega_{X}^{\bullet})=
H_{p,q}^{\mathrm{BM}}(X;\mathbb{Z}).
\end{align}
This is the Poincar\'e duality
for tropical manifolds 
\cite[Corollary 6.3]{gross2019sheaftheoretic}.
The Poincar\'e duality for tropical manifold with
a global face structure was firstly 
proved in \cite[Theorem 5.3]{MR3894860}.

We recall the construction 
of $\opn{PVD}^{(n)}_X$ in \cite{gross2019sheaftheoretic}.
Let $\mathscr{Z}_n^{X}$ be the sheaf of tropical 
$n$-cycles \cite[Definition 3.5]{gross2019sheaftheoretic}
on an $n$-dimensional tropical manifold $X$ and 
$\mathcal{H}^{n}_X$ the $n$-th 
homology sheaf \cite[Definition 4.6]{gross2019sheaftheoretic}.
Then, there exists the following isomorphism:
\begin{align}
\mathbb{Z}_X \simeq\mathscr{Z}_n^{X}\simeq 
\mathcal{H}om(\tform^{n}_{\mathbb{Z},X},
\mathcal{H}^{n}_X).
\end{align}
The first isomorphism comes from \cite[Lemma 2.4]{MR3041763} and
the second isomorphism is true for $n$-dimensional rational polyhedral 
spaces \cite[Proposition 5.1]{gross2019sheaftheoretic}.
Let $X_{\reg}$ the set of generic points in $X$, i.e.,
the set of points which has an open neighborhood 
isomorphic to a Euclidean space
\cite[\textsection 4.B]{MR3894860}
(see also \cite[Definition 2.7]{gross2019sheaftheoretic}).
The constant function
$1_{X_{\reg}}\colon X_{\reg} \to \mathbb{Z}; 
x\mapsto 1$ defines a generator of 
$Z_n(X)=\Gamma(X;\mathscr{Z}_n^{X})\simeq 
\opn{Hom}(\tform^{n}_{\mathbb{Z},X},
\mathcal{H}^{n}_X)\simeq 
\opn{Hom}_{D^{b}(\mathbb{Z}_X)}(\tform^{n}_{\mathbb{Z},X}[n],
\upomega_X^{\bullet})$ if $X$ is connected.
In particular, $1_{X_{\reg}}$ defines a 
natural morphism 
$\tform^{n}_{\mathbb{Z},X}[n]\to \mathcal{H}^{n}_X[n] 
\to \upomega_X^{\bullet}$.
This morphism is just $\opn{PVD}^{(n)}_X$.
We call $[X]\deq \opn{PVD}_X^{(n)}$
the \emph{fundamental class} of $X$
(cf.~\cite[Definition 4.8]{MR3894860}).
We stress that the fundamental class
of every integral affine manifold (as tropical manifold) 
is determined without the data of orientation.

If $X$ is a compact and connected tropical manifold
of dimension $n$, 
we can define the \emph{trace map} 
on tropical cohomology like complex manifolds 
(e.g. \cite[Example 13.A.3]{MR2810322}):
\begin{align} 
\label{equation-trace-integration}
\int_X \colon H^{n,n}(X;\mathbb{Z})\to 
H_{0,0}^{\opn{BM}}(\{\opn{pt}\};\mathbb{Z})\simeq 
\mathbb{Z},
\qquad \alpha \mapsto a_{X*}(\alpha \frown [X]).
\end{align}
The trace map extends on $H^{\bullet,\bullet}(X;\mathbb{Z})$
naturally.
The fundamental class $[X]$ of $X$ is a generator of 
$H_{n,n}^{\opn{BM}}(X;\mathbb{Z})$, and thus we 
can define the degree of a proper morphism.
We can extend \eqref{equation-trace-integration}
for compact tropical manifolds naturally.
\begin{definition}[{cf.~\cite[Definition 2.11]{MR3668972}}]
Let $f\colon X \to Y$ be a proper surjective morphism
of $n$-dimensional tropical manifolds
such that $Y$ is connected.

The (tropical) degree
$\opn{deg}_{\opn{trop}}(f)$
of $f$ is the integer $m$
such that $f_*([X])=m[Y]$. 
\end{definition}
We write $\opn{deg}_{\opn{trop}}(f)$ by $\opn{deg}(f)$
for simplicity.
From the projection formula, every element 
$c\in H^{n,n}(Y;\mathbb{Z})$ have the equation below
when $X$ and $Y$ are compact:
\begin{align}
\int_{X}f^{*}c
=a_{Y*}(f_*(f^{*}c\frown [X]))
=a_{Y*}(c\frown \opn{deg}(f)[Y])
=\opn{deg}(f)\int_Y c.
\end{align}

\begin{example}
Let $f\colon X\to Y$ be a morphism of 
$n$-dimensional connected tropical manifolds such that $f$ is a 
covering map of topological degree $m$ and the associated map $d_xf\colon 
T_{x}^{\mathbb{Z}} X\to T_{f(x)}^{\mathbb{Z}}Y$ is 
isomorphic.
Let $\opn{cyc}_X$ be the cycle class map of $X$ 
\cite[Definition 5.4]{gross2019sheaftheoretic}.
Then, $f_*[X]=\opn{cyc}_X(f_*1_{X_{\reg}})
=m[\opn{cyc}_Y(1_{Y_{\reg}})]=m
[Y]$ from the commutativity of
tropical cycle maps
\cite[Proposition 5.6]{gross2019sheaftheoretic}
and the definition of the pushforward
of tropical cycles \cite[Definition 3.6]{gross2019sheaftheoretic}.
In this case, $\opn{deg}(f)=\opn{deg}_{\mathrm{top}}(f)$.
\end{example}

\begin{definition} \label{definition-etale-covering}
Let $f\colon X\to Y$ be a proper morphism 
of tropical manifolds such
that $Y$ is connected and $f$ is a covering map.
The morphism $f$ is a \emph{tropical \'etale} covering map 
if $\opn{deg}(f)=\opn{deg}_{\mathrm{top}}(f)$.  
\end{definition}
\begin{remark}
Our definition of \'etale covering is different from that of 
\cite[Definition 1.1]{grossMirrorSymmetryLogarithmic2006a}
when $X$ and $Y$ are integral affine manifolds.
A typical example of non-\'etale covering map 
in our sense but \'etale in the sense of
\cite{grossMirrorSymmetryLogarithmic2006a} is given
from the $n$-th Frobenius endomorphism 
$\opn{Fr}_{n,X}^{\natural}\colon \mathcal{O}_X \to \mathcal{O}_X; f\mapsto nf$ of 
the structure sheaf of tropical manifolds. 
This is a tropical analog of Frobenius morphisms of 
schemes in positive characteristic 
(e.g. \cite[IV. Remark 2.4.1]{hartshorneAlgebraicGeometry1977a}).
As like classical Frobenius morphisms, 
the morphism $\opn{Fr}_{n,X}=(\opn{id}_X,\opn{Fr}_{n,X}^{\natural})$ of semiringed space does 
not induce an endomorphism of a given tropical manifold 
(over $\mathbb{T}$), but
there exists the following commutative diagram:
\begin{equation}
\begin{tikzcd}
 (X,\mathcal{O}_X) \arrow[r,"{\opn{Fr}_{n,X}}"] 
\arrow[d,"{a_X}"']
 &   (X,\mathcal{O}_X) \arrow[d,"{a_Y}"] \\
(\{\opn{pt}\},\mathbb{T})
\arrow[r,"{\opn{Fr}_{n,\{\opn{pt}\}}}"]
 & (\{\opn{pt}\},\mathbb{T}).
\end{tikzcd}
\end{equation}
By the base change map $\opn{Fr}_{n,\{\mathrm{pt}\}}$, we get 
a new tropical manifold $X^{(n)}$
and the corresponding morphism 
$\opn{Fr}_{n,X}\colon X^{(n)}\to X$. 
If $X=\mathbb{R}^{d}$, then $X^{(n)}$ is 
isomorphic to $X$ and 
$\opn{Fr}_{n,X}$ corresponds to 
the $n$-th dilation map of $X$.
\end{remark}

\subsection{Sheaf cohomology for integral affine manifolds}

From now on, we recall the commutative diagram 
\eqref{equation-smooth-cartier-diagram}
 for integral affine manifolds.
In this case, every sheaf
in \eqref{equation-smooth-cartier-diagram} naturally comes 
from symplectic geometry.

Of course, 
$\mcal{A}_{B}^{0,0}=\mcal{C}^{\infty}_{B}$ and
$\mcal{Z}^{1}_{B}$ is the sheaf of 
closed $1$-forms on $B$.
We recall the sheaf $\mcal{Z}^{1}_{B}$ can be 
considered as the sheaf 
$\opn{Lag}(T^{*}B)$ of Lagrangian sections 
$s:U \to T^{*}U$ for open set $U \subset B$ 
(see \cite[3.2]{MR1853077} or some other standard textbook
of symplectic geometry).
The $\mcal{T}_{\Z,B}^{\vee}$ is isomorphic to
the period lattice of 
$\check{f}_{B}\colon \check{X}(B)\to B$ 
 \cite{duistermaatGlobalActionangleCoordinates1980a}.
Another important thing is that 
$\mcal{Z}^{1}(B)/\mcal{T}_{\Z,B}^{\vee}$ 
is isomorphic to the sheaf of germs of Lagrangian sections 
of the Lagrangian torus fibration 
$\check{f}_{B}\colon \check{X}(B)\to B$ 
\cite[(2.7), (2.11)]{duistermaatGlobalActionangleCoordinates1980a}.

Thus, we can rewrite the commutative diagram \eqref{equation-smooth-cartier-diagram} like this 
(e.g. \cite[p.468]{MR2567952}):

\begin{equation} \label{equation-cartier-lagrangian}
  \begin{tikzcd}
    & 0 \arrow[d]    & 0 \arrow[d]           &                      &   \\
    & \mb{R}_{B} \arrow[r,equal] \arrow[d]                & \mb{R}_{B} \arrow[d]           &                      &   \\
    0 \arrow[r] & \mathcal{O}_{B}^{\times} \arrow[r] \arrow[d]         & \mathcal{C}^{\infty}_{B} \arrow[r] \arrow[d] & \mathcal{C}^{\infty}_{B}/\mathcal{O}_{B}^{\times}  \arrow[r] \arrow[d,equal] & 0 \\
    0 \arrow[r] & \mcal{T}_{\Z,B}^{\vee} \arrow[r] \arrow[d] & \opn{Lag}(T^{*}B) \arrow[r] \arrow[d]  & \opn{Lag}(\check{X}(B)) \arrow[r]   & 0. \\
    & 0 & 0 &  &
  \end{tikzcd}
\end{equation}

From the diagram \eqref{equation-cartier-lagrangian},
we have the following group isomorphism already remarked 
before:

\begin{proposition}
\label{proposition-cartier-lagrangian}
Let $B$ be an integral affine manifold.
Then, the following group isomorphism exists:
\begin{align}
\CDiv(B)\simeq \Gamma(B;
\opn{Lag}(\check{X}(B)));s\mapsto L_s.
\end{align}
\end{proposition}

If two $C^{\infty}$-divisors $s,s'$ are linearly equivalent, 
i.e., $s-s'=(f)^{\mathrm{sm}}$ for some
$f\in C^{\infty}(B)$, 
then $s$ is the image of a one-time Hamiltonian flow of $s'$
on $\check{X}(B)$
\cite[Exercise 6.65]{MR2567952}.

From now on, we start to prove
\cref{theorem-MRR-hesse-preface}.

\begin{theorem} \label{theorem-MRR-hesse}
Let $B$ be an $n$-dimensional compact 
integral affine manifold admitting a Hessian form and
$s$ a $C^{\infty}$-divisor such that its
Lagrangian section $L_s$ intersects 
with the zero section transversely. Then,
  \begin{align} \label{equation-Hesse-RR}
\chi(\opn{LMD}^{\bullet}(B,s))=\frac{1}{n!}\int_{B}c_1(s)^{n}.
  \end{align}
\end{theorem}

\begin{proof}

As mentioned in \cref{remark-compact-hessian},
$B$ is a finite unramified cover of a tropical torus $T$.
Fix a tropical \'etale cover $p:T \to B$ of $B$.
If \cref{theorem-MRR-hesse} is true for tropical tori,
then \cref{theorem-MRR-hesse} is also true for compact
integral manifolds admitting a Hessian form from
\cref{proposition-euler-number-etale} for $p$:
\begin{align}
\chi(\opn{LMD}^{\bullet}(B;s))
=\opn{deg}(p)^{-1}\chi(\opn{LMD}^{\bullet}(T;p^{*}(s)))
=&\opn{deg}(p)^{-1}\frac{1}{n!}\int_T c_1(p^{*}s)^{n}
\notag \\
=&\frac{1}{n!}\int_{B}c_1(s)^{n}.
\end{align}

From now on, we assume $B=T$.
Let $s$ be an element of 
$\Gamma(T;\mcal{C}^{\infty}(T)/\mcal{O}_{T}^{\times})\simeq
\Gamma(T;\opn{Lag}(\check{X}(T)))$.
Every $s$ is linearly equivalent to a Lagrangian section
$L_{dq_s}$
defined from the differential $dq_s$ of a quadratic 
polynomial $q_s$ on the universal cover $\tilde{T}$ of $T$
(see \cite{mikhalkinTropicalCurvesTheir2008a} or 
\cite[\textsection 3.3]{MR4229604}).

From an explicit calculation of the ring structure of
$\mathbb{H}^{\bullet}(T;\tform_{\mathbb{Z},B}^{\bullet})
\simeq \bigwedge H^{0}(T;\tform_{\mathbb{Z},B}^{1})$
\cite[(6-2)]{MR4582532}
(or Sumi's result
\cite[Theorem 47]{MR4229604} and the cycle map), 
we get 
$\frac{1}{n!}\int_{T}c_1(s)^{n}$ is equal to the determinant
of the linear part of $dq_s$. 
The intersection number of Lagrangian
section $L_s$ and the zero section $L_0$ is equal to
$\chi(\opn{LMD}^{\bullet}(s_0,s))$ up to 
signature. The signature only depends on
the choice of 
an orientation of $\check{X}(B)$
(see e.g. \cite[\textsection 5.2]{MR1336822}
for the definition of 
the intersection number of smooth 
submanifolds).

Every linearly equivalent 
class is given by a one-time Hamiltonian flow on 
$\check{X}(B)$ by the pullback $\check{\pi}_{B}^{*}(h)$
of a smooth function $h$ on $B$.
We note the one-time flow of a submanifold of
$\check{X}(B)$ does not change
the homology class in $H_{\bullet}(\check{X}(B);\Z)$ and
the intersection number \cite[5.2.1. Theorem]{MR1336822}.
Therefore, 
$\chi(\opn{LMD}^{\bullet}(T,s))=
\chi(\opn{LMD}^{\bullet}(T,dq_{s}))$.

We can calculate $\chi(\opn{LMD}^{\bullet}(T,s))$ as follows:

(i) If $\det dq_s\ne 0$, $dq_s$ intersects to 
the zero section transversely. We can see that 
$\sharp(L_{0}\cap L_{dq_s})=|\det dq_s|$ directly, and thus
we get \eqref{equation-Hesse-RR}. 

(ii) If $\det dq=0$, we can choose a smooth vector field
 $v$ on $\check{X}(T)$ such that the one-time flow 
$\phi$ of $v$ make 
$\phi(L_{dq_s})\cap L_0=\emp$. 
Thus, $\chi(\opn{LMD}^{\bullet}(T,s))=0$.
\end{proof}

\begin{proof}[{Proof of \cref{theorem-MRR-hesse-preface}}]
We can deform $s$ to a 
linearly equivalent $C^{\infty}$-divisor $s'$ 
such that $L_{s'}$ intersects the zero section 
transversally and $s|_{B\setminus K}=s'|_{B\setminus K}$
for a sufficiently small neighborhood $K$
of $s_0\cap s$ by Morse perturbations for 
a local data for $s$. 
From \eqref{equation-poincare-hopf}, we have
$\chi(\opn{LMD}^{\bullet}(B,s))
=\chi(\opn{LMD}^{\bullet}(B,s'))$.
Therefore, \cref{theorem-MRR-hesse} 
for $s'$ proves \cref{theorem-MRR-hesse-preface}.
\end{proof}

\begin{remark}
\label{remark-todd-class}
According to \cite[\textsection 5.3]{mikhalkinTropicalGeometryIts2006},
the $k$-th Chern cycle $c_k(X)$
of an $n$-dimensional tropical manifold $X$ 
should be supported on the $(n-k)$-skeleton of $X$.
I don't know the explicit and
precise definition of $k$-th
Chern cycles of $X$ except
$k=0,1$ or $k=\dim X$
(see \cite[Definition 3.20]{shawTropicalSurfaces2015a} for 
the definition of the top Chern cycle of a tropical manifold), 
but we can consider the $0$-th Chern cycle $c_0(B)$ of 
a (connected) integral affine manifold $B$ should be $1$,
and the $k$-th Chern cycle should be $0$ when $k\ne 0$
since $B=B_{\reg}$.
Thus, the Todd class $\opn{td}(B)$ of $B$
should be $1$ when the Todd class of a tropical manifold
can be written as a formal power series of the Chern classes
as like complex geometry.
On the other hand,
 $\opn{td}(X(B))=1$
since the tangent bundle on $X(B)$ admits a flat
structure.

Besides, 
the Chern--Schwartz--MacPherson cycles for matroids 
were defined in 
\cite{lopezdemedranoChernSchwartzMacPhersonCyclesMatroids2020}.
As explained in 
\cite[Previous work]{lopezdemedranoChernSchwartzMacPhersonCyclesMatroids2020},
the Chern--Schwartz--MacPherson cycles of matroids should 
give a precise definition of higher Chern classes
of tropical manifolds
(see \cite[Previous work]{lopezdemedranoChernSchwartzMacPhersonCyclesMatroids2020}
for more details about it).
In particular, the above observation of the Chern classes
of integral affine manifolds is compatible with the 
properties of the Chern--Schwartz--MacPherson cycles
of matroids 
\cite[Proposition 2.12]{lopezdemedranoChernSchwartzMacPhersonCyclesMatroids2020}.
In a recent study,
a Lagrangian interpretation of
the Chern--Schwartz--MacPherson cycles of matroids 
were introduced in \cite{MR4583774}.
We expect that our approach has a good relation
with this interpretation. 
\end{remark}

\begin{remark}
When $s=\vep (f)^{\opn{sm}}$ for a sufficiently small
positive real number $\vep >0$ and 
a Morse function $f$ on $B$, 
$\chi(\opn{LMD}^{\bullet}(B,s))=
\chi_{\opn{top}}(B)=0$ is 
truly a special case of Poincar\'e--Hopf theorem for $B$.
We also note there exists another tropical analog of Poincar\'e--Hopf theorem
\cite{MR4540954}.
This analog is about the \emph{tropical Euler characteristic}
$\chi_{\opn{trop}}(B)\deq 
\chi(\mb{H}^{\bullet}(B;\tform_{\mathbb{R},B}^{\bullet}))$,
but not for the topological Euler characteristic $\chi_{\opn{top}}(B)$.
\end{remark}

\begin{remark}[{Kodaira--Thurston surface}]
There exists a complete and compact integral affine surface
which has no Hessian forms.
I learned the following example from
~\cite[Example 1.14]{grossMirrorSymmetryLogarithmic2006a}.
Mishachev classified integral affine structures 
on 2-torus in \cite[Theorem A]{MR1422337}.
If $B$ is an integral affine 2-torus and not isomorphic
to a tropical $2$-torus, then 
$H^{1}(\check{X}(B);\mathbb{R})=H^{1}(X(B);\mathbb{R})=3$
from the Leray spectral sequence for 
$\check{f}_B\colon \check{X}(B)\to B$ and
\cite[Theorem A]{MR1422337}. 
Therefore, the both $\check{X}(B)$ and $X(B)$ has no 
K\"ahler structure.
An example of such a $\check{X}(B)$ firstly
essentially appeared in \cite{MR402764}.
Since $X(B)$ has no K\"ahler structure,
$B$ also has no Hessian forms.
From Kodaira's classification of complex surfaces, 
$X(B)$ is isomorphic to a primary Kodaira surface.
Explicit calculations of the tropical cohomologies
$H^{\bullet}(B;\tform_{\mathbb{Z},B}^{\bullet})$
and the radiance obstructions of
tropical primary Kodaira surfaces are in \cite{maehara2023}
(see also \cite[\textsection 6]{MR1422337}).
Integral affine structures on a Klein bottle
and their tropical cohomologies are discussed in
\cite{MR2737696,shawTropicalSurfaces2015a,MR3894860}.

\end{remark}

\subsection{Relationships between Floer cohomology
and integral affine manifolds}
\label{section-floer-lmd}
 
Let $B$ be a compact integral affine manifold
such that $\pi_2(B)=0$. 
The condition $\pi_2(B)=0$ descends
to a good condition for Lagrangian sections
of $\check{f}_B\colon \check{X}(B)\to B$
by the following elementary proposition.  

\begin{proposition}

\label{proposition-unobstructed-lagrangian}
Let $B$ be a connected integral affine manifold 
and $\check{s}\colon B\to \check{X}(B)$ a 
Lagrangian section of 
$\check{f}_{B}\colon \check{X}(B)\to B$. 
Then, $\pi_i(B)\simeq \pi_i(\check{X}(B))$ 
for $i\geq 2$.
In particular, 
$\pi_i(B)=\pi_i(\check{X}(B))=0$ if 
$B$ is complete and $\pi_2(B,L_s)=0$,
where $L_s$ is the image of $\check{s}$.
\end{proposition}
\begin{proof}
Let $p\colon \widetilde{X} \to X$ be the universal cover of 
a topological space $X$, $p\in X$ and 
$\tilde{p}\in f^{-1}(p)$.

We note $\pi_{i}(X,p)\simeq 
\pi_{i}(\widetilde{X},\tilde{p})$ for $i\geq 2$ and 
the $i$-th homotopy group functor from the category of 
pointed spaces to the category of groups preserves finite
products.
Therefore,
\begin{align}
\pi_{i}(\check{X}(B))\simeq 
\pi_{i}(\widetilde{\check{X}(B)})\simeq 
\pi_{i}(\widetilde{T^{*}B})\simeq 
\pi_{i}(T^{*}\widetilde{B})\simeq 
\pi_{i}(T_p^* B)\times \pi_{i}(\widetilde{B})\simeq 
\pi_{i}(B).
\end{align}
Thus, we have 
$\pi_{i}(B)\simeq \pi_{i}(\check{X}(B))\simeq 
\pi_{i}(X(B))$ for $i\geq 2$. 

Let $i\colon L_{s}\to \check{X}(B)$ be the inclusion map of 
$L_{s}$.
Since $\check{s}$ is a section of 
$\check{f}_{B}$, 
$i_*\colon \pi_i(L_{s})\to \pi_i(\check{X}(B))$ is 
injective. Therefore, we get 
$\pi_2(\check{X}(B),L_s)=0$ by considering the long exact
sequence of the relative homotopy groups.
\end{proof}

\cref{proposition-unobstructed-lagrangian} deduces 
that every pseudoholomorphic disk
$\psi\colon D\to \check{X}(B)$ such that 
$\psi(\partial D)\subset L_s$, 
has the zero symplectic area.
In particular, $L_s$ is 
\emph{tautologically unobstructed} defined in
\cite[(2.50)]{MR3656481}.
Every Lagrangian section
$L_s\xto{i} \check{X}(B)\xto{\check{\pi}_{B}}  B$
induces a homomorphism of 
cohomology 
$H^{\bullet}(B;\mathbb{F}_2)\to 
H^{\bullet}(\check{X}(B);\mathbb{F}_2)
\to 
H^{\bullet}(L_s;\mathbb{F}_2)$.
In particular, 
$i^{*}(\check{\pi}_{B}^{*}(w_2(B)))=w_2(L_s)$,
and thus any Lagrangian section $L_s$ satisfies 
the condition \cite[(2.54)]{MR3656481}.
Every Lagrangian section $L_s$ has 
a canonical lifting to a Lagrangian 
brane $\mathscr{L}_s$ as described in 
\cite[5.2]{MR1882331}.
Therefore, every Lagrangian section of 
$\check{f}_B\colon \check{X}(B)\to B$
has a lift to an object of Fukaya category
of $\check{X}(B)$ in the sense of 
\cite{MR4301560}.

We also note about the Floer cohomology of 
a pair of Lagrangian sections.
If a pair $(L_s,L_s')$ of Lagrangian sections 
intersects transversally, then the graded 
module of Floer complex of the pair
$(\mathscr{L}_s,\mathscr{L}_{s'})$
(over $\Lambda_{\opn{nov}}^{\mathbb{C}}$) 
is the following:
\begin{align}
\opn{CF}^{\bullet}(\mathscr{L}_s,\mathscr{L}_{s'})
\deq \bigoplus_{p\in L_s\cap L_{s'}}
\Lambda_{\opn{nov}}^{\mathbb{C}}
[-\mu_{(\mathscr{L}_s,\mathscr{L}_{s'})}(p)],
\end{align}
where $\mu_{(\mathscr{L}_s,\mathscr{L}_{s'})}(p)$ is 
the Maslov index of the pair 
$(\mathscr{L}_s,\mathscr{L}_{s'})$ at $p$.
In this case, the Maslov index at $p$
is equal to the Morse index of the 
local intersection data $f_{p}$ of 
the $C^{\infty}$-divisor $s'-s$ at $p$
\cite[Remark 13]{MR1882331}. Therefore, 
if $L_s$ and $L_{s'}$ intersect 
transversally, then there exists the following
isomorphism of graded modules:
\begin{align}
\opn{CF}^{\bullet}(\mathscr{L}_s,\mathscr{L}_{s'})
\simeq 
\opn{LMD}^{\bullet}_{\Lambda_{\opn{nov}}^{\mathbb{C}}}
(s,s').
\end{align}
This isomorphism supports that our analog of the graded module
of Floer complex is compatible with the classical 
Floer complex.
To summarize, every pair $(L_1,L_2)$ of Lagrangian sections
of $B$ has a 
Floer cohomology $\opn{HF}^{\bullet}(\mathscr{L}_1,
\mathscr{L}_2)$, and its Euler 
characteristic is equal to the intersection number 
of $L_1$ and $L_2$ (up to signature). We can 
see this fact from \cite[Remark 13]{MR1882331} directly.

\begin{remark}
\label{remark-markus-conjecture}
The condition $\pi_2(B)=0$
follows from the \emph{Markus conjecture}
\cite[p.53]{markus1963cosmological},
which implies that the universal cover
of a compact integral affine manifold is contractible.
\end{remark}

\section{More examples}
\label{section-more-examples}
In the last section, we mention more examples of 
tropical analog of the Euler characteristic of 
the sheaf cohomology of line bundles.
(See also \cref{section-toric-geometry}
for the local Morse data for
tropical toric manifolds associated with
lattice polytopes.)

\subsection{Tropical multidivisors and 
Lagrangian multi-sections}
\label{section-tropical-multi-section}

Let $B$ an integral affine manifold. 
Then, $B$ has 
the notion of Lagrangian multi-section, 
which is a mirror part of some 
vector bundle on $X(B)$. The Lagrangian
multi-sections are studied by 
various researchers. 
We can consider (unramified) Lagrangian multi-sections 
as a certain analog of  
tropical multidivisors on a compact tropical
curve \cite[Definition 3.1]{gross2022principal}, 
so we can also define 
$C^{\infty}$-multidivisors
for tropical curves and 
the local Morse data of them naturally, 
and thus we can generalize 
\cref{theorem-MRR-curve-preface} for 
$C^{\infty}$-multidivisors satisfying
\cref{condition-good-divisor} by
the Riemann--Hurwitz formula for tropical 
curves \cite[Theorem 2.14]{MR2525845}.

\subsection{K\"unneth-type formula}
We remark that there exists
a tropical analog of the K\"unneth formula for 
local Morse data (\cref{corollary-kunneth-type-formula}).
For proving \cref{corollary-kunneth-type-formula},
we use the following sheaf theoretic version of
the Thom--Sebastiani formula:

\begin{theorem}[{Thom--Sebastiani Theorem for sheaves \cite[Theorem 1.2.2]{MR2031639}}]
Let $i=1,2$.
Fix locally compact Hausdorff spaces $X_i$ with
countable topology and of finite cohomological dimension,
and sheaves $\mathcal{F}_i$ on $X_i$,
and compact subsets $S_i$ of $\{f_i=0\}$.
For continuous functions $f_i\colon X_1 \to\mathbb{R}$,
we write $f_1\boxplus  f_2\deq \pi_1^{*}f_1+\pi_2^{*}f_2$
and 
$\mu_{f_i}\mathcal{F}\deq R\Gamma_{\{f_i\geq 0\}}
(\mathcal{F})|_{\{f_i=0\}}$.
If the above data satisfy
the condition of cohomological version of a Milnor 
fibration \cite[Assumption 1.1.1]{MR2031639},
then there exists the following isomorphism 
for $\mcal{F}_1\boxtimes^{L} \mcal{F}_2\deq 
p_1^{*}\mcal{F}_1\otimes^{L}p^{*}_2\mcal{F}_2$:
\begin{align}
    R\Gamma(S_1\times S_2;\mu_{f_1\boxplus f_2}(\mcal{F}_1\boxtimes^{L} \mcal{F}_2))
    \simeq R\Gamma(S_1;\mu_{f_1}(\mcal{F}_1))
    \otimes^{L}_{A_X}R\Gamma(S_2;\mu_{f_2}(\mcal{F}_2)).
\end{align}

\end{theorem}

\begin{example}[{\cite[p.22]{MR2031639}}]
If $S_1=\set{v},S_2=\set{w}$ and $\mcal{F}_1=\Z_V, \mcal{F}_2=\Z_{W}$,
then the Thom-Sebastiani theorem gives
a certain K\"unneth type formula for
$\opn{LMD}^{\bullet}(\Z_{V\times W},f_1\boxplus f_2,(v,w))$:
\begin{align}
\opn{LMD}^{\bullet}(\Z_{V\times W},f_1\boxplus f_2,(v,w))
&\simeq \opn{LMD}^{\bullet}(\Z_{V},f_1,v)
\otimes_{\Z} \opn{LMD}^{\bullet}(\Z_{W},f_2,w), \quad \\
\opn{ind}_{(v,w)}(f_1\boxplus f_2)&=\opn{ind}_v(f_1)\cdot 
\opn{ind}_w(f_2).
\end{align}
  From instance, if $f_1(x)=\|\cdot\|_{\mathbb{R}^{n}}^{2}$
  and $f_2(x)=-\|\cdot\|_{\mathbb{R}^{m}}^{2}$, then we have
\begin{align}
    \opn{LMD}^{\bullet}(\Z_{{\mathbb{R}}^{n+m}},f_1\boxplus f_2,0)
    \simeq \tilde{H}^{\bullet -1}(S^{m-1};\Z)
    \simeq \Z[-\opn{ind}_{\mathrm{Morse}}(f_1\boxplus f_2,0)],
\end{align}
  where $\opn{ind}_{\mathrm{Morse}}(f,0)$ is the Morse index
  of a Morse function $f$ at the origin.
\end{example}

We won't state a K\"unneth type formula for 
local Morse data for $C^{\infty}$-divisors in a fully
general setting, but the following corollary 
is enough for performing that there exists 
a K\"unneth type formula for the local Morse data
of $C^{\infty}$-divisors.

\begin{corollary}[{K\"unneth type formula}]
\label{corollary-kunneth-type-formula}
Fix boundaryless compact rational polyhedral spaces $X$ and $Y$.
Let $s$ and $s'$ a permissible $C^{\infty}$-divisor of $X$ and $Y$ 
admitting intersection data which satisfy
the condition of cohomological version of a Milnor 
fibration \cite[Assumption 1.1.1]{MR2031639} locally.
Then, the external tensor product 
$s\boxtimes s'\deq \opn{pr}_X^{*} (s)+\opn{pr}_Y^{*}(s')
\in \CDiv(X\times Y)$ is also permissible and
has the following equation:
\begin{align}
\chi(\opn{LMD}^{\bullet}(X\times Y;s\boxtimes s'))=
\chi(\opn{LMD}^{\bullet}(X,s))\chi(\opn{LMD}^{\bullet}(Y,s')).
\end{align}

\end{corollary}
\begin{proof}
By
\cite[Proposition 5.4.1]{MR1299726},
$s\boxtimes s'$ is also prepermissible.
Of course, $s_0\cap (s\boxtimes s')$ is finite.
Let $\{(U_v,f_v)\}_{v\in s_0\cap s}$
(resp. $\{(U_w,g_w)\}_{v\in s_0\cap s'}$) is an
intersection data for $s$ (resp. $s'$), 
and $\{(U_{v}\times U_w,f_v\boxplus  g_w)\}_{
(v,w)\in s_0 \cap 
(s\boxtimes s')}$ the 
corresponding intersection data 
for $s\boxtimes s'$.  
We get an isomorphism of graded modules
from the sheaf theoretic Thom-Sebastiani theorem for 
$X_1=U_{v}$, $X_2=U_w$, $S_1=\{v\}$, $S_2=\{w\}$,
$f_1= f_v-f_v(v)$ and 
$f_2= f_w-f_w(w)$:
\begin{align}
\opn{LMD}^{\bullet}(X\times Y;s\boxtimes s') 
& =\bigoplus_{(v,w)\in s_0\cap s\boxtimes s'}
(R^{\bullet}_{\{f_v\boxplus  g_w\geq f_v\boxplus  g_w(v,w)\}}
\Z_{U_{v}\times U_w})_{(v,w)} \\
& \simeq \bigoplus_{(v,w)\in s_0\cap s\boxtimes s'}
(R^{\bullet}_{\{f_v\geq f_v(v)\}}\Z_{U_v})_v
\otimes_{\Z} (R^{\bullet}_{\{g_w\geq g_w(w)\}}\Z_{U_w})_w \\
& \simeq
\opn{LMD}^{\bullet}(X;s)\otimes_{\Z} 
\opn{LMD}^{\bullet}(Y;s').
\end{align}

\end{proof}

\begin{remark}
In the proof of \cref{corollary-kunneth-type-formula},
we only use
the condition of cohomological version of Milnor fiber. 
Besides, we can find many prepermissible $C^{\infty}$-divisors
satisfying this condition, see \cite[p.35]{MR2031639}.  
\end{remark}

As another result, we get a formula
for the $n$-th symmetric product 
$S^{n}X\deq 
(X^{n}/\mathfrak{S}_n,q_*^{\mathfrak{S}_n}\mathcal{O}_{X^{n}}^{\times})$ of 
a rational polyhedral space $X$, even though 
the $S^{n}X$ is not a rational polyhedral space. 
Here, $\mathfrak{S}_n$ is the $n$-th permutation group 
and $q_*^{\mathfrak{S}_n}\mathcal{O}_{X^{n}}^{\times}$ the $\mathfrak{S}_n$-invariant 
part of the pushforward $q_* \mathcal{O}_{X^{n}}^{\times}$ for the canonical projection 
$q\colon X^{n}\to X^{n}/\mathfrak{S}_{n}$. In this case, 
we can define the $n$-th symmetric product
$\opn{Sym}^{n}(s)$ of 
$s$ as an element of $H^{0}(X^{n}/\mathfrak{S}_n;
q_*^{\mathfrak{S}_n}\mathcal{A}^{0,0}_{X^{n}}/
q_*^{\mathfrak{S}_n}\mathcal{O}_{X^{n}}^{\times})$
naturally.
We temporarily define $\opn{Sym}^{n}(s)$ is prepermissible if
the pullback $q^{*} \opn{Sym}^{n}(s)=s^{\boxtimes n}$ is 
prepermissible. We also define 
$s_0 \cap \opn{Sym}^{n}(s)\deq q_*(s_0\cap s^{\boxtimes n})$ 
and $\opn{LMD}^{\bullet}_{A}(S^{n}X;\opn{Sym}^{n}(s))$ like 
that of $s$.
\begin{corollary}
Let $k$ be a field of characteristic $0$ and $X$ 
a compact boundaryless rational 
polyhedral space
and $s$ a prepermissible $C^{\infty}$-divisor 
such that $\opn{LMD}^{\bullet}_k(X;s)$ is
finitely generated and 
satisfies \cite[Assumption 1.1.1]{MR2031639} locally.
Let $\opn{Sym}^{n}(s)$ be the $n$-th symmetric product of 
$s$ and
$\opn{LMD}^{\bullet}_{k}(S^{n}X;\opn{Sym}^{n}(s))$
the local Morse data of $\opn{Sym}^{n}(s)$.
Then,
\begin{align}
\label{equation-symmetric-product}
\chi(\opn{LMD}_k^{\bullet}(S^{n}X;\opn{Sym}^{n}(s)))=
\binom{n+\chi(\opn{LMD}_{k}^{\bullet}(X;s))-1
}{n}.
\end{align}

\end{corollary}
\begin{proof} 
Let $\{(U_p,f_p)\}_{p\in s_0\cap s}$ be
an intersection data for $s$ such that 
$f_p(p)=0$ for all $p\in s_0\cap s$,
$U=\coprod_{p\in s_0\cap s}U_p$
and 
$f\deq \coprod_{p\in s_0\cap s} f_p$.
Then, 
$\opn{LMD}_k^{\bullet}(X;s)
=\opn{LMD}_k^{\bullet}(A_{U},f)$ from definition.

Let $f^{\boxplus  n}\colon U^{n}\to \mathbb{R}$
be the continuous function induced from $f$,
and $f_{\mathfrak{S}_n}\colon
U^{n}/\mathfrak{S}_n\to \mathbb{R}$ 
is a continuous function induced by the quotient map 
$q\colon X^{n}\to X^{n}/\mathfrak{S}_n$.
Then, we have 
$\opn{LMD}^{\bullet}_{k}(S^{n}X;\opn{Sym}_n(s))
\simeq \opn{LMD}^{\bullet}
(k_{U^{n}/\mathfrak{S}_n};f_{\mathfrak{S}_n})$.
The latter graded module is isomorphic to
the cohomology of the complex
$R\Gamma(s_0\cap \opn{Sym}^{n}(s),
R\Gamma_{\{f^{\boxplus  n}\geq 0\}/\mathfrak{S}_n}
(k_{U^{n}/\mathfrak{S}_n}))$.
From now on, we calculate this complex
using the theory of $\mathfrak{S}_n$-equivariant
sheaves. 
(See \cite[Chapter V]{MR102537} or
\cite[\textsection 0.2]{MR1299527}
for the theory of $G$-equivariant sheaves
of a discrete group $G$.
We mainly follow the terminologies and notions
of equivariant sheaves
from \cite[\textsection 8]{MR1299527}).
In this case, every $\mathfrak{S}_n$-equivariant
sheaves on a space $X$ with a (right) action 
$\rho \colon X\times \mathfrak{S}_n\to X$
is equivalent to a sheaf $\mathcal{F}$ on $X$
with a family $\{a_{\sigma}\}_{\sigma\in \mathfrak{S}_n}$
of morphisms $a_{\sigma}\colon \mathcal{F}\to \rho(\sigma)^{-1}\mathcal{F}$
satisfying $a_{\sigma \tau}=\rho(\tau)_*a_{\sigma}\circ a_{\tau}$ 
for all $\sigma,\tau \in \mathfrak{S}_n$ 
(see also \cite[Chapter IV.9]{iversenCohomologySheaves1986a}).
In particular, the constant sheaf 
$k_{U^{n}}$ is an $\mathfrak{S}_n$-equivariant sheaf
trivially.
Let $\catn{Mod}_{\mathfrak{S}_n}(k_{U^n})$ be 
the category of $\mathfrak{S}_n$-equivariant
sheaves of $k_{U^n}$-modules. 
From \cite[Proposition 5.1.1]{MR102537},
$\catn{Mod}_{\mathfrak{S}_n}(k_{U^n})$ have enough
injectives.
We note the forgetful functor
$\catn{Mod}_{\mathfrak{S}_n}(k_{U^n})
\to \catn{Mod}(k_{U^n})$ preserve injectives
~\cite[Corollary 8.2.4]{MR1299527}.
The subset $\{f^{\boxplus  n}\geq 0\}$ of $U^{n}$ is 
closed and has an $\mathfrak{S}_n$-action induced
from $U^n$.
Let $(\mathcal{F},\{a_{\sigma}\}_{\sigma\in \mathfrak{S}_n})$
be an $\mathfrak{S}_n$-equivariant sheaf on $U^{n}$. 
Then,
\begin{align}
\Gamma_{\{f^{\boxplus  n}\geq 0\}}(a_{\sigma \tau})
=\rho(\tau)_*(\Gamma_{\{f^{\boxplus  n}\geq 0\}}(a_{\sigma}))
\circ \Gamma_{\{f^{\boxplus  n}\geq 0\}}(a_{\tau})
\end{align}
for all $\sigma, \tau \in \mathfrak{S}_n$ 
from \cite[(2.3.20)]{MR1299726}.
Therefore, 
the functor $\Gamma_{\{f^{\boxplus  n}\geq 0\}}(\cdot)$ 
lifts a functor from 
$\catn{Mod}_{\mathfrak{S}_n}(k_{U^n})$
to itself.
The same logic works for $(\cdot)_{\{f^{\boxplus  n}\geq 0\}}$
since $\rho(\sigma)_*\cdot \rho(\sigma)^{-1}=\opn{id}=
\rho(\sigma)_*\circ \rho(\sigma^{-1})_*$
for all $\sigma\in\mathfrak{S}_n$.
The functor $q_*^{\mathfrak{S}_n}\colon 
\catn{Mod}_{\mathfrak{S}_n}(k_{U^n})\to 
\catn{Mod}(k_{U^n/\mathfrak{S}_n});\mathcal{F}
\to q_*^{\mathfrak{S}_n}(\mathcal{F})$ 
is exact and satisfies the following isomorphism:
\begin{align}
\label{equation-equivariant-quotient}
q_*^{\mathfrak{S}_n}\Gamma_{\{f^{\boxplus  n}\geq 0\}}k_{U^{n}}
=\Gamma_{\{f^{\boxplus  n}\geq 0\}/\mathfrak{S}_n}q_*^{\mathfrak{S}_n}k_{U^{n}}
=\Gamma_{\{f^{\boxplus  n}\geq 0\}/
\mathfrak{S}_n}k_{U^{n}/\mathfrak{S}_n}.
\end{align}
Moreover,
$\Gamma_{\{f^{\boxplus n}\geq 0\}}(\cdot)$
and $q_*^{\mathfrak{S}_n}$ preserve 
injectives since
both of them have left adjoint exact functors
\cite[Proposition 8.4.1]{MR1299527}. So, we can get
the derived version of \eqref{equation-equivariant-quotient}.
By \eqref{equation-equivariant-quotient} and
the sheaf theoretic Thom--Sebastiani formula
(see also \cite[Corollary 1.3.2]{MR2031639}), 
we get 
\begin{align}
R\Gamma(s_0\cap \opn{Sym}^{n}(s);
R\Gamma_{\{f^{\boxplus  n}\geq 0\}/\mathfrak{S}_n}
(k_{U^{n}/\mathfrak{S}_n}))
&\simeq 
R\Gamma(s_0\cap \opn{Sym}^{n}(s);
q_*^{\mathfrak{S}_n}R\Gamma_{\{f^{\boxplus  n}\geq 0\}}
(k_{U^{n}})) \\
&\simeq
(R\Gamma(s_0\cap s^{\boxtimes n};
R\Gamma_{\{f^{\boxplus  n}\geq 0\}}
(k_{U^{n}})))^{\mathfrak{S}_n} \\
&\simeq
(R\Gamma(s_0\cap s^{\boxtimes n};
R\Gamma_{\{f\geq 0\}^{n}}
(k_{U^{n}})))^{\mathfrak{S}_n} \\
&\simeq
\label{equation-macdonald}
((R\Gamma(s_0\cap s;
R_{\{f\geq 0\}}(k_{U})))^{\overset{L}{\otimes}n}
)^{\mathfrak{S}_n}.
\end{align}
The group action of $\mathfrak{S}_n$ on
$(R\Gamma(s_0\cap s;
R_{\{f\geq 0\}}(k_{U})))^{\overset{L}{\otimes}n}$
is compatible with 
the natural $\mathfrak{S}_n$-action
on the $n$-th tensor product of 
$\opn{LMD}^{\bullet}_k(X;s)$.
Therefore, we can apply the Macdonald formula
\cite{MR143204} for
the cohomology of \eqref{equation-macdonald}.
The Macdonald formula and the
expression of the formal power series
$\frac{1}{(1-t)^{m}}=\sum_{n=0}^{\infty}\binom{n+m-1}{n}t^{n}$
deduces \eqref{equation-symmetric-product}.
\end{proof}

\begin{remark}
The above corollary is an analog of 
\cite[Lemma 5.1]{MR1795551}.
\end{remark}

\appendix

\section{Tropical Riemann--Roch theorem and 
its difficulty}
\label{section-tropical-riemann-roch}

Many authors have studied homological algebra of semimodules 
or non-Abelian categories (e.g. 
\cite{MR3051517,MR3211743,MR3939048,https://doi.org/10.48550/arxiv.2202.01573}),
but the theory of sheaf cohomology of 
$\mathbb{T}$-semimodules on tropical spaces 
have been not applied to prove a tropical
analog of the Hirzebruch--Riemann--Roch theorem yet.
The difficulty of homological algebra of semimodules
over an idempotent semiring relates 
with the difficulty of the formulation of
higher dimensional
Riemann--Roch theorem for tropical varieties.
In fact, current tropical analogs of Riemann--Roch 
theorem are not formulated and proved by 
homological algebra of semimodules.

We recall the tropical Riemann--Roch theorem for tropical curves 
in \cite{gathmannRiemannRochTheoremTropical2008a}. 
The rank $r(D)$ of the linear system of a divisor $D$ 
on tropical curve in 
\cite{gathmannRiemannRochTheoremTropical2008a}
is \emph{not} an invariant of $\mb{T}$-(semi)
modules
\myfootnote{
Yoshitomi \cite[Example 6.5]{yoshitomi2011generators} 
gave a simple example of divisors on 
tropical curves such that 
the ranks of divisors are different even though 
the $\mathbb{T}$-semimodules of the global sections of 
tropical line bundles are isomorphic.},
but this is truly a good analog of classical one 
(see \cite[Lemma 2.4]{MR2448666}).
The rank of linear system is generalized by Cartwright
\cite{MR4131998,MR4251610}.
Cartwright defines an invariant 
$h^{0}(\Delta,D)$ for a divisor on a tropical complex
\cite[Definition 3.1]{MR4251610}.
The $h^{0}(\Delta,D)$ is a certain 
analog of the dimension of
the $0$-th cohomology of a line bundle on algebraic 
variety.
If $\dim \Delta=1$, then $h^{0}(\Delta,D)=r(D)+1$
\cite[Proposition 3.3]{MR4251610}, and thus $h^{0}(\Delta,D)$
is a generalization of $r(D)+1$ for tropical complexes. 
Cartwright conjectured the Riemann--Roch inequality
$h^{0}(\Delta,D)+h^{0}(\Delta,K_{\Delta}-D)\geq 
\frac{D(D-K)}{2}+\chi_{\opn{top}}(\Delta)$
in \cite[Conjecture 3.6]{MR4251610}.
Sumi prove for the inequality for tropical tori and 
tropical toric surfaces \cite{MR4229604,sumi2021riemannroch}.
On the other hand, a tropical analog of
the Riemann--Roch theorem
for higher dimensional tropical manifolds 
(or tropical complexes) is not formulated
since there hasn't been any good definition of 
Euler characteristic of line bundles on tropical 
manifolds as an extension of the above notions yet.

We also note the other definition of the dimension of $\mb{T}$-semimodules 
(which are invariant for $\mb{T}$-semimodules) are
introduced by several researchers 
(see for instance 
\cite[Definition 2.3]{mikhalkinTropicalCurvesTheir2008a}
and \cite[p.8]{yoshitomi2011generators}) but 
these dimensions have no tropical analog of 
Riemann--Roch formula like 
\cite{MR2355607,gathmannRiemannRochTheoremTropical2008a}.
On the other hand, there exists an analog of Riemann--Roch formula 
for the tropicalization of the linear system of divisors
on an algebraic curve \cite[Corollary D]{MR4444458}.

\begin{remark} \label{remark-list-tropical-euler}
Here is the list of the reason why the Euler 
characteristic of the structure sheaf of tropical manifold
should be equal to the topological Euler characteristic of it.
{
\setlength{\leftmargini}{22pt}
\begin{enumerate}
\item If $C$ is a compact tropical curve and $D$ 
is the trivial divisor, then the tropical Riemann--Roch 
formula in 
\cite{gathmannRiemannRochTheoremTropical2008a}
says $r(D)-r(K_C-D)=\chi_{\opn{top}}(C)=
\frac{1}{2}\opn{deg}(-K_C)$. 
\item If $S$ is a tropical surface satisfying a certain 
good condition, then there exists the Noether formula
$\chi_{\opn{top}}(S)=\frac{c_1(K_S)^{2} +c_2(S)}{12}$ 
\cite[Theorem 5.1]{shawTropicalSurfaces2015a}.
Cartwright proved Noether's formula for
a good class of $2$-dimensonal tropical complexes
\cite[Proposition 1.3]{cartwright2015combinatorial}. 

\item 
The Hodge number $h^{p,q}(Z_w)$ of a general fiber $Z_w$ of
a one-parameter family $\mathcal{Z}$ of 
complex projective varieties over a punctured disc which 
has a tropical limit to a smooth projective 
$\mathbb{Q}$-tropical variety $X$, 
is equal to that of tropical homology 
$H_{q}(X;\mform_p)$ of it
\cite[Corollary 2]{itenbergTropicalHomology2019b}. If $p=0$, then
$h^{0,q}(Z_w)=\dim_{\mathbb{R}}H_{q}(X;\mathbb{R})
=\dim_{\mathbb{R}} H^{q}(X;\mathbb{R})$.
Therefore, $\chi(H^{\bullet}(Z_w;\mathcal{O}_{Z_w}))=
\chi_{\opn{top}}(X)$.
\end{enumerate}
}
\end{remark}

\section{Integral affine manifold and geometric prequantization}
\label{sec: BSRR}
\label{appendix-geometric-quantization}

In this subsection, we recall the relationships
between lattice points of 
strongly integral affine manifolds and geometric prequantization 
from \cite{1999math......2027T,MR4237881}.

A symplectic manifold $(M,\omega)$ is 
\emph{prequantizable}
if the de Rham cohomology class $[\omega]$ of $\omega$ is in 
$\opn{Im}(H^2(X;\Z)\to H^2(X;\mathbb{R}))$. 
An integral affine manifold $B$ has a strongly integral affine structure 
(\cref{definition-integral-affine-manifold})
if and only if the standard symplectic form 
on $\check{X}(B)$ is prequantizable
(see for example 
\cite[Lemma 2.8]{MR4275791}).

From the theory of connection on a vector bundle,
 there exists a complex line bundle $\mcal{L}$ with a $U(1)$-connection $\nabla$ and $R_{\nabla}=-2\pi\sqrt{-1}\omega$.  
Such pair $(\mcal{L},\nabla)$ is called 
a \emph{prequantum line bundle} on $M$. 
Obviously, $c_1(\mcal{L})=[\omega]$.
Let $L$ be a Lagrangian submanifold of $(M,\omega)$.
Then, $\nabla|_{L}$ is a flat $U(1)$-connection from definition.
$L$ is a 
\emph{Bohr--Sommerfeld orbit} 
(BS-orbit for short)
if $\nabla|_{L}$ is a trivial connection on $L$. Let $\pi :M\to B$ be a Lagrangian fibration. 
A point $x\in B$ is 
\emph{Bohr--Sommerfeld point} 
if $\pi^{-1}(x)$ is a Bohr--Sommerfeld orbit.

The origin of Bohr--Sommerfeld orbit comes from
old quantum theory.
See also \cite[15.2, Chapter 22]{MR3112817}
and \cite[Remark 4.3]{MR1270931} for a physical meaning 
of BS-orbits.

Fix $M=\mathbb{Z}^{n}$.
If $(B,\{(U_i,\phi_i)\}_{i\in I})$ is
an $n$-dimensional strongly integral affine manifold,
then we can define the set of lattice points $B(\Z)$ 
on $B$ for a fixed strongly integral affine structure from 
the inverse image $\phi^{-1}_i(U_i\cap M)$ of each atlas.
In other words, we can define 
`$\underline{\Z}$-rational points' of $B$,
where $\underline{\Z}$ is the max-plus semifield of 
extended integers,
see also \eqref{equation-rational-point}.

We can create a canonical section 
$s_{B}\colon B \to X(B)$ of the torus fibration
 $\pi_B:X(B)\to B$
induced from the graph of the composition map of developing map
$\opn{dev}_{B}\colon \widetilde{B}\to M_{{\mathbb{R}}}$ 
\cite[p.641]{goldmanRadianceObstructionParallel1984a} and 
the projection
$p:M_{{\mathbb{R}}}\to M_{{\mathbb{R}}}/M$. 
Then, we can consider $B(\Z)$ as the intersection of 
$s_{B}$ and the zero section of 
$\pi_{B}:X(B)\to B$;
\begin{align}
\label{equation-bohr-sommerfeld-riemann-roch}
B(\mathbb{Z})=\pi_{B}(s_0\cap s_{B})
=\{p\in B \mid  \check{\pi}_{B}^{-1}(p) 
\text{ is a BS-orbit}\}.
\end{align}
This construction is a certain dual of the semi-flat SYZ 
transformation 
(see \cite[Theorem 1.1]{MR1876073} or \cite[\textsection 9.1]{MR1882331}). 
The following equation is firstly proved in 
\cite[Corollary 4.1]{MR1461965} as a corollary of 
the index formula for reducible non-negative polarizations:
\begin{align} \label{equation-lattice-index}
\sharp (B(\Z))=\sharp (s_0\cap s_{B})
=\opn{vol}(\check{X}(B))=\opn{vol}(B)
=\int_{\check{X}(B)}\opn{ch}(\mathcal{L})\hat{\mathcal{A}}(T(\check{X}(B))).
\end{align}
\eqref{equation-bohr-sommerfeld-riemann-roch} also 
appears in \cite{MR1461965} essentially.
The equation \cref{equation-lattice-index} was reproved
in \cite{MR2676658} as a 
corollary of the theory of localization of the 
index of a Dirac type operator.
The comparison of LHS and RHS of 
\eqref{equation-lattice-index} for
Calabi--Yau manifolds is studied in 
\cite{1999math......2027T} under 
the name of the \emph{numerical quantization problem}.

\eqref{equation-lattice-index} can be considered
as a numerical version of 
\eqref{equation-hf-sheaf-cohomology} 
when $B$ has a Hessian form as explained
in \cite[\textsection 8.5]{MR4237881}.
If $B$ has a Hessian form, then $s_{B}$ can be 
considered as a Lagrangian section of the dual torus 
fibration $\pi_{\check{B}}\colon 
\check{X}(\check{B})\to \check{B}$ of dual Hessian
manifold $\check{B}$, see also 
\cite[pp.181-183]{grossMirrorSymmetryLogarithmic2006a}
and \cite[Proposition 6.9]{MR2567952}.
A similar result exists for projective toric manifolds 
and their prequantum line bundle 
\cite[Theorem 3.20]{yamaguchimaster}, see also 
\cref{section-toric-geometry}. 
These observations were developed for 
integral affine manifolds with singularities
$B$ and the associated toric degenerations 
by Gross and Siebert. 
Gross and Siebert conjectured 
that the set $B(\mathbb{Z})$ of integer points of 
$B$ forms a canonical basis of the global section a 
relatively ample line bundle of a toric 
degeneration of Calabi--Yau manifolds
\cite[Conjecture 1.6]{MR3525095}
(see also \cite[Conjecture 4.9]{grossSpecialLagrangianFibrations1998a}.)
We also mention that the numerical quantization problem 
is developed by several other perspectives
(e.g. \cite{MR2879247,
https://doi.org/10.48550/arxiv.1904.04076}).

\section{Verdier duality} 
\label{section-verdier-dual}

In this subsection, we recall the Verdier duality 
used in \cref{section-integral-affine-manifold}
from \cite[Chapter 3]{MR1299726}.

\begin{theorem}[{Verdier duality}]
\label{theorem-relative-verdier}
Let $f\colon X\to Y$ be a continuous map between locally 
compact Hausdorff spaces. If $f_!$ has finite cohomological
dimension, then
\begin{align}
R\opn{Hom}_{A_Y}(R f_!\mcal{F}^{\bullet},\mcal{G}^{\bullet})    & \simeq R\opn{Hom}_{A_X}(\mcal{F}^{\bullet},f^{!}\mcal{G}^{\bullet}),        \\
R \mcal{H}om_{A_Y}(R f_! \mcal{F}^{\bullet},\mcal{G}^{\bullet}) & \simeq Rf_* R\mcal{H}om_{A_X}(\mcal{F}^{\bullet},f^{!}\mcal{G}^{\bullet}).
\end{align}
\end{theorem}
The proof of \cref{theorem-relative-verdier} is in 
\cite[Proposition 3.1.10]{MR1299726}.
The \emph{dualizing complex} on $X$ over $A$ is 
$\upomega_{A_X}^{\bullet}\deq a^{!}_{X}A_{\{\opn{pt}\}}$.
For simplicity, we write 
$\upomega_X^{\bullet}\deq 
\upomega_{\mathbb{Z}_X}^{\bullet}$.

\begin{example}
If $X$ is a topological manifold of dimension $n$, 
then $\upomega^{\bullet}_{X}\simeq \opn{or}_{X}^{\Z}[n]$ 
where $\opn{or}_{X}^{\mathbb{Z}}$ is the sheaf 
associated with the presheaf $U\mapsto 
\opn{Hom}(H^{n}_c(U;\mathbb{Z}_{X}),\mathbb{Z})$
\cite[Proposition 3.3.6]{MR1299726}.
\end{example}

The complex $\mathcal{D}_{A_X}(\mathcal{F}^{\bullet})
\deq R\mathcal{H}om_{A_X}(\mathcal{F}^{\bullet},
\upomega_{A_X}^{\bullet})$ 
is called the \emph{Verdier dual} of 
$\mathcal{F}^{\bullet}$.
The dual functor
$\mathcal{D}_{A_X}\colon 
D^{b}(\catn{Mod}(A_X))\to 
D^{b}(\catn{Mod}(A_X)); \mathcal{F}^{\bullet} \mapsto  R\mathcal{H}om_{A_X}(\mathcal{F}^{\bullet},
\upomega_{A_X}^{\bullet})$
has the following canonical isomorphism when 
$\mathcal{D}_{A_X}(\mathcal{D}_{A_X}(\mathcal{G}^{\bullet}))
\simeq \mathcal{G}^{\bullet}$ by the tensor-hom adjunction
\cite[Proposition 2.6.3]{MR1299726}:
\begin{align}
R\mathcal{H}om_{A_X}
(\mathcal{D}_{A_X}(\mathcal{G}^{\bullet}),
\mathcal{D}_{A_X}(\mathcal{F}^{\bullet})) 
& \simeq R\mathcal{H}om_{A_X}
(\mathcal{F}^{\bullet}\otimes^{L}_{A_X}
\mathcal{D}_{A_X}(\mathcal{G}^{\bullet}),
\upomega_{A_X}^{\bullet}) \\
& \simeq R\mathcal{H}om_{A_X}
(\mathcal{F}^{\bullet},
\mathcal{D}_{A_X}(\mathcal{D}_{A_X}(\mathcal{G}^{\bullet}))) \\
& \simeq R\mathcal{H}om_{A_X}
(\mathcal{F}^{\bullet},
\mathcal{G}^{\bullet}).
\label{equation-verdier-dual}
\end{align}

The pair of adjoint functors $Rf_! \dashv f^{!}$ 
induces the \emph{trace map} $\opn{Tr}_{f,\mathcal{F}^{\bullet}}\colon Rf_!f^{!}\mathcal{F}^{\bullet}
\to \mathcal{F}^{\bullet}$ for each 
$\mathcal{F}^{\bullet}\in \opn{Ob}(D^{b}(A_X))$.
We follow the notation $\opn{Tr}_f$ from 
\cite[p.20]{MR1299726}.
Let $f\colon X\to Y $ and $g\colon Y\to Z$ be 
a continuous map and $f_!,g_!$ has finite cohomological 
dimension. Then, $(g\circ f)^{!}=f^{!}\circ g^{!}$
\cite[Proposition 3.1.8]{MR1299726}.
In particular, $\upomega_X^{\bullet}=f^{!}\upomega_Y^{\bullet}$.
From elemental properties of the composition of
adjoint functors \cite[p.103]{MR1712872}, we have
\begin{align}
\label{equation-trace}
\opn{Tr}_{g\circ f,\mathcal{F}^{\bullet}}=
\opn{Tr}_{g,\mathcal{F}^{\bullet}} 
\circ Rg_!(\opn{Tr}_{f,g^{!}(\mathcal{F}^{\bullet})}).
\end{align}
By the global section functor with compact support and
\eqref{equation-trace}, we have
\begin{align}
\label{equation-trace-push}
Ra_{Z!}(\opn{Tr}_{g\circ f,\upomega_{Z}^{\bullet}})
=Ra_{Z!}(\opn{Tr}_{g,\upomega_{Z}^{\bullet}})
\circ Ra_{Y!}(\opn{Tr}_{f,\upomega_{Y}^{\bullet}}).
\end{align}

\section{Lattice polytopes and local Morse data}
\label{section-toric-geometry}
In this section, we will see relationships
between lattice polytopes and tropical toric
varieties via local Morse data
(\cref{proposition-ehrhart-revisit}). 
For simplicity, we only consider them for 
tropical toric manifolds, i.e., 
smooth tropical toric varieties.

Before explaining about it, we recall some 
elemental properties of projective toric varieties 
from \cite{MR2810322}.
Let $P$ be an $n$-dimensional lattice polytope in 
$\mathbb{R}^{n}$, $X_P$ the projective toric variety
of $P$, $\mathcal{L}_P$ the line bundle on 
$X_P$ of $P$.
The sheaf cohomology of $\mathcal{L}_P$ has 
the following explicit isomorphisms:
\begin{align}
H^{\bullet}(X_P;\mathcal{L}_P)&=
H^{0}(X_P;\mathcal{L}_P)\simeq
\bigoplus_{u\in P\cap \mathbb{Z}^{n}} \mathbb{C}z^{u},
\label{equation-danilov-formula} \\
H^{\bullet}(X_P;\mathcal{L}_P^{\vee})&=
H^{n}(X_P;\mathcal{L}_P^{\vee})\simeq 
\bigoplus_{u\in \opn{int}(P)\cap \mathbb{Z}^{n}}\mathbb{C}z^{u}.
\end{align}
Besides, there exists a polynomial
$\opn{Ehr}_P\in \mathbb{Q}[x]$ such that the following 
equations hold
\begin{align} \label{equation-ehrhart-reciprocity-1}
\chi(H^{\bullet}(X;\mathcal{L}_P^{\otimes l}))
&=\opn{Ehr}_P(l)
=\sharp (lP\cap \mathbb{Z}^{n}), \\
\label{equation-ehrhart-reciprocity-2}
\chi(H^{\bullet}(X;(\mathcal{L}_P^{\vee})^{\otimes l}))
&=\opn{Ehr}_P(-l)=
\sharp (\opn{int}(lP)\cap \mathbb{Z}^{n})
\end{align}
for all $l\in \mathbb{Z}_{\geq 0}$.
\eqref{equation-ehrhart-reciprocity-1} and
\eqref{equation-ehrhart-reciprocity-2} 
are called the \emph{Ehrhart reciprocity law}. 
The proof of the law is in \cite[Theorem 9.4.2]{MR2810322},
for instance.
We also note that \eqref{equation-danilov-formula} 
has a differential geometrical interpretation,
see also \cref{appendix-geometric-quantization}.

Next, we recall elementary facts about topical 
toric varieties.
Basic concepts and properties 
of tropical toric varieties are studied and 
explained in 
\cite{MR2428356,MR2511632,meyer2011intersection,MR3287221,MR4016643}.
In this section, $k$ means a ground semifield.
We set
\begin{align}
N\deq \mathbb{Z}^{n},& \quad
N_{k}\deq 
\mathbb{Z}^{n}\otimes_{\mathbb{Z}}k^{\times}\simeq 
\opn{Hom}_{\mathbb{Z}}(N^{\vee},k^{\times}), \\
M\deq N^{\vee},& \quad 
M_{k}\deq M\otimes_{\mathbb{Z}}k^{\times}.
\end{align}
For instance, if $k$ is the max-plus semifield
$\underline{\mathbb{R}}$, then 
$N_{\underline{\mathbb{R}}}\simeq \mathbb{R}^{n}$.
At first, we can formally define (normal) 
toric varieties over a semifield as like over a 
field. For instance, there exists the following 
adjoint between the category of commutative monoids and 
the category of algebras over $k$:
\begin{align}
\label{equation-rational-point}
\opn{Hom}_{\catn{CMon}}(S_{\sigma},k)=
\opn{Hom}_{\catn{Alg}(k)}(k[S_{\sigma}],k),
\end{align}
where $\sigma$ is a rational convex polyhedral cone 
in $N_{\underline{\mathbb{R}}}$ and $S_{\sigma}\deq \sigma^{\vee}\cap M$ 
is the monoid associated with the dual cone 
$\sigma^{\vee}$ of $\sigma$ (\cite[p.30]{MR2810322}). 
Therefore, we can define the set of 
$k$-rational points of 
affine toric varieties for $\sigma$ over a semifield 
$k$ formally. By gluing each the set of $k$-rational points of 
affine toric variety, we get the set of $k$-rational points 
in a toric variety $X_{\Sigma}$ for a rational polyhedral fan
$\Sigma$.
If $k$ is a semifield $\mathbb{R}_{\geq 0}$ of 
nonnegative real numbers, then 
$(U_{\sigma})_{\geq 0}\deq \opn{Hom}_{\catn{CMon}}(S_{\sigma},\mathbb{R}_{\geq 0})$
can be considered as the nonnegative part of 
the affine toric variety
$U_{\sigma}\deq \opn{Hom}_{\catn{Alg}(\mathbb{C})}
(\mathbb{C}[S_{\sigma}],\mathbb{C})$ over $\mathbb{C}$
\cite[\textsection 12.2]{MR2810322}
(see also \cite[\textsection 1.3]{MR966447}).

Let $\opn{exp}\colon 
\underline{\mathbb{R}}\to \mathbb{R}_{\geq 0}$ be an 
extension of the exponential map 
$\opn{exp}\colon \mathbb{R}\to \mathbb{R}$ such 
that $\opn{exp}(-\infty)=0$ and 
$\opn{log}\colon \mathbb{R}_{\geq 0} \to 
\underline{\mathbb{R}}$ the inverse of $\opn{exp}$.
These maps are bijective, so they induce 
the following binary operations on 
$\underline{\mathbb{R}}$:
\begin{align}
\label{equation-rational-point-monoid}
x\oplus_{\opn{log}}y\deq 
\opn{log}(\opn{exp}(x)+\opn{exp}(y)),\quad 
x\odot_{\opn{log}}y \deq
\opn{log}(\opn{exp}(x)\cdot\opn{exp}(y))=x+y
\end{align}
The triple 
$\underline{\mathbb{R}}_{\opn{log}}
\deq (\underline{\mathbb{R}},\oplus_{\opn{log}},+)$ forms 
a semifield isomorphic to the semifield 
$(\mathbb{R}_{\geq 0},+,\cdot)$ of nonnegative 
of real numbers (see also \cite{MR874583}).
The semifield $\underline{\mathbb{R}}_{\opn{log}}$ 
is called the \emph{log semifield} or 
\emph{log semiring}. The log semifield 
$\underline{\mathbb{R}}_{\opn{log}}$ has a 
topological structure induced from 
$\opn{log}\colon \mathbb{R}_{\geq 0}\to 
\underline{\mathbb{R}}$.
This topological structure induces the topological 
structure for 
$N_{\underline{\mathbb{R}}}$ and this is homeomorphic 
to the $n$-dimensional Euclidean space $\mathbb{R}^{n}$
(\cite[Definition 1.2]{MR2428356}).
The isomorphism 
$\mathbb{R}_{\geq 0}\simeq 
\underline{\mathbb{R}}_{\opn{log}}$
induces the following commutative diagram of 
continuous maps for any continuous map 
$h\colon (\mathbb{R}_{> 0})^{n}\to \mathbb{R}_{> 0}$:
\begin{equation}
\begin{tikzcd}
	{(\mathbb{R}_{> 0})^{n}} & {\mathbb{R}_{> 0}} \\
	{\mathbb{R}^{n}} & {\mathbb{R}}
	\arrow["\opn{id}_N\otimes \log"', from=1-1, to=2-1]
	\arrow["h", from=1-1, to=1-2]
	\arrow["\log", from=1-2, to=2-2]
	\arrow["{f}"', from=2-1, to=2-2],
\end{tikzcd}
\end{equation}
where $f\deq \log \circ h\circ (\opn{id}_N \otimes \exp)$.
If $h$ is a $C^{\infty}$-function, so do $f$.

Let $h\in \mathbb{R}_{\geq 0}
[x_{1}^{\pm},\ldots,x_n^{\pm}]\setminus \{0\}$ be 
a nonzero Laurent 
polynomial of $n$-variables over $\mathbb{R}_{\geq 0}$.
The polynomial $h$ induces a positive
$C^{\infty}$-function on 
$(\mathbb{R}_{> 0})^{n}$ naturally.

Let $P$ be an $n$-dimensional convex lattice polytope in
$M_{\underline{\mathbb{R}}}=(\mathbb{R}^{n})^{\vee}$ and
$f_P\colon \mathbb{R}^{n} \to \mathbb{R}$ a 
polynomial function over $\underline{\mathbb{R}}$ as follows:
\begin{align}
  f_P(x)\deq\log (\sum_{u\in P\cap (\mathbb{Z}^{n})^{\vee}} 
\opn{exp}(a_u+\abk{u,x})),\quad  (a_u=0 
\text{ for all } u\in P \cap (\mathbb{Z}^{n})^{\vee}),
\label{equation-log-polynomial} 
\end{align}
where $\abk{\cdot,\cdot}\colon 
(\mathbb{R}^{n})^{\vee}\times \mathbb{R}^{n}\to\mathbb{R}$
is the canonical pairing of a vector space and its dual.

Let $x_1,\ldots,x_n$ be the standard coordinates of 
$\mathbb{R}^{n}$.
$\mathbb{R}^{n}$ has a canonical integral affine 
structure induced from the standard coordinates 
and the Hessian of $f_P$ induces a Hessian form 
on $\mathbb{R}^{n}$. 
The composition of the exact $1$-form
$df_P:{\mathbb{R}}^{n}\to T^{*}({\mathbb{R}}^{n}); x\mapsto 
(x;\frac{\partial f_P}{\partial x_1}(x),\ldots,\frac{\partial f_P}{\partial x_n}(x))$ 
of $f$, and the canonical projection 
$\opn{pr}\colon T^{*}({\mathbb{R}}^{n})\mapsto (\mathbb{R}^n)^{\vee}$
is an embedding onto the (relative) interior of $P$
\cite[p.124 Exercise]{MR1301331}. 
We set $\dot{\mu}_{P}^{\opn{trop}}\deq \opn{pr}\circ df$ 
for simplicity.
$\dot{\mu}_{P}^{\opn{trop}}$ is called the 
\emph{Legendre transform}
associated with $f_P$
\cite[p.121]{MR1301331}.
This map is also naturally appeared as an 
exponential family in algebraic statistic,
information geometry and Hessian geometry.
See also \cite[Appendix 2]{MR1301331} for general theory for
the relationship between the Hessian form of $f_P$ and 
the induced K\"{a}hler potential on
toric manifolds.

We can write down the 
$\dot{\mu}_P^{\opn{trop}}$ as follows:

\begin{align}
  \dot{\mu}_P^{\opn{trop}}(x)= \frac{1}{\sum_{m\in P\cap (\mathbb{Z}^{n})^{\vee}} \opn{exp}(\abk{x,m})}
  \sum_{m\in P\cap (\mathbb{Z}^{n})^{\vee}} \opn{exp}(\abk{x,m})m.
\end{align}

\begin{example}
Let $P=[0,1]$.
Then, $f_P\colon \mathbb{R} \to {\mathbb{R}}; x\mapsto 
\log (1+e^{x})$ be a soft 
plus function. 
The differential is a sigmoid function 
$df(x)=\frac{e^{x}}{1+e^{x}}$ and 
the image of $df$ is the open interval $(0,1)$.
The softmax functions are also appeared 
naturally in a similar way 
when we consider tropical projective spaces.
\end{example}

We can naturally extend $\dot{\mu}_{P}^{\opn{trop}}$ on tropical 
projective toric 
varieties $X_{P}^{\opn{trop}}$,
and then the extension map is called the \emph{tropical moment map}
$\mu_{P}^{\opn{trop}}: X_{P}^{\opn{trop}}\simto P$ 
\cite[Definition 2.1 (2)]{MR2428356}.
This is a tropical analog of algebraic moment map 
for toric varieties
(see for example
\cite[\textsection 12.2]{MR2810322}).
Instead of the moment map for Hamiltonian action of Lie group on
symplectic manifold, we
need not assume the convex lattice polytope is Delzant.
Let 
$q: ({\mathbb{R}}^{n})^{\vee}\to ({\mathbb{R}}^{n})^{\vee}/(\Z^{n})^{\vee}$ 
a canonical projection map.
We consider the following map
\begin{align}
\dot{s}_P\colon X_{P}^{\mathrm{trop}}\to 
X_{P}^{\mathrm{trop}}\times 
(\mathbb{R}^{n})^{\vee}/(\mathbb{Z}^{n})^{\vee};x\mapsto 
(x,q\circ \mu_{P}^{\mathrm{trop}}(x)).
\end{align}

This is a section of the 
fiber bundle 
$\hat{f}_P\colon X_{P}^{\mathrm{trop}}\times 
(\mathbb{R}^{n})^{\vee}/(\mathbb{Z}^{n})^{\vee}
\to X_{P}^{\mathrm{trop}}$. 
Readers may feel that
$\dot{s}_P$ look like a `Lagrangian section of
$\hat{f}_P$'.
In fact, $\dot{s}_P|_{\mathbb{R}^{n}}$
is truly a Lagrangian section of the 
Lagrangian torus fibration 
$\check{f}_{\mathbb{R}^{n}}\colon 
\check{X}(\mathbb{R}^{n})\to \mathbb{R}^{n}$.
(A certain dual of $\dot{s}_P$ appears in 
\cite[Theorem 3.20]{yamaguchimaster}.)
Let $0_P$ be the zero section of $\hat{f}_P$. 
Then, the following formula is obvious:
\begin{align}
\label{equation-moment-map-intersection}
\sharp (0_{P}\cap \dot{s}_P)
=\sharp (P \cap (\Z^{n})^{\vee}).
\end{align}
We can see \eqref{equation-moment-map-intersection}
is very similar with the intersection number of
two Lagrangian 
sections.

Now, 
let's reinterpret 
\eqref{equation-danilov-formula} -\!
\eqref{equation-ehrhart-reciprocity-2} and
\eqref{equation-moment-map-intersection}
using the local Morse data 
for a permissible $C^{\infty}$-divisor.
First, let's construct a
prepermissible $C^{\infty}$-divisor
on $X_{P}^{\mathrm{trop}}$ from a Delzant polytope $P$. 
Every compact tropical toric manifold is not 
boundaryless, so we need to use the sheaf
of weakly-smooth functions which we defined 
in \cref{definition-weakly-smooth}.
On the other hand, from \eqref{equation-log-polynomial} 
we can construct 
a prepermissible $C^{\infty}$-divisor
$s_P\in H^{0}(X_P;
\mathcal{A}^{\mathrm{weak}}_{X_P}/
\mathcal{O}^{\times}_{X_P})$ for 
a given tropical projective toric manifold 
$X_P^{\mathrm{trop}}$ as follows:
{\setlength{\leftmargini}{22pt}
\begin{enumerate}
\item For $m\in P\cap 
(\mathbb{Z}^{n})^{\vee}$, we set $p\deq 
(\mu_{P}^{\mathrm{trop}})^{-1}(m)$ and 
$f_{P,p}\deq f_P-\langle m,x \rangle=f_{P-m}$
where $P-m\deq\set{p-m\in (\mathbb{R}^{n} )^{\vee}\mid p\in P}$.
\item If $p$ is a vertex of $P$
and $\sigma_p$ is the convex cone corresponding 
to $p$, then 
$f_{P,p}$ can be considered as an element of 
$H^{0}(U_{\sigma_p},\mathcal{A}_{X_P}^{\mathrm{weak}})$
naturally since $f_{P,p}\in 
\underline{\mathbb{R}}_{\opn{log}}[S_{\sigma_p}]$. 
These data form a $C^{\infty}$-Cartier 
divisor $s_P$ on $X^{\mathrm{trop}}_P$.
\item Every local cone $\opn{LC}_x X_{P}^{\mathrm{trop}}$
in $X_P^{\mathrm{trop}}$ at $x$ is isomorphic to an 
$\mathbb{R}^{l}$, and thus $s_P$ obviously satisfies 
the prepermissibility condition.
\end{enumerate}
}
\begin{example} \label{eg: TP1Cartier}
Let $P\deq [0,n]$. Then,
$f_P\colon \mathbb{R}\to \mathbb{R}; 
x\mapsto \opn{log}(\sum_{i=0}^{n}e^{ix})$ and 
$X_P\simeq \mathbb{T}P^{1}$.
The space $\mathbb{T}P^{1}$ is covered by two open subsets:
$U_{-\infty}\deq \underline{\mathbb{R}}$ and 
$U_{+\infty}\deq \overline{\mathbb{R}}$.
Take two continuous functions on each $U_i$ 
($i=\pm \infty$) as follows:
\begin{align}
f_{-\infty,n}\colon \underline{\mathbb{R}}\to \underline{\mathbb{R}};
x \mapsto \log (\sum_{i=0}^{n}e^{ix}),\quad
f_{+\infty,n}\colon \overline{\mathbb{R}} \to 
\underline{\mathbb{R}};
x\mapsto \log(\sum_{i=0}^{n}e^{ix})-nx.
\end{align}

A local data 
$\{(U_i,f_{i,n})\}_{i=\pm \infty}$ defines
a $C^{\infty}$-divisor $s_n$ on $\mathbb{T}P^{1}$.
An element 
$D_n\deq (f_{-\infty,n}-f_{+\infty,n}=nx)\in 
\check{H}^{1}(\mb{T}P^1;
\mcal{O}_{\mb{T}P^1}^{\times})\simeq \Z$
as a \v{C}ech cocycle.
Then, we have
\begin{align}
\label{equation-MRR-tropical-line}
\sharp([0,n]\cap \Z)=n+1=
\opn{deg}([s_n])+\chi_{\opn{top}}(\mb{T}P^1).
\end{align}
\end{example}
Let $P$ be an $n$-dimensional Delzant polytope again.
Every face $F$ of $P$ corresponds a convex cone 
$\sigma_F$ of the normal fan of $P$.
$s_P$ induces a family 
$\{(O(\sigma_F),f_F|_{O(\sigma_F)})\}_{F \prec P}$ of Laurent polynomials 
$f_F$ on each tropical torus orbit $O(\sigma_F)$
of $\sigma_{F}$
naturally (up to monomials).
If the image $\mu_P^{\mathrm{trop}}(p)$
of a point $p\in X^{\mathrm{trop}}_P$  
in a relative interior of a face $F$ 
of $P$, then
$p$ is in 
$s_0\cap s_P$ if and only if 
$df_F(p)\in (\mathbb{Z}^{n})^{\vee}$.
Therefore,
the following equations hold:
\begin{align}
s_0\cap -s_P=s_0\cap s_P=(\mu_P^{\opn{trop}})^{-1}
(P\cap (\mathbb{Z}^{n})^{\vee}).
\end{align}

\begin{proposition}[{Ehrhart reciprocity for 
local Morse data}]
\label{proposition-ehrhart-revisit}
Let $s_P$ be a prepermissible $C^{\infty}$-divisor on 
the $n$-dimensional
projective tropical toric manifold $X_P$ for 
a (Delzant) lattice polytope $P$. Then, 
$s_P$ satisfies \cref{condition-good-divisor}. 
Besides,
\begin{align}
\chi(\opn{LMD}^{\bullet}(X^{\mathrm{trop}}_P;s_P))
&=\sharp(\opn{int}(P\cap \Z^{n})), \\
\chi(\opn{LMD}^{\bullet}(X^{\mathrm{trop}}_P;-s_P))
&=(-1)^{n}\sharp 
(\opn{int}(P\cap \Z^{n})).
\end{align}
\end{proposition}
\begin{proof}
Let $\{(U_p,f_{P,p})\}_{p\in s_0\cap s_P}$ be an intersection
data for $s_P$.
For any $f_{P,p}$, $\{f_{P,p}<f_{P,p}(p)\}$ is empty 
and $\{-f_{P,p}<-f_{P,p}(p)\}=U_p\setminus \{p\}$. 
If $\mu_P^{\mathrm{trop}}(p)=\partial P\cap \mathbb{Z}^{n}$,
then $\opn{LMD}^{\bullet}(U_p,-f_{P,p},p)$ is trivial since 
we can choose a sufficiently small 
$U_p$ such that $U_p\setminus \{p\}$
is contractible.
\end{proof}

\begin{remark}
There exist several versions of 
the homological mirror symmetry for the derived categories
of coherent sheaves on toric manifolds 
(e.g. \cite{MR2529936,MR2564372,MR2871160,MR4234675}). 
These references are related with tropical geometry.
Especially, $C^{\infty}$-divisors on tropical toric 
manifolds are deeply related to
Lagrangian sections defined in 
\cite[Definition 3.1]{MR2564372} 
or the wrapped Fukaya category 
$\opn{Fuk}(T^{*}T_{\mathbb{R}}^{\vee};\overline{\Lambda}_{\Sigma})$
in \cite[Theorem 2]{MR2871160}.
We also note the homological mirror symmetry between the derived category of
coherent sheaves on a non-singular complete toric variety $X_{\Sigma}$ and
the wrapped Fukaya category 
$\opn{Fuk}(T^{*}T_{\mathbb{R}}^{\vee};\overline{\Lambda}_{\Sigma})$
is finally solved by Kuwagaki's results \cite{MR4132582}.
\end{remark}

\begin{remark}
In \cite{MR4155409},
a new Ehrhart theory of tropical geometry appears, 
but this is different from our approach. 
\end{remark}

\bibliography{tropical-permissible-divisor}
\bibliographystyle{alpha}

\end{document}